\documentclass[twoside,12pt,leqno]{article}
\usepackage{amssymb}
\usepackage{mathrsfs}
\usepackage[centertags]{amsmath}
\usepackage{amsfonts}
\usepackage{amsthm}
\usepackage{color}

\usepackage{fancyhdr}

\newcommand\dela[1]{\Green{{\small{14.12.2009}}}}

\newcommand\Rng{\mathrm{Rng}}
\theoremstyle{plain}

\newtheorem{theorem}{Theorem}[section]
\newtheorem{proposition}[theorem]{Proposition}
\newtheorem{defi}[theorem]{Definition}
\newtheorem{lem}[theorem]{Lemma}

\newtheorem{remark}{Remark}[section]
\newtheorem{acknowledgements}{Acknowledgements}

\newtheorem{assu}{Assumption}[section]

\input colordvi

\newcommand\E{{\mathbb E}}
\newcommand\mE{{\mathscr{E}}}
\newcommand\cadlag{c\`{a}dl\`{a}g}
\newcommand\N{{\mathbb N}}
\newcommand\R{{\mathbb R}}

\allowdisplaybreaks[1]

\numberwithin{equation}{section}
\title{{\bf  Stochastic nonlinear beam equations driven by compensated Poisson random measures} }
\author{{\bf   Zdzis{\l}aw Brze\'{z}niak$^{a}$ and
 Jiahui Zhu$^{a}$
}\\
{\footnotesize $a.$ Department of Mathematics, University of York, Heslington, York, YO10 5DD, UK}\\
}
\date{}

\begin{document}
\maketitle

\begin{abstract}
We consider a type of stochastic nonlinear beam
equation driven by L\'{e}vy noise. By using a suitable Lyapunov
function and applying the  Khasminskii test we show the nonexplosion of
the mild solutions. In addition, under some additional assumptions we prove the exponential stability of the solutions.

\smallskip

\end{abstract}
% \tableofcontents
\section{Introduction and Motivation}

 The Euler-Bernoulli beam equation
         \begin{align*}
           EI\frac{d^4u}{dx^4}=w
         \end{align*}
as a simplification of linear beam theory was first introduced in
1750 to describe the relationship between the deflection and applied
load. The transversal deflection $u$ of a hinged extensible beam of length
$l$ under an axial force $H$ which satisfies the following form

\begin{align}\label{Woinowsky-Krieger}
     \frac{\partial^2u}{\partial t^2}+\frac{EI}{\rho}\frac{\partial ^4 u}{\partial
     x^4}=\left(\frac{H}{\rho}+\frac{EA}{2\rho l}\int_0^l\left(\frac{\partial u}{\partial
     x}\right)^2dx\right)\frac{\partial^2u}{\partial x^2}.
\end{align}
was studied by S. Woinowsky-Krieger \cite{[Woinowsky-Krieger]}.
See also Eisley \cite{[Eisley]} and Burgreen \cite{[Burgreen]} for
more details.
Chueshov \cite{[Chueshov]} considered a problem of the following form
\begin{align*}
   u_{tt}+\gamma u_t+A^2 u+m(\|A^{\frac{1}{2}}u\|^2)Au+Lu=p(t)
\end{align*}
which arises in the nonlinear theory of oscillations of a plate in a
supersonic gas flow moving along an $x_1$-axis described by
\begin{align*}
   \frac{\partial ^2 u}{\partial t^2}+\gamma \frac{\partial u}{\partial
   t}+\triangle^2u+\left(\alpha-\int_D|\triangledown
   u|^2dx\right)\triangle u+\rho\frac{\partial u}{\partial
   x_1}=p(x,t),\ \ x\in(x_1,x_2)\subset D,
\end{align*}
where $u(x,t)$ measures the plate deflection at the point $x$ and
the moment $t$, $\gamma>0$, $\rho\geq0$ and function $p(x,t)$
describles the transverse load on the plate.
 In \cite{[Patcheu]} Patcheu
considered a model of \eqref{Woinowsky-Krieger} with a nonlinear
friction force. The existence and uniqueness of global solutions of
a nonlinear version of the Euler-Bernoulli with white noise arising
from vibration of an aeroelastic panels
\begin{align}\label{sde-1}
    \frac{\partial^2 u}{\partial t^2}-\left(a+b\int_0^{l}\left(\frac{\partial u}{\partial x}\right)^2dx\right)\frac{\partial^2 u}{\partial x^2}&+\gamma\frac{\partial^4u}{\partial x^4}+f\left(t,x,\frac{\partial u}{\partial t},\frac{\partial u}{\partial x}\right)\nonumber\\
    &+\sigma\left(t,x,\frac{\partial u}{\partial t},\frac{\partial u}{\partial x}\right)\dot W(t)=0
\end{align}
 has been investigated by Chow and Menaldi in \cite{[Chow+Menaldi_1999]}.
The first named authour, Maslowski and Seidler \cite{Brz+Masl+S_2005} proved
the existence of global mild solutions of the following stochastic
beam equations including a white noise type and a nonlinear random
damping term in a Hilbert space $H$
\begin{align}\label{SDE-3}
    u_{tt}&+A^2u+g(u,u_t)+m(\|B^{\frac{1}{2}}u\|^2)Bu=\sigma(u,u_t)\dot{W},
\end{align}
where the operators $A$ and $B$ are self-adjoint and $\mathcal{D}(A)\subset\mathcal{D}(B)$.

 It is of interest
to know whether the theory can be extended to the problems with jump
noise which is in some sense more realistic.
In our paper, we consider a stochastic beam equation in some
Hilbert space $H$ with stochastic jump noise perturbations of the
form
\begin{align}\label{SDE-2}
u_{tt}=-A^2u-f(t,u,u_t)-m(\|B^{\frac{1}{2}}u\|^2)Bu+\int_{Z}g(t,u,u_t,z)\tilde{N}(t,dz),
\end{align}
where $m$ is a nonnegative function in $\mathcal{C}^1([0,\infty))$,
$A,B$ are self-adjoint operators and $\tilde{N}$ is a compensated
Poisson random measure. We will show that under some suitable locally
Lipschitz continuity and linear growth assumptions of the coefficients $f$ and $m$, the stochastic
beam equation \eqref{SDE-2} has a unique maximal local mild solution
$\mathfrak{u}$ which satisfies
    \begin{align}\label{intro_locally mild solution}
     \mathfrak{u}(t\wedge\tau_n)=e^{t\mathcal{A}}\mathfrak{u}_0+\int_0^{t\wedge\tau_n}e^{(t\wedge\tau_n-s)
     \mathcal{A}}F(s,\mathfrak{u}(s))\,ds
     +I_{\tau_n}(G(\mathfrak{u}))(t\wedge\tau_n)\ \ \mathbb{P}\text{-a.s.}\  t\geq0,
  \end{align}
  where $\{\tau_n\}_{n\in\N}$ is a sequence of stopping times and $I_{\tau_n}(G(\mathfrak{u}))$ is a process defined by
  \begin{align*}
       I_{\tau_n}(G(\mathfrak{u}))(t)=\int_0^{t}\int_Z1_{[0,\tau_n]}e^{(t-s)\mathcal{A}}G(s,\mathfrak{u}(s-),z)\tilde{N}(ds,dz),\ t\geq0.
  \end{align*}
 We also show the nonexplosion of the local maximal solution. The basic
method that we shall use in showing the nonexplosion is the Khasminskii test. For this aim, the essence is to be able to
construct an appropriate Lyapunov function.  As we all known the It\^{o} formula can not be applied to the mild solution directly, a standard method of solving this problem which was used in \cite{Brz+Gat_1999} is to approximate the equation \eqref{SDE-3} by a sequence of equations with the operator $A$ replaced by the Yosida approximations $A_n$ of $A$. However, since the factorization method used in showing the uniform-$L^p$-convergence of the Yosida approximating stochastic convolutions w.r.t. the Wiener noise may not be applicable to our case, in contrast to \cite{Brz+Gat_1999} we follow the approximating procedure introduced in \cite{[Tu-1]} and \cite{[Tubaro]}. We first derive some estimates when $u$ is in $
\mathcal{D}(\mathcal{A})$, where $\mathcal{D}(\mathcal{A})$ is the
domain of the generator $\mathcal{A}$. In fact, we can always
approximating $u$ by such functions in $\mathcal{D}(\mathcal{A})$
and pass the limit as  in \cite{[Tubaro]} to get the desired
estimate of Lyapunov function. Moreover, the asymptotic stability
and uniform boundedness of the solution can also be established
in the same manner by a suitable choice of another Lyapunov
function. We also show that under some natural conditions all the
results in this paper we've achieved for \eqref{SDE-2} can be
applied to a wide class of models including the following problem
\begin{align}\label{eq-1-intro}
    \frac{\partial^2 u}{\partial t^2}-m\left(\int_D|\triangledown u|^2dx\right)\triangle u +\gamma\triangle^2u&+G\left(t,x,u,\frac{\partial u}{\partial t},\triangledown u\right)\nonumber
    \\
    &=\int_Z\Pi(t,x,u,\frac{\partial u}{\partial t},\triangledown u,z)\tilde{N}(t,du)
\end{align}
with either the clamped boundary conditions
\begin{align}
      u=\frac{\partial u}{\partial n}=0\ \text{on }\partial D,
\end{align}
or the hinged boundary conditions
\begin{align}
      u=\triangle u=0\ \text{on }\partial D.
\end{align}
In the above $\frac{\partial}{\partial
n}$ denotes the outer normal derivative.

In light of the results achieved in this paper together with \cite{Brz+Gat_1999}, it is straight forward to extend our problem
to L\'{e}vy noise. Moreover, the Feller property of the solution to the problem \eqref{SDE-2} can also be obtained and this makes it possible to define an invariant probability measure for the process \eqref{intro_locally mild solution}.  The existence of invariant measure for \eqref{intro_locally mild solution}, in contrary to the finite dimensional case e.g. in \cite{[ABW]}, is still an open problem.

Stochastic PDEs driven by discontinuous noise is a very new subject. So far mainly problems with Lipschitz coefficients have been investigated, see the recent monograph \cite{Peszat+Zabczyk_2007}. A type of stochastic PDEs with monotone and coercive coefficients, which is weaker than the usual Lipschitz and linear growth assumptions, driven by some discontinuous perturbations were studied by Gy\"{o}ngy and Krylov in \cite{Gyongy+K_1980} for the finite-dimensional case and extended by Gy\"{o}ngy to infinite-dimensional spaces in \cite{Gyongy_1982}.
Stochastic reaction diffusion equations driven by Levy noise have been a subject of a recent paper \cite{Brzez+Haus_reaction} by one of the authours and Hausenblas, where also some comments on the existing literature can be found.
The approach of the current paper is different as it does not use any compactness methods but instead follow a more natural route of contracting maximal local solution and then proving that its lifespan  is equal to infinity. To our best  knowledge the present paper is the first one in which this approach is applied to SPDEs with non-Lipschitz coefficients.

\section{Main results}

Throughout the whole paper we assume that $H$ is a real separable
Hilbert space with inner product $\langle\cdot,\cdot\rangle$ and
corresponding norm $\Vert\cdot\|_H$.
By $\mathcal{B}(H)$ we denote the Borel $\sigma$-field on $H$, i.e. the $\sigma$-field generated by the family of all open subsets of $H$. Let $B:\mathcal{D}(B)\rightarrow H,$ $\mathcal{D}(B)\subset H$, 
 be a self-adjoint operator.
Suppose that $A:\mathcal{D}(A)\rightarrow H$, where $\mathcal{D}(A)\subset\mathcal{D}(B)$, is a self-adjoint (unbounded)
operator and
 $A\geq\mu I$ for some $\mu>0$. Moreover, we assume that $B\in\mathcal{L}(\mathcal{D}(A),H)$.
 Here $\mathcal{D}(A)$ is the domain of $A$ endowed with the graph norm $\|x\|_{\mathcal{D}(A)}:=\|Ax\|$. 
  Let $m$ be a nonnegative function of class $C^1$.
  Let $(\Omega,\mathcal{F},\mathbb{P})$ be a probability space with the filtration $\mathfrak{F}=(\mathcal{F}_t)_{t\geq0}$
  satisfying the usual hypotheses and $(Z,\mathcal{Z},\nu)$ be a measure space, where $\nu$ is a $\sigma$-finite measure. We denote by 
  $$\tilde{N}((0,t]\times B)=N((0,t]\times B)-t\nu(B),\ t\geq0,\ B\in\mathcal{Z},$$
  the compensated Poisson random measure on $[0,T]\times\Omega\times Z$ with the intensity measure $\nu(\cdot)$.
       Let $\mathcal{B}\mathcal{F}$ denote the $\sigma$-field of the progressively measurable sets on $[0,T]\times\Omega$, i.e.
      $$\mathcal{B}\mathcal{F}= \{ A\subset[0,T]\times\Omega:\forall\ t\in[0,T], A\cap([0,t]\times\Omega)\in\mathcal{B}([0,t])\otimes\mathcal{F}_t\}.$$

  \begin{defi}[\bf{Predictability}]\label{defi: predictablility}
  Let $\mathcal{P}$ denote the $\sigma$-field on $[0,\infty)\times \Omega$ generated by all real-valued left-continuous and $\mathfrak{F}$-adapted processes.

Let $\hat{\mathcal{P}}$ denote the $\sigma$-field on $\mathbb{R}_+\times\Omega\times Z$ generated all real-valued functions $g:\mathbb{R}_+\times\Omega\times Z\rightarrow \mathbb{R}$ satisfying the following properties
\begin{enumerate}
    \item[(1)] for every $t>0$, the mapping $\Omega\times Z\ni (\omega,z)\mapsto g(t,\omega,z)\in \mathbb{R}$ is $\mathcal{Z}\otimes\mathcal{F}_t/\mathcal{B}(\mathbb{R})$-measurable;
    \item[(2)] for every $(\omega,z)\in\Omega\times Z$, the path $\mathbb{R}_+\ni t\mapsto g(t,\omega,z)\in \mathbb{R}$ is left-continuous.
\end{enumerate}
Let $(E,\mathcal{B}(H))$ be a measurable space. We say that an $E$-valued process $g=(g(t))_{t\geq 0}$ is \textbf{predictable} if the mapping $[0,\infty)\times\Omega\ni(t,\omega)\mapsto g(t,\omega)\in E$ is $\mathcal{P}/\mathcal{B}(E)$-measurable.\\
We say that an $E$-valued function $g:\mathbb{R}_+\times\Omega\times Z\rightarrow E$ is \textbf{$\mathfrak{F}$-predictable} if it is $\hat{\mathcal{P}}/\mathcal{B}(E)$-measurable.
\end{defi}

 In this paper, our main aim is to consider the following stochastic evolution equation
 \begin{align}\label{SDE}
 \begin{split}
                &u_{tt}=-A^2u-f(t,u,u_t)-m(\|B^{\frac{1}{2}}u\|^2)Bu+\int_{Z}g(t,u(t-),u_t(t-),z)\tilde{N}(t,dz),\\
        &u(0)=u_0,\ u_t(0)=u_1.
        \end{split}
 \end{align}
Here $f:\mathbb{R}_+\times \mathcal{D}(A)\times H\ni
(t,\xi,\eta)\mapsto f(t,\xi,\eta)\in H,$ is a
$\mathcal{B}(\mathbb{R}_+)\otimes\mathcal{B}(\mathcal{D}(A))\otimes\mathcal{B}(H)/\mathcal{B}(H)$-measurable
function and $g:\mathbb{R}_+\times \mathcal{D}(A)\times H\times
Z\ni(t,\xi,\eta,z)\mapsto g(t,\xi,\eta,z)\in H,$ is a
$\mathcal{B}(\mathbb{R}_+)\otimes\mathcal{B}(\mathcal{D}(A))\otimes\mathcal{B}(H)\otimes
\mathcal{Z}/\mathcal{B}(H)$-measurable function. One can transform
Equation \eqref{SDE1} into the following first order system
\begin{align}\label{eq1}
\begin{split}
     &du=u_tdt\\
     &du_t=-A^2udt-f(u,u_t)dt-m(\|B^{\frac{1}{2}}u\|^2)Budt+\int_{Z}g(t,u(t-),u_t(t-),z)\tilde{N}(dt,dz).
     \end{split}
\end{align}
Or equivalently, we can rewrite it in the form
\begin{align*}
  \left( \begin{array}{c} du \\ du_t \end{array} \right)
                 =\left(\begin{array}{cc} 0 & I\\ -A^2 & 0 \end{array}\right)\left( \begin{array}{c} u \\ u_t \end{array} \right)dt&+\left( \begin{array}{c} 0 \\ -f(t,u,u_t)-m(\|B^{\frac{1}{2}}u\|)Bu \end{array} \right)dt\\
                 &+\left( \begin{array}{c} 0 \\ \int_Zg(t,u(t-),u_t(t-),z)\tilde{N}(dt,dz) \end{array} \right).
\end{align*}
Now we introduce a new space $\mathcal{H}:=\mathcal{D}(A)\times H$
with the product norm
\begin{align*}
    \left\|\left( \begin{array}{c} x \\ y \end{array} \right)\right\|
    ^2_{\mathcal{H}}:=\|Ax\|^2_H+\|y\|^2_H.
\end{align*}
It is easy to see that $\mathcal{H}$ is a Hilbert space with norm
$\|\cdot\|_{\mathcal{H}}$. We also define functions
\begin{align}
     &F:\mathbb{R}_+\times \mathcal{D}(A)\times H\ni(t,\xi,\eta)\mapsto\left( \begin{array}{c} 0 \\ -f(t,\xi,\eta)-m(\|B^{\frac{1}{2}}\xi\|^2)B\xi \end{array} \right)\in\mathcal{H}\label{function F}\\
     &G:\mathbb{R}_+\times \mathcal{D}(A)\times H\times Z\ni(t,\xi,\eta,z)\mapsto\left( \begin{array}{c} 0 \\ g(t,\xi,\eta,z) \end{array}\right)\in\mathcal{H} .
\end{align}
Put
\begin{align*}
 \mathcal{A}=\left( \begin{array}{cc} 0 & I \\ -A^2 & 0 \end{array}
 \right),\ \mathcal{D}(\mathcal{A})=\mathcal{D}(A^2)\times H.
\end{align*}
Set $\mathfrak{u}=(u,u_t)^{\top}$ and
$\mathfrak{u}_0=(u_0,u_1)^{\top}$.
Then Equation \eqref{SDE} allows the following form
\begin{align}\label{SDE1}
\begin{split}
   &d\mathfrak{u}=\mathcal{A}\mathfrak{u}dt+F(t,\mathfrak{u}(t))dt+\int_ZG(t,\mathfrak{u}(t-),z)\tilde{N}(dt,dz),\ \ t\geq0\\
   &\mathfrak{u}(0)=\mathfrak{u}_0.
   \end{split}
\end{align}

\begin{remark}\label{rem-9}
(1) The operator $\mathcal{A}$ generates a $C_0$-unitary group, denoted by $e^{t\mathcal{A}}$, $-\infty<t<\infty$,  on $\mathcal{H}$, see also Chapter V in \cite{[Lax]}.

(2) The functions $f$ and $g$ appearing in the equation \eqref{SDE} can also be assume to be random, namely,
$$f:\mathbb{R}_+\times\Omega\times \mathcal{D}(A)\times H\ni
(t,\xi,\eta)\mapsto f(t,\omega,\xi,\eta)\in H,$$ is a
$\mathcal{B}\mathcal{F}\otimes\mathcal{B}(\mathcal{D}(A))\otimes\mathcal{B}(H)/\mathcal{B}(H)$-measurable
function and $$g:\mathbb{R}_+\times\Omega\times \mathcal{D}(A)\times H\times
Z\ni(t,\omega,\xi,\eta,z)\mapsto g(t,\omega,\xi,\eta,z)\in H,$$ is a
$\mathcal{P}\otimes\mathcal{B}(\mathcal{D}(A))\otimes\mathcal{B}(H)\otimes
\mathcal{Z}/\mathcal{B}(H)$-measurable function. But due to \cite{[Metivier_1982]}, the functions $f$ and $g$ need to satisfy the following additional property:

if $X$ and $Y$ are two $\mathcal{H}$-valued c\`{a}dl\`{a}g processes and $\tau$ is a stopping time such that
\begin{align*}
     X1_{[0,\tau)}=Y1_{[0,\tau)},
\end{align*}
then we have
\begin{align*}
      1_{[0,\tau]}f(\cdot,\cdot,X)=1_{[0,\tau]}f(\cdot,\cdot,Y) \ \text{and }1_{[0,\tau]}g(\cdot,\cdot,X,\cdot)=1_{[0,\tau]}g(\cdot,\cdot,Y,\cdot).
      \end{align*}

\end{remark}
Let $\mathcal{M}^2_{loc}(\mathcal{B}\mathcal{F})$ be the space of all $\mathcal{H}$-valued
progressively measurable processes $\phi:\mathbb{R}_+\times\Omega\rightarrow \mathcal{H}$ such that for all $T\geq 0$,
 \begin{align*}
     \E\int_0^{T}\|\phi(t)\|^2\,dt<\infty.
\end{align*}
 Let $\mathcal{M}^2_{loc}(\hat{\mathcal{P}})$ be the space of all $\mathcal{H}$-valued $\mathfrak{F}$-predictable processes $\varphi:\mathbb{R}_+\times\Omega\times Z\rightarrow \mathcal{H}$ such that for all $T\geq0$,
 \begin{align*}
     \E\int_0^{T}\int_Z\|\varphi(t,z)\|^2\nu(dz)dt<\infty.
\end{align*}
\begin{defi}\label{defi: strong solution}A \textbf{strong solution} to Equation \eqref{SDE1} is a $\mathcal{D}(\mathcal{A})$-valued $\mathfrak{F}$-adapted stochastic process $(X(t))_{t\geq0}$ with c\`{a}dl\`{a}g paths such that
\begin{enumerate}
   \item[(1)]$X(0)=\mathfrak{u}_0$ a.s.,
   \item[(2)]the processes $\phi$, $\varphi$ defined by
   \begin{align*}
   \phi(t,\omega)&=F(t,X(t,\omega))\ \ (t,\omega)\in\mathbb{R}_+\times\Omega;\\
   \varphi(t,\omega,z)&=G(t,X(t-,\omega),z)\ \ (t,\omega,z)\in\mathbb{R}_+\times\Omega\times Z
   \end{align*}
   belong to the spaces $\mathcal{M}^2_{loc}(\mathcal{B}\mathcal{F})$ and $\mathcal{M}^2_{loc}(\hat{\mathcal{P}})$ respectively.

   \item[(3)]for any $t\geq 0$, the equality
   \begin{align}\label{strong solution}
     X(t)=\mathfrak{u}_0+\int_0^t\mathcal{A}X(s)\,ds+\int_0^tF(s,X(s))\,ds+\int_0^{t}\int_Z G(s,X(s-),z)\tilde{N}(ds,dz)
   \end{align}
   holds $\mathbb{P}$-a.s.
\end{enumerate}
\end{defi}
\begin{remark}Note that if a process $X(t)$, $t\geq0$ is adapted and has c\`{a}dl\`{a}g paths, then the left-limit process $X(t-)$, $t\geq0$ is left continuous and adapted, hence the process $X(t-)$, $t\geq0$ is predictable. In such a case, the definition \ref{defi: strong solution} is reasonable.
\end{remark}

\begin{defi}\label{defi: mild solution} A \textbf{mild solution} to Equation \eqref{SDE1} is
an $\mathcal{H}$-valued $\mathfrak{F}$-adapted stochastic process $(X(t))_{t\geq 0}$ with c\`{a}dl\`{a}g paths defined on
$(\Omega,\mathcal{F},\mathfrak{F},\mathbb{P})$ such
that the conditions (1) and (2) in the definition of \ref{defi: strong solution} are satisfied and

   for any $t\geq0$, the equality
   \begin{align}\label{mild solution}
     X(t)=e^{t\mathcal{A}}\mathfrak{u}_0+\int_0^te^{(t-s)\mathcal{A}}F(s,X(s))\,ds+\int_0^{t}\int_Ze^{(t-s)\mathcal{A}}G(s,X(s-),z)\tilde{N}(ds,dz)
   \end{align}
   holds $\mathbb{P}$-a.s.
\end{defi}

We say that a solution $(X(t))_{t\geq0}$ to the Equation \eqref{SDE1} is \textbf{pathwise unique} (or \textbf{up to distinguishable}) if for any other solution $(Y(t))_{t\geq0}$, we have
\begin{align*}
    \mathbb{P}(X(t)=Y(t), \text{ for all }t\geq0)=1.
\end{align*}
\begin{defi} We say that $X$ is a mild solution on a closed stochastic interval $[0,\sigma]$ if the integral on the right of \eqref{mild solution} is defined on $[0,\sigma]$ and it equals to $X$ on $[0,\sigma]$, $\mathbb{P}$-a.s., namely
\begin{align}\label{eq-340}
X(t)=e^{t\mathcal{A}}\mathfrak{u}_0+\int_0^te^{(t-s)\mathcal{A}}F(s,X(s))\,ds+\int_0^{t}\int_Ze^{(t-s)\mathcal{A}}G(s,X(s-),z)\tilde{N}(ds,dz)
\end{align}
holds on $[0,\tau]$, $\mathbb{P}\text{-a.s.}$.
\end{defi}

\begin{remark}\label{rem-101}
Alternatively, we may rewrite \eqref{eq-340} in the following equivalent form
    \begin{align}
     X(t\wedge\tau)=e^{t\mathcal{A}}\mathfrak{u}_0+\int_0^{t\wedge\tau}e^{(t\wedge\tau-s)\mathcal{A}}F(s,X(s))\,ds
     +I_{\tau}(G(X))(t\wedge\tau)\ \ t\geq0,\ \mathbb{P}\text{-a.s., }
  \end{align}
  where $I_{\tau}(G(X))$ is a process defined by
  \begin{align*}
       I_{\tau}(G(X))(t)=\int_0^{t}\int_Z1_{[0,\tau]}(s)e^{(t-s)\mathcal{A}}G(s,X(s-),z)\tilde{N}(ds,dz),\ t\geq0.
  \end{align*}
\end{remark}

\begin{remark}\label{rem-102}
According to  Colollary 13.7  in the monograph \cite{[Metivier_1982]} every predictable
and right-continuous martingale is continuous, so   if we impose both
properties on a process, it turns out that we are assuming nothing but the
continuity of the process. In our definition, the reason why we need the
predictability of the process $X$ is to get the predictability of the integrand
$e^{(t-s)A}G(t,X(s),z)$. But since we assume that the process is \cadlag, we can get
around this difficulty by taking the left-limit process.
\end{remark}

We first deal with a simple case in which the function $F$ is given
by
\begin{align}\label{condition on F}
    F:[0,\infty)\times \mathcal{D}(A)\times H\ni(t,\xi,\eta)\mapsto\left( \begin{array}{c} 0 \\ -f(t,\xi,\eta) \end{array} \right)\in\mathcal{H}.
\end{align}
In order to show the existence and uniqueness of a mild solution to
problem \eqref{SDE1}, we impose certain growth conditions and global
Lipschitz conditions on the functions $f$ and $g$.
\begin{assu}\label{assu: growth condition}
There exist constants $K_f$ and $K_g$ such that for all $t\geq 0$
and all $x=(x_1,x_2)^{\top}\in\mathcal{H}$,
\begin{align}
      \|f(t,x_1,x_2)\|^2_H&\leq K_f(1+\|x\|^2_{\mathcal{H}})\label{growth condition of f}\\
      \int_Z\|g(t,x_1,x_2,z)\|^2_H\nu(dz)&\leq K_g(1+\|x\|^2_{\mathcal{H}}).\label{growth condition of g}
\end{align}
\end{assu}

\begin{assu}\label{assu: globally Lipschtiz condition}
There exist constant $L_f$ such that for all $t\geq0$ and
all $x=(x_1,x_2)^{\top}\in\mathcal{H}$, $y=(y_1,y_2)^{\top}\in\mathcal{H}$,
\begin{align}
      \|f(t,x_1,x_2)-f(t,y_1,y_2)\|_{H}\leq L_f\|x-y\|_{\mathcal{H}},\label{globally Lipschitz condition of f}
\end{align}
\end{assu}  
\begin{assu}\label{assu-g}
	        There exist constant $L_g$ such that for all $t\geq0$ and
			all $x=(x_1,x_2)^{\top}\in\mathcal{H}$, $y=(y_1,y_2)^{\top}\in\mathcal{H}$,
			\begin{align}  		      \int_Z\|g(t,x_1,x_2,z)-g(t,y_1,y_2,z)\|_{H}^2\nu(dz)\leq L_g\|x-y\|_{\mathcal{H}}^2.\label{globally Lipschitz condition of g}
			\end{align}
\end{assu}

\begin{theorem}\label{theo: exsitence and uniqueness of mild solution}

   Suppose that functions $f,g$ satisfy Assumptions \ref{assu: growth condition}, \ref{assu: globally Lipschtiz condition} and \ref{assu-g}. Then there exists a unique (up to distinguishable) mild solution of
   Equation \eqref{SDE1}. 
 In particular, if $$\mathfrak{u}_0\in\mathcal{D}(\mathcal{A}), \ F(\cdot,\mathfrak{u}(\cdot))\in\mathcal{M}^2_{loc}(\mathcal{B}\mathcal{F};\mathcal{D}(\mathcal{A}))\ 
 \text{and}\   G(\cdot,\mathfrak{u}(\cdot))\in\mathcal{M}^2_{loc}(\hat{\mathcal{P}};\mathcal{D}(\mathcal{A})),$$
 then the mild solution coincides with
probability $1$ with a strong solution at all the points over
$\mathbb{R}_+$. More precisely, the mild solution satisfying \eqref{mild solution} is $\mathbb{P}$-equivalent to the strong solution satisfying \eqref{strong solution}.

\end{theorem}

%\begin{remark}The following equation is a special form of Equation \eqref{SDE1},
%\begin{align}\label{SDE4}
%\begin{split}
%   %&d\mathfrak{u}=\mathcal{A}\mathfrak{u}dt+F(t)dt+\int_ZG(t,z)\tilde{N}(dt,dz),\\
%   &\mathfrak{u}(0)=\mathfrak{u}_0.
%   \end{split}
%\end{align}
%Hence in view of Theorem \ref{theo: exsitence and uniqueness of mild solution} it has a unique mild solution. Therefore, $\mathfrak{u}$ satisfies
%\begin{align*}
      %\mathfrak{u}(t)=e^{t\mathcal{A}}\mathfrak{u}_0+\int_0^te^{(t-s)\mathcal{A}}F(s)\,ds+\int_0^t\int_Ze^{(t-s)}G(s,z)\tilde{N}(ds,dz) \ \mathbb{P}\text{-a.s.} \ \ t\geq0.
%\end{align*}
%In such a case, if $\mathfrak{u}_0\in\mathcal{D}(\mathcal{A})$,
%$F(s)\in\mathcal{D}(\mathcal{A})$ and
%$G(s,z)\in\mathcal{D}(\mathcal{A})$, for every $s\geq0$ and $z\in
%Z$, Equation \eqref{SDE4} has a unique strong solution which is %$\mathbb{P}$-equivalent to a mild solution. Hence by Definition \ref{defi: strong solution}, the process  $\mathfrak{u}$ satisfies
%satisfies
%\begin{align*}
   %\mathfrak{u}(t)=\mathfrak{u}_0+\int_0^t\mathcal{A}\mathfrak{u}(s)\,ds+\int_0^tF(s)\,ds
 %  +\int_0^{t}\int_ZG(s,z)\tilde{N}(ds,dz)\ \
 %  \mathbb{P}\text{-a.s.} \ \ t\geq0.
%\end{align*}
%\end{remark}  

Another generalization is to the situation where the Lipschitz condition \eqref{globally Lipschitz condition of f} is released to be the locally Lipschitz condition below.
\begin{assu}\label{assu: locally Lipschitz condition}
Assume that for every $R>0$, there exists $L_R>0$ such that for all
$t\geq0$ and for every $x=(x_1,x_2)^{\top}$,
$y=(y_1,y_2)^{\top}\in\mathcal{H}$ satisfying
$\|x\|_{\mathcal{H}},\|y\|_{\mathcal{H}}\leq R$,
\begin{align}\label{locally Lipschitz condition}
      \|f(t,x_1,x_2)-f(t,y_1,y_2)\|_{H}\leq L_R\|x-y\|_{\mathcal{H}}.
\end{align}

\end{assu}
Now we shall examine stochastic equation \eqref{SDE1} of a more
general type than the equation with $F$ defined by \eqref{condition
on F} in the preceding Theorem. Note that the function
$\mathcal{H}\ni x=(x_1,x_2)\mapsto
m(\|B^{\frac{1}{2}}x_1\|^2)Bx_1\in H$, is locally Lipschitz
continuous. Hence if we suppose that $f$ satisfies Assumption
\eqref{locally Lipschitz condition}, then the function $F$ given by
\eqref{function F} satisfies the locally Lipschitz condition as well.

For future reference we specifically state the following important observations.
\begin{remark}\label{rem-main}

    (1) Because of the continuity of the function $F(t,x)$ with respect to the  $x$ variable, since   $F$ is  integrated  with respect to $t$ variable, the equation \eqref{SDE1} can be rewritten in the following equivalent form
    \begin{align*}
        d\mathfrak{u}=\mathcal{A}\mathfrak{u}dt+F(t,\mathfrak{u}(t-))dt+\int_ZG(t,\mathfrak{u}(t-),z)\tilde{N}(dt,dz),\ \ t\geq0.
    \end{align*}

    (2) Suppose that $X$ and $Y$ are two c\`{a}dl\`{a}g processes and $\tau$ is a stopping time. If $X$ and $Y$ coincide on the open interval $[0,\tau)$, i.e.
       \begin{align*}
              X(s,\omega)1_{[0,\tau)}(s)=Y(s,\omega)1_{[0,\tau)}(s),\ (s,\omega)\in\R_+\times\Omega,
       \end{align*}
       then we have
       \begin{align*}
              G(s,X(s-),z)1_{[0,\tau]}=G(s,Y(s-),z)1_{[0,\tau]}.
       \end{align*}
       This is because, the function $G(s,X(s-),z)$ depends only on the values of $X$ on $[0,\tau)$. However, if $G(t,\omega,x,z)$ itself is a stochastic process rather than a deterministic function, the above fact may no longer hold. In such a case, we require the condition introduced in Remark \ref{rem-9}.

\end{remark}

%-----------------------------------------------

\begin{defi}
   A stopping time $\tau$ is called accessible if there exists an
   increasing sequence $\{\tau_n\}_{n\in\mathbb{N}}$ of stopping
   times such that $\tau_n<\tau$ and
   $\lim_{n\rightarrow\infty}\tau_n=\tau$ a.s. We call such
   sequence $\{\tau_n\}_{n\in\mathbb{N}}$ the approximating sequence
   for $\tau$.
 A local mild solution to \eqref{SDE1} is an $\mathcal{H}$-valued,
   adapted, c\`{a}dl\`{a}g local process $X=(X(t))_{0\leq
   t<\tau}$, where $\tau$ is an accessible stopping time with an
   approximating sequence $\{\tau_n\}_{n\in\mathbb{N}}$, such that
   for any $n\in\mathbb{N}$ and $t>0$, the stopped process
   $X_t^{\tau_n}:=X(t\wedge\tau_n)$, $t\geq 0$ satisfies,
    \begin{align}\label{locally mild solution}
     X(t\wedge\tau_n)=e^{t\mathcal{A}}\mathfrak{u}_0+\int_0^{t\wedge\tau_n}e^{(t\wedge\tau_n-s)\mathcal{A}}F(s,X(s))\,ds
     +I_{\tau_n}(G(X))(t\wedge\tau_n)\ \mathbb{P}\text{-a.s.}\ t\geq0.
  \end{align}
  where $I_{\tau_n}(G(X))$ is a process defined by
  \begin{align}\label{eqn-I_tau_n}
       I_{\tau_n}(G(X))(t)=\int_0^{t}\int_Z1_{[0,\tau_n]}(s)e^{(t-s)\mathcal{A}}G(s,X(s-),z)\tilde{N}(ds,dz),\ t\geq0.
  \end{align}
  Here we call $\tau$ a life span of the local mild solution $X$.

A local mild solution $X=(X(t))_{0\leq t\leq \tau}$  to equation
\eqref{SDE1} is pathwise unique if for any other local mild solution
$\tilde{X}=\{\tilde{X}_{0\leq t<\tilde{\tau}}\}$ to equation
\eqref{SDE1},
\begin{align*}
      X(t,\omega)=\tilde{X}(t,\omega),\ \ (t,\omega)\in [0,\tau\wedge\tilde{\tau})\times\Omega.
\end{align*}

 A local mild solution $X=(X(t))_{0\leq t<\tau}$ is called a maximal mild solution if for any other local mild
   solution $\tilde{X}=(\tilde{X}(t))_{0\leq t<\tilde{\tau}}$
   satisfying $\tilde{\tau}\geq \tau$ a.s. and $\tilde{X}|_{[0,\tau)\times\Omega}\sim
   X$, then $\tilde{X}=X$.   Furthermore, if $\mathbb{P}(\tau<\infty)>0$, the stopping time $\tau$ is called an \textbf{explosion time} and if $\mathbb{P}(\tau=+\infty)=1$, the local mild solution $X$ have no explosion and it is called a global mild
   solution to Equation \eqref{SDE1}.
\end{defi}
\begin{remark}\label{SDE_pp_lemma_1}

(1) There is an alternative way to define a local mild solution. We say that an $\mathcal{H}$-valued c\`{a}dl\`{a}g process $X$ defined on an open interval $[0,\tau)$ is a local mild solution if there exists an increasing sequence $\{\tau_n\}$ of stopping times such that $\tau_n\nearrow\tau$, or in other words $[0,\tau)=\cup_n[0,\tau_n]$, and $X$ is a mild solution to problem \eqref{SDE1} on every closed interval $[0,\tau_n]$, $n\in\mathbb{N}$, see Remark \ref{rem-101}.

(2) If the Equation \eqref{SDE1} has the property of uniqueness for local
solutions, then the uniqueness of local maximal solution holds as
well.

\end{remark}

\begin{theorem}\label{theo: existence and uniqueness of locally mild solution} Suppose that Assumptions \ref{assu: growth condition}, \ref{assu-g} and \ref{assu: locally Lipschitz condition} are satisfied. Then there exists a unique maximal local mild solution to Equation \eqref{SDE1}.

\end{theorem}

 Now we shall apply Khas'minski's test to show that $\tau_{\infty}=+\infty$ a.s. That is $\mathfrak{u}$ is a unique global mild solution.

\begin{theorem}\label{theo: lifespan} Suppose that Assumptions \ref{assu: growth condition}, \ref{assu-g} and \ref{assu: locally Lipschitz condition} are satisfied and $\mathfrak{u}_0$ is $\mathcal{F}_0$-measurable. Let $\mathfrak{u}$ be the unique maximal local mild solution to the Equation \eqref{SDE1} with life span $\tau_{\infty}$. Then $\tau_{\infty}=+\infty$ $\mathbb{P}$-a.s.
\end{theorem}

Now we shall consider the stability of the solution to Equation
\eqref{SDE1}. To simplify our problem, we will impose the following extra
assumptions.

\begin{assu}\label{assu: stability}
    \begin{enumerate}
         \item[1).] Suppose that function $f$ is given by $f(x)=\beta x_1$ for some $\beta\geq0$, where $x=(x_1,x_2)^{\top}\in\mathcal{H}$;
         \item[2).] Assumptions \eqref{assu: growth condition} and \eqref{assu: locally Lipschitz condition} hold;
         \item[3).] There exist nonnegative constants $R_g$ and $K$ such that
                     \begin{align*}
                             \int_Z\|g(x,z)\|^2_H\nu(dz)\leq R_g^2\|x\|_{\mathcal{H}}^2+K.
                     \end{align*}
         \item[4).] There exists $\alpha>0$ such that for all nonnegative real number $y$
                         \begin{align*}
                            ym(y)\geq\alpha M(y).
                         \end{align*}

    \end{enumerate}
\end{assu}

Before starting the main theorem for stability, we establish an auxiliary lemma, the proof of which can be found in
\cite{Brz+Masl+S_2005}.

\begin{lem}\label{lem: property of operator P} Define an operator $P:\mathcal{H}\rightarrow\mathcal{H}$ by
      \begin{align*}
        P:=\left(\begin{array}{cc} \beta^2 A^{-2}+2I & \beta A^{-2}\\ \beta I & 2I \end{array}\right).
      \end{align*}
      Then $P$ is self-adjoint isomorphism of $\mathcal{H}$ and satisfies the following
          \begin{align*}
          &(1)\ \ \ \|P\|_{\mathcal{L}(\mathcal{H})}^{-1}\langle Px,x\rangle_{\mathcal{H}}\leq \|x\|^2_{\mathcal{H}}\leq \langle Px,x\rangle_{\mathcal{H}},\ \ \ \ x\in\mathcal{H};\\
          &(2)\ \ \ \langle \left(\begin{array}{c} 0 \\ -\beta x_2 \end{array}\right),Px\rangle_{\mathcal{H}}=-\beta^2\langle x_1,x_2\rangle=-2\beta\|x_2\|^2\ \ \ \ x=(x_1,x_2)^{\top}\in\mathcal{H};\\
          &(3)\ \ \ \langle \mathcal{A}x,Px\rangle_{\mathcal{H}}=-\beta\|Ax_1\|^2_H+\beta^2\langle x_1,x_2\rangle+\beta\|x_2\|^2.
          \end{align*}

\end{lem}

We define for $x=(x_1,x_2)^{\top}\in\mathcal{H}$
\begin{align*}
       \mathscr{E}(x)&=\E\big{[}\|x\|^2_{\mathcal{H}}+M(\|B^{\frac{1}{2}}x_1\|_H^2)\big{]}.
\end{align*}
%-----------------theorem----------
\begin{theorem}\label{theo: stability}Suppose that Assumption \eqref{assu: stability} is satisfied and $\mathscr{E}(\mathfrak{u}_0)<\infty$. Let $\mathfrak{u}$ be the unique mild global solution to Equation \eqref{SDE1}. Let $K$ be the constant given in Part $(3)$ of Assumption \eqref{assu: stability}. If $K=0$, then the solution is exponentially mean-square stable, that is there exist constants $C<\infty$, $\lambda>0$ such that for all $t\geq 0$,
\begin{align*}
      \E\|\mathfrak{u}(t)\|^2_{\mathcal{H}}\leq Ce^{-\lambda t}\mathscr{E}(\mathfrak{u}_0).
\end{align*}
If $K>0$, then $$\sup_{t\geq
0}\E\|\mathfrak{u}(t)\|^2_{\mathcal{H}}<\infty.$$

\end{theorem}

\section{Stochastic nonlinear beam
equations}\label{section_application}

In this section we will examine that all the results achieved in the
preceding section can be applied to the following problem
\begin{align}\label{eq-1-section-beam}
    \frac{\partial^2 u}{\partial t^2}-m\left(\int_D|\triangledown u|^2dx\right)\triangle u &+\gamma\triangle^2u+\Upsilon\left(t,x,u,\frac{\partial u}{\partial t},\triangledown u\right)\nonumber
   \\
   &=\int_Z\Pi(t,x,u,\frac{\partial u}{\partial t},\triangledown
    u,z)\tilde{N}(t,dz)
\end{align}
with the hinged boundary condition
\begin{align}\label{eq-cod-1}
u=\triangle u=0\ \text{on}\ \partial D.
\end{align}
Here $\Upsilon,\Pi:[0,T]\times D\times \mathbb{R}\times
\mathbb{R}^n\times\mathbb{R}\rightarrow\mathbb{R}$ are Borel
functions, $m\in\mathcal{C}^1(\mathbb{R}_+)$ is a nonnegative function, $\gamma>0$ and $D\subset\R^n$ is a bounded domain with a $\mathcal{C}^{\infty}$- boundary $\partial D$.

We shall also make the following standing assumptions on the functions $\Upsilon$ and $\Pi$ under considerations.
\begin{enumerate}
\item
For every $n\in\N$, there exist constants $L_N$ and $L$ such that for all $t\in [0,T]$,
 $x\in D$, $c_1,c_2\in\mathbb{R}$ and for all
$a_1,a_2\in\mathbb{R}$, $b_1,b_2\in\mathbb{R}^n$ satisfying
$|a_1|,|a_2|\leq N$ and $|b_1|,|b_2|\leq N$,
\begin{align}\label{assu_sec_beam-Lip. con._1}
    |\Upsilon(t,x,a_1,b_1,c_1)-\Upsilon(t,x,a_2,b_2,c_2)|\leq
    L_N|a_1-a_2|+L_N|b_1-b_2|+L|c_1-c_2|.
\end{align}
\item There exist constant $L_{\Upsilon}$ such that for all $t\in [0,T]$,
$x\in D$, $a\in\mathbb{R}$, $b\in\mathbb{R}^n$ and $c\in\mathbb{R}$,
\begin{align}\label{assu_sec_beam-growth-cond.1}
        |\Upsilon(t,x,a,b,c)|^2\leq
    L_{\Upsilon}(1+|c|^2).
\end{align}
\item There exist constant $L'$ such that for all $t\in [0,T]$,
 $x\in D$, $c_1,c_2\in\mathbb{R}$
 $a_1,a_2\in\mathbb{R}$ and $b_1,b_2\in\mathbb{R}^n$,
\begin{align}\label{assu_sec_beam-lip.con.2}
    &\int_Z|\Pi(t,x,a_1,b_1,c_1,z)-\Pi(t,x,a_2,b_2,c_2,z)|^2\nu(dz)\nonumber\\
    &\hspace{4cm}\leq
    L'|a_1-a_2|^2+L'|b_1-b_2|^2+L'|c_1-c_2|^2.
\end{align}
\item There exist constant $L_{\Pi}$ such that for all $t\in [0,T]$,
$x\in D$, $a\in\mathbb{R}$, $b\in\mathbb{R}^n$ and $c\in\mathbb{R}$,
\begin{align}\label{assu_sec_beam-growth-cond.2}
        \int_Z|\Pi(t,x,a,b,c,z)|^2\nu(dz)\leq
    L_{\Pi}(1+|c|^2).
\end{align}
\end{enumerate}

Let $H=L^2(D)$. Let $A$ and $B$ be both the Laplacian with Dirichlet
boundary conditions. That is
\begin{align*}
       &A \psi=-\Delta\psi,\ \ \psi\in \mathcal{D}(A),\\
       &\mathcal{D}(A)=H^2(D)\cap H^1_0(D).
\end{align*}
Then $A\geq \mu I$, for some $\mu>0$. To see this, since
$\mathcal{D}(A)\subset H^1_0(\Omega)$, on the basis of Poincar\'{e}
inequality, we have
\begin{align*}
      \langle A\psi, \psi\rangle_{L^2(D)}=-\int_D \Delta\psi\cdot\psi
      dx=\int_D|\nabla \psi|^2dx\geq C |\psi|_{L^2(D)},\ \ \text{for
      }\psi\in\mathcal{D}(A).
\end{align*}
Note that our results are thus valid also for unbounded domains satisfying Poincar\'{e} inequality.
 Let us set
\begin{align}\label{sec_beam_function_f}
   f: [0,T]\times\mathcal{D}(A)\times L^2(D)\ni (t,\psi,\phi)\mapsto
   \Upsilon(t,\cdot,\psi(\cdot),\nabla\psi(\cdot),\phi(\cdot))\in
   L^2(D)
\end{align}
and
\begin{align}\label{sec_beam_function_g}
    g:[0,T]\times\mathcal{D}(A)\times L^2(D)\ni
    (t,\psi,\phi)\mapsto
    \Pi(t,\cdot,\psi(\cdot),\nabla\psi(\cdot),\phi(\cdot))\in
    L^2(D).
\end{align}
In such case, one can easily see that equation \eqref{eq-1-section-beam} is a
particular case of equation \eqref{SDE}. In order to make use of the results
presented in the preceding section, one also need to verify that all the
assumptions \ref{growth condition of f}, \ref{growth condition of
g}, \ref{globally Lipschitz condition of g} and \ref{locally
Lipschitz condition} given in the preceding section on the functions
$f$ and $g$ are fulfilled. To prove the local lipschitz continuity
of the function $f$, we first notice first that
$\mathcal{D}(A)\subset H^2(D)$. Hence by the Sobolev embedding
theorem, when $n=1$, we have $H^2(D)\hookrightarrow \mathcal{C}^1(D)$, so
there exists a constant $M$ such that
$|\psi|_{L^{\infty}(D)}+|\nabla\psi|_{L^{\infty}(D)}\leq
M|\psi|_{H^2(D)}.$ Take $\phi_i\in H$ and
$\psi_i\in\mathcal{D}(A)\subset H^2(D)$, $i=1,2$ such that
$|\psi_i|_{H^2(D)}\leq N$. It follows that
$|\psi|_{L^{\infty}(D)}\leq MN$ and $|\nabla\psi|_{L^{\infty}(D)}\leq
MN$ which gives that $|\psi(x)|\leq MN$ and $|\nabla\psi(x)|\leq MN$
for almost all $x\in D$. We obtain on the basis of the first
assumption \ref{assu_sec_beam-Lip. con._1} and the boundedness assumption of the domain $D$ that
\begin{align}\label{eq-lip-v-1}
   &|f(t,\psi_1,\phi_1)-f(t,\psi_2,\phi_2)|_{L^2(D)}\nonumber\\
   &=\int_D|\Upsilon(t,\psi_1(x),\nabla\psi_1(x),\phi_1(x))-
   \Upsilon(t,\psi_2(x),\nabla\psi_2(x),\phi_2(x))|^2dx\nonumber\\
   &\leq \int_D
   L_{MN}|\psi_1(x)-\psi_2(x)|^2+L_{MN}|\nabla\psi_1(x)-\nabla\psi_2(x)|^2\\
   &\hspace{2cm}+L|\phi_1(x)-\phi_2(x)|^2dx\nonumber\\
   &\leq
   L_{MN}|D||\psi_1-\psi_2|^2_{L^{\infty}(D)}+L_{MN}|D||\nabla\psi_1-\nabla\psi_2|^2_{L^{\infty}(D)}+L|\phi_1-\phi_2|^2_{L^2(D)}\nonumber\\
   &\leq M^2|D|L_{MN}|\psi_1-\psi_2|^2_{H^2(D)}+L|\phi_1-\phi_2|^2_{L^2(D)},\nonumber
\end{align}
In particular, if the function $\Upsilon$ doesn't depend on the third variable, that is $f(t,\psi,\phi)=\Upsilon(t,\cdot,\psi(\cdot),\phi(\cdot))$. The Sobolev embedding theorem tells us that $H^2(D)\hookrightarrow\mathcal{C}(D)$, for $n\leq 3$, which implies that there exists $K$ such that
$|\psi|_{L^{\infty}(D)}\leq
K|\psi|_{H^2(D)}.$ Take $\phi_i\in H$ and
$\psi_i\in\mathcal{D}(A)\subset H^2(D)$, $i=1,2$ such that
$|\psi_i|_{H^2(D)}\leq N$. Again, in view of the assumption \ref{assu_sec_beam-Lip. con._1}, we infer that
\begin{align}\label{eq-lip-v-2}
   |f(t,\psi_1,\phi_1)-f(t,\psi_2,\phi_2)|_{L^2(D)}%\nonumber\\
   &=\int_D|\Upsilon(t,\psi_1(x),\phi_1(x))-
   \Upsilon(t,\psi_2(x),\phi_2(x))|^2dx\nonumber\\
   &\leq \int_D
   L_{KN}|\psi_1(x)-\psi_2(x)|^2+L|\phi_1(x)-\phi_2(x)|^2dx\\
   &\leq K^2|D|L_{KN}|\psi_1-\psi_2|^2_{H^2(D)}+L|\phi_1-\phi_2|^2_{L^2(D)}.\nonumber
\end{align}

From \eqref{eq-lip-v-1} and \eqref{eq-lip-v-2}, we see that the function $f$ defined by \ref{sec_beam_function_f} is locally Lipschitz continuous which verifies Assumption \ref{locally Lipschitz condition}.

For the growth condition \ref{growth condition of f} of $f$, by
making use of Assumption \ref{assu_sec_beam-growth-cond.1}, it can
be easily achieved as follows
\begin{align*}
    |f(t,\psi,\phi)|^2_{L^2(D)}&=\int_D|\Upsilon(t,x,\psi(x),\nabla\psi(x),\phi(x))|^2dx\\
    &\leq \int_D L_{\Upsilon}(1+|\phi(x)|^2)dx\\
    &\leq L_{\Upsilon}|D|(1+|\phi|^2_{L^2(D)})\\
    &\leq L_{\Upsilon}|D|(1+|\psi|^2_{H^2(D)}+|\phi|^2_{L^2(D)}).
\end{align*}
Let us now show that the global Lipschitz condition \eqref{assu: globally
Lipschtiz condition} are satisfied for the function $g$ defined by
\eqref{sec_beam_function_g}. Take $\phi_i\in L^2(D)$ and
$\psi_i\in\mathcal{D}(A)$. By using Assumption
\ref{assu_sec_beam-lip.con.2}, an analogous calculation as verifying the Lipschitz continuity of $f$ before,
shows that if $n=1$, then
\begin{align*}
    & \int_Z|g(t,\psi_1,\phi_1)-g(t,\psi_2,\phi_2)|^2_{L^2(D)}\nu(dz)\\
     &=\int_Z\int_D|\Pi(t,x,\psi_1(x),\nabla\psi_1(x),\phi_1(x))-\Pi(t,x,\psi_2(x),\nabla\psi_2(x),\phi(x))|^2dx\nu(dz)\\
&=\int_D\int_Z|\Pi(t,x,\psi_1(x),\nabla\psi_1(x),\phi_1(x))-\Pi(t,x,\psi_2(x),\nabla\psi_2(x),\phi(x))|^2\nu(dz)dx\\\
&\leq
\int_DL'|\psi_1(x)-\psi_2(x)|^2+L'|\nabla\psi_1(x)-\nabla\psi_2(x)|^2+L'|\phi_1(x)-\phi_2(x)|^2dx\\
&=L'|\psi_1-\psi_2|^2_{L^2(D)}+L'|\nabla\psi_1-\nabla\psi_2|^2_{L^2(D)}+L'|\phi_1-\phi_2|^2_{L^2(D)}\\
&\leq
L'|D||\psi_1-\psi_2|^2_{L^{\infty}(D)}+L'|D||\nabla\psi_1-\nabla\psi_2|^2_{L^{\infty}(D)}+L'|\phi_1-\phi_2|^2_{L^2(D)}\\
&\leq
L'|D|M^2|\psi_1-\psi_2|^2_{H^2(D)}+L'|\phi_1-\phi_2|^2_{L^2(D)},
\end{align*}
and if $n\leq 3$ and $\Pi$ does depends on the third variable, then
 \begin{align*}
     \int_Z|g(t,\psi_1,\phi_1)&-g(t,\psi_2,\phi_2)|^2_{L^2(D)}\nu(dz)\\
     &=\int_Z\int_D|\Pi(t,x,\psi_1(x),\phi_1(x))-\Pi(t,x,\psi_2(x),\phi(x))|^2dx\nu(dz)\\
&\leq
\int_DL'|\psi_1(x)-\psi_2(x)|^2+L'|\phi_1(x)-\phi_2(x)|^2dx\\
&=L'|\psi_1-\psi_2|^2_{L^2(D)}+L'|\nabla\psi_1-\nabla\psi_2|^2_{L^2(D)}+L'|\phi_1-\phi_2|^2_{L^2(D)}\\
&\leq
L'|D|K^2|\psi_1-\psi_2|^2_{H^2(D)}+L'|\phi_1-\phi_2|^2_{L^2(D)},
\end{align*}
which verifies the global Lipschitz condition \eqref{assu: globally
Lipschtiz condition} of the function $g$. In exactly the same
manner, we have
\begin{align*}
     \int_Z|g(t,\psi,\phi)|^2_{L^2(D)}\nu(dz)
     &=\int_Z\int_D|\Pi(t,x,\psi(x),\nabla\psi(x),\phi(x))|^2dx\nu(dz)\\
     &=\int_D\int_Z|\Pi(t,x,\psi(x),\nabla\psi(x),\phi(x))|^2\nu(dz)dx\\
     &\leq L_{\Pi}\int_D(1+|\phi(x)|^2)dx\\
     &\leq L_{\Pi}|D|(1+|\psi|^2_{H^2(D)}+|\phi|^2_{L^2(D)}).
\end{align*}
To deal with the Equation \eqref{eq-1-section-beam} with the clamped
boundary condition
$$u=\frac{\partial u}{\partial n}=0\ \text{on}\ \partial D,$$
we define an operator $C$ by
\begin{align*}
\mathcal{D}(C)&=\{\varphi\in H^4(D):\ \varphi=\frac{\partial
\varphi}{\partial
n}=0\ \text{on}\ \partial D\}\\
 C\varphi &=\triangle^2\varphi, \text{ for }\varphi\in \mathcal{D}(C).
\end{align*}
It is easy to observe that the operator $C$ is positive. To see
this, take $\varphi\in\mathcal{D}(C)$. Then the Green formula tells
us that
\begin{align*}
   \langle C\varphi,\varphi\rangle_H=\int_D
   \triangle^2\varphi\cdot\varphi\
   dx=\int_D(\triangle\varphi,\triangle\varphi) dx=\|\triangle
   \varphi\|_H^2\geq0.
\end{align*}
Further, by Lemma 9.17 in \cite{[Gilbarg]} , since
$\mathcal{D}(C)\subset H^2(D)\cap H^2_0(D)$, we have
\begin{align*}
 \langle C\varphi,\varphi\rangle_H=\|\triangle
   \varphi\|_H^2\geq \frac{1}{K}\|u\|^2_H,\ \ \phi\in\mathcal{D}(C),
\end{align*}
where the constant $K$ is independent of $\varphi$. This part also
shows that the operator $C$ is uniformly positive with $C\geq
\frac{1}{K}$. In this case, we set $$A=C^{\frac{1}{2}}.$$ Then by the
uniqueness of positive square root operator, we find out that
$A=-\triangle$ and $\mathcal{D}(A)=\{\varphi\in H^2(D):\
\varphi=\frac{\partial \phi}{\partial n}=0\text{ on }\partial D\}$.
Since $\mathcal{D}(A)\subset H^1_0(D)$, by the Poincar\'{e}
inequality, we infer that $A\geq \mu I$, for some $\mu>0$.
Analogously, we define
\begin{align*}
       &B \psi=-\Delta\psi,\ \ \psi\in \mathcal{D}(A),\\
       &\mathcal{D}(B)=H^2(D)\cap H^1_0(D).
\end{align*}
By adapting the definitions \eqref{sec_beam_function_f},
\eqref{sec_beam_function_g} of the functions $f$ and $g$ and
assumptions \eqref{assu_sec_beam-Lip.
con._1}-\eqref{assu_sec_beam-growth-cond.2} of the functions
$\Upsilon$ and $\Pi$, all the requirements on the functions $f$ and
$g$ are fulfilled in the same way as above.

\section{Proofs}
In order to prove Theorem \ref{theo: exsitence and uniqueness of mild
solution}, we will first establish the following auxiliary Lemma.

\begin{proposition}\label{Lem:predictability}Suppose that $Z:\mathbb{R}_+\time\Omega\time \Omega\rightarrow\mathcal{H}$ is a progressively measurable process.
Let $X(t)=e^{t\mathcal{A}}Z(t)$, $t\geq0$ and
$Y(t)=e^{-t\mathcal{A}}Z(t)$ Then $X(t)$ and $Y(t)$, $t\geq0$ are
progressively measurable processes.
\end{proposition}
\begin{proof}[Proof of Proposition \ref{Lem:predictability}]
   Define a function $\alpha:\mathbb{R}_+\times\mathcal{H}\ni(t,x)\mapsto e^{t\mathcal{A}}x\in\mathcal{H}$.  Since $e^{t\mathcal{A}}$, $t\geq0$ is a contraction $C_0$-semigroup, so $\|e^{t\mathcal{A}}\|_{\mathcal{L}(\mathcal{H})}\leq1$ and for every $x\in\mathcal{H}$, $\alpha(\cdot,x)$ is continuous. Also, for every $t\geq0$, $\alpha(t,\cdot)$ is continuous. Indeed, let us fix $x_0\in\mathcal{H}$. Then for every $x\in\mathcal{A}$
   \begin{align*}
          \|\alpha(t,x)-\alpha(t,x_0)\|_{\mathcal{H}}=\|e^{t\mathcal{A}}(x-x_0)\|_{\mathcal{H}}\leq\|x-x_0\|_{\mathcal{H}}.
   \end{align*}
   Thus $\alpha(t,\cdot)$ is continuous. This shows that the function $\alpha$ is separably continuous. Since by the assumption the process $Z$ is progressively measurable, one can see that the mapping
   \begin{align*}
    \mathbb{R}_+\times\Omega\ni (s,\omega)\mapsto (s,Z(s,\omega))\in\mathbb{R}_+\times\mathcal{H}
   \end{align*}
    is progressively measurable as well. So the composition mapping
   $$    \mathbb{R}_+\times\Omega\ni(s,\omega)\mapsto (s,Z(s,\omega))\mapsto \alpha(s,Z(s,\omega))\in\mathcal{H}$$
   is progressively measurable, 
and hence, the process $X(t)$, $t\geq0$ is progressively measurable. The
progressively measurability of process $Y(t)$, $t\geq0$ follows from the above
proof with $\mathcal{A}$
 replaced by $-\mathcal{A}$.
\end{proof}

\begin{proof}[Proof of Theorem \ref{theo: exsitence and uniqueness of mild solution}]
   Given $T\geq0$, we denote by $\mathcal{M}_T^2$ the set of all $\mathcal{H}$-valued progressively measurable processes $X:\mathbb{R}_+\times\Omega\rightarrow\mathcal{H}$ such that
   \begin{align*}
         \|X\|_T:=\sup_{0\leq t\leq T}(\E\|X(t)\|^2_{\mathcal{H}})^{\frac{1}{2}}<\infty.
   \end{align*}
   Then the space $\mathcal{M}^2_T$ endowed with the norm $\|X\|_{\lambda}:=\sup_{0\leq t\leq T}e^{-\lambda t}(\E\|X(t)\|^2_{\mathcal{H}})^{\frac{1}{2}}$, $\lambda>0$, is a Banach space. Note that the norms $\|\cdot\|_{\lambda}$, $\lambda\geq0$, are equivalent.
   Let us define a map $\Phi_T:\mathcal{M}^2_T\rightarrow\mathcal{M}_T^2$ by
   \begin{align*}
      (\Phi_TX)(t)=e^{t\mathcal{A}}\mathfrak{u}_0+\int_0^te^{(t-s)\mathcal{A}}F(s,X(s))\,ds+\int_0^{t}\int_Ze^{(t-s)\mathcal{A}}G(s,X(s),z)\tilde{N}(ds,dz).
   \end{align*}
   We shall show that the operator $\Phi_T$ is a contraction operator on $\mathcal{M}^2_T$ for sufficiently large values of $\lambda$. We first verify that if $X\in\mathcal{M}^2_T$, then $\Phi_TX\in\mathcal{M}^2_T$.\\
   \textbf{\textit{Claim 1}}. The process $\int_0^te^{(t-s)\mathcal{A}}F(s,X(s))\,ds$, $t\in[0,T]$, is progressively measurable.\\
   Proof of Claim $1$:
   Since $F$ is $\mathcal{B}(\mathbb{R}_+)\otimes\mathcal{B}(\mathcal{H})/\mathcal{B}(\mathcal{H})$-measurable and the process $X(t),\ t\in[0,T]$ is progressively measurable,
   so the mapping
   $$[0,T]\times\Omega\ni(t,\omega)\mapsto (t,X(t,\omega))\mapsto F(t,X(t,\omega))\in\mathcal{H}$$
   is progressively measurable as well.

    By Lemma \ref{Lem:predictability} we know that $e^{(-s)\mathcal{A}}F(s,X(s))$ is also progressively measurable. It then follows from the Fubini Theorem that the integral $\int_0^te^{(-s)\mathcal{A}}F(s,X(s))\,ds$ is $\mathcal{F}_t$-measurable.

   Since the process $[0,T]\ni t\mapsto\int_0^te^{(-s)\mathcal{A}}F(s,X(s))\,ds\in\mathcal{H}$ is continuous in $t$, this together with the adaptedness assert the progressively measurability of the process $\int_0^te^{(-s)\mathcal{A}}F(s,X(s))\,ds$, $t\in[0,T]$.    Again, by Proposition \ref{Lem:predictability}, we infer that the process
   \begin{align*}
           \int_0^te^{(t-s)\mathcal{A}}F(s,X(s))\,ds=e^{t\mathcal{A}}\int_0^te^{-s\mathcal{A}}F(s,X(s))\ ds,\ \ t\in[0,T],
   \end{align*}
    is also progressively measurable.\\
    \textbf{\textit{Claim 2}}. The process $\int_0^{t}\int_Z e^{(t-s)\mathcal{A}}G(s,X(s),z)\tilde{N}(ds,dz)$, $t\in[0,T]$ has a progressively measurable version.\\
 Proof of Claim $2$: First of all, we show that the process $\int_0^{t}\int_Ze^{(t-s)\mathcal{A}}G(s,X(s),z)\tilde{N}(ds,dz)$,
$0\leq t\leq T$ is $\mathfrak{F}$-adapted. Let us fix $t\in[0,T]$. Since by assumption the process $X$ is progressively measurable, a similar argument as in the proof of claim $1$ shows that the integrand function $e^{(t-s)\mathcal{A}}G(s,X(s),z)$ is progressively measurable.
Hence by assumption \ref{growth condition of g}, the integral process
\begin{align*}
       \int_0^r\int_Z1_{(0,t]}e^{(t-s)\mathcal{A}}G(s,X(s),z)\tilde{N}(ds,dz),\ \ r\in[0,T]
\end{align*}
is well defined. Moreover, we know from \cite{[Zhu_2010]} that
this process
is none but a martingale. In particular, for each $r\in[0,T]$, the integral
$\int_0^r\int_Z1_{(0,t]}e^{(t-s)\mathcal{A}}G(s,X(s),z)\tilde{N}(ds,dz)$ is
$\mathcal{F}_r$-measurable. By taking $r=t$, we infer that
$\int_0^t\int_Z1_{(0,t]}e^{(t-s)\mathcal{A}}G(s,X(s),z)\tilde{N}(ds,dz)$ is
$\mathcal{F}_t$-measurable.

Also, notice that the stochastic convolution process $\int_0^{t}\int_Z e^{(t-s)\mathcal{A}}G(s,X(s),z)\tilde{N}(ds,dz)$, $t\in[0,T]$ has a c\`{a}dl\`{a}g modification, see \cite{[Brzezniak_Hau_Zhu]}. Therefore, we infer that the process     $$\int_0^{t}\int_Z e^{(t-s)\mathcal{A}}G(s,X(s),z)\tilde{N}(ds,dz),\ t\in[0,T]$$
  has a progressively measurable version.

  In conclusion, the process $(\Phi_TX)(t)$, $t\geq 0$ is progressively measurable. So it remains to show that $\|\Phi_TX\|_{\lambda}^2<\infty$.

First, we find out that
  \begin{align*}
        \|\Phi_T X\|_{\lambda}&\leq \|e^{\cdot\mathcal{A}}\mathfrak{u}_0\|_{\lambda}+\left\|\int_0^{\cdot}e^{(\cdot-s)\mathcal{A}}F(s,X(s))\,ds\right\|_{\lambda}\\
        &\hspace{2cm}+\left\|\int_0^{\cdot}\int_Ze^{(\cdot-s)\mathcal{A}}G(s,X(s),z)\tilde{N}(ds,dz)\right\|_{\lambda}\\
        &=I_1+I_2+I_3.
  \end{align*}
  For the first term $I_1$, by the definition of the norm $\|\cdot\|_{\lambda}$, we have
  \begin{align*}
         I_1=\|e^{\cdot\mathcal{A}}\mathfrak{u}_0\|_{\lambda}
         =\sup_{0\leq t\leq T}e^{-\lambda t}\left(\E\left\| e^{t\mathcal{A}}\mathfrak{u}_0\right\|_{\mathcal{H}}^2\right)^{\frac{1}{2}}\leq \|\mathfrak{u}_0\|_{\mathcal{H}}.
  \end{align*}
  where we used the fact that $e^{t\mathcal{A}}$ is a contraction $C_0$-semigroup.
  Also, by using the Cauchy-Schwartz inequality and the growth conditions \eqref{growth condition of f} and \eqref{growth condition of g}, for the second term $I_2$, we obtain
  \begin{align*}
       I_2&=
           \sup_{0\leq t\leq T}e^{-\lambda t}\left(\E\left\|\int_0^t e^{(t-s)\mathcal{A}}F(s,X(s))\,ds\right\|_{\mathcal{H}}^2\right)^{\frac{1}{2}}\\
           &\leq \sup_{0\leq t\leq T}e^{-\lambda t}T^{\frac{1}{2}}\left(\E\int_0^t \|F(s,X(s))\|^2_{\mathcal{H}}ds\right)^{\frac{1}{2}}\\
           &\leq \sup_{0\leq t\leq T}e^{-\lambda t}T^{\frac{1}{2}}K_f^{\frac{1}{2}}\left(\E\int_0^t (1+ \|X(s)\|^2_{\mathcal{H}})\,ds\right)^{\frac{1}{2}}\\
           &\leq TK_f^{\frac{1}{2}}+\sup_{0\leq t\leq T}T^{\frac{1}{2}}K_f^{\frac{1}{2}}\left(\int_0^t e^{-2\lambda(t-s)}ds\sup_{0\leq s\leq T}\E\ e^{-2\lambda s}\|X(s)\|_{\mathcal{H}}^2\right)^{\frac{1}{2}}\\
           &\leq TK_f^{\frac{1}{2}}+\frac{1}{2\lambda}T^{\frac{1}{2}}K_f^{\frac{1}{2}}\sup_{0\leq s\leq T}e^{-\lambda s}\left(\E \|X(s)\|_{\mathcal{H}}^2\right)^{\frac{1}{2}}.
  \end{align*}

In the same way, we have
  \begin{align*}
       I_3
          &=\sup_{0\leq t\leq T}e^{-\lambda t}\left(\E\left\|\int_0^t\int_Ze^{(t-s)\mathcal{A}}G(s,X(s),z)\tilde{N}(ds,dz)\right\|_{\mathcal{H}}^2\right)^{\frac{1}{2}}\\
          &=\sup_{0\leq t\leq T}e^{-\lambda t}\left(\E\int_0^t\int_Z\|e^{(t-s)\mathcal{A}}G(s,X(s),z)\|_{\mathcal{H}}^2\nu(dz)\,ds\right)^{\frac{1}{2}}\\
           &\leq K_g^{\frac{1}{2}}\sup_{0\leq t\leq T}e^{-\lambda t}\left(\E\int_0^t(1+\|X(s)\|^2_{\mathcal{H}})\,ds\right)^{\frac{1}{2}}\\
           &\leq K_g^{\frac{1}{2}}T^{\frac{1}{2}}+K_g^{\frac{1}{2}}\left(\int_0^t e^{-2\lambda(t-s)}ds\sup_{0\leq s\leq T}\E\ e^{-2\lambda s}\|X(s)\|_{\mathcal{H}}^2\right)^{\frac{1}{2}}\\
           &\leq K_g^{\frac{1}{2}}T^{\frac{1}{2}}+\frac{1}{2\lambda}K_g^{\frac{1}{2}}\|X(s)\|_{\lambda},
  \end{align*}
  where the second equality follows from the isometry property of It\^{o} integral w.r.t. compensated Poisson random measures and the second inequality follows from the growth condition \eqref{growth condition of g} of the function $g$.
  Combining the above three estimates, we get
  \begin{align}
       \|\Phi_T(X)\|_{\lambda}^2\leq\|\mathfrak{u}_0\|_{\mathcal{H}}^2+TK_f^{\frac{1}{2}}+K_g^{\frac{1}{2}}T^{\frac{1}{2}}
       +\frac{1}{2\lambda}(T^{\frac{1}{2}}K_f^{\frac{1}{2}}+K_g^{\frac{1}{2}})\|X(s)\|_{\lambda}<\infty,
  \end{align}
  which implies that $\Phi_T(X)\in\mathcal{M}_T^2$.\\
   Now we shall show that $\Phi_T$ is a contraction provided $\lambda$ is chosen to be large enough. For this we take $X_1,X_2\in\mathcal{M}_T^2$. Then we obtain the following inequlaity
  \begin{align}\label{eq-12}
    \|\Phi_T(X_1)-\Phi_T(X_2)\|_{\lambda}
    &=\Big{\|}\int_0^{\cdot} e^{(\cdot-s)\mathcal{A}}\Big{(}F(s,X_1(s))-F(s,X_2(s))\Big{)}ds\nonumber\\
    &\hspace{1cm}+\int_0^\cdot\int_Ze^{(\cdot-s)\mathcal{A}}\Big{(}G(s,X_1(s),z)-G(s,X_2(s),z)\Big{)}\tilde{N}(ds,dz)\Big{\|}_{\lambda}^2\nonumber\\
    &\leq \left\|\int_0^\cdot e^{(\cdot-s)\mathcal{A}}\Big{(}F(s,X_1(s))-F(s,X_2(s))\Big{)}ds\right\|_{\lambda}\\
    &\hspace{1cm}+\left\|\int_0^\cdot\int_Ze^{(\cdot-s)\mathcal{A}}\Big{(}G(s,X_1(s),z)-G(s,X_2(s),z)\Big{)}\tilde{N}(ds,dz)\right\|_{\lambda}\nonumber\\
    &=I_4+I_5.\nonumber
  \end{align}
  Observe first that, similarly to the estimates on $I_2$ before, we have
    \begin{align*}
      I_4
      &=\sup_{0\leq t\leq T}e^{-\lambda t}\left(\E\left\|\int_0^te^{(t-s)\mathcal{A}}\Big{(}F(s,X_1(s))-F(s,X_2(s))\Big{)}ds\right\|_{\mathcal{H}}^2\right)^{\frac{1}{2}}\\
      &\leq T^{\frac{1}{2}}\sup_{0\leq t\leq T}e^{-\lambda t}\left(\E\int_0^t\left\|e^{(t-s)\mathcal{A}}\Big{(}F(s,X_1(s))-F(s,X_2(s))\Big{)}\right\|_{\mathcal{H}}^2ds\right)^{\frac{1}{2}}\\
      &\leq T^{\frac{1}{2}}L_f\sup_{0\leq t\leq T}e^{-\lambda t}\left(\E\int_0^t\left\|X_1(s)-X_2(s)\right\|_{\mathcal{H}}^2ds\right)^{\frac{1}{2}}\\
     &\leq T^{\frac{1}{2}}L_f\left(\int_0^Te^{-2\lambda(t-s)}ds\right)^{\frac{1}{2}}\left(\sup_{0\leq s\leq T}\E e^{-2\lambda s}\|X_1(s)-X_2(s)\|_{\mathcal{H}}^2\right)^{\frac{1}{2}}\\
      &\leq \frac{T^{\frac{1}{2}}L_f}{2\lambda}\|X_1(s)-X_2(s)\|_{\lambda},
  \end{align*}
  where we used the Cauchy-Schwartz inequality and the globally Lipschitz condition \eqref{growth condition of f} on $f$.
  Also on the basis of the It\^{o} isometry property, see \cite{[Zhu_2010]}, and the global Lipschitz condition \eqref{growth condition of g} on $g$, we find out that
  \begin{align*}
     I_5     &=\sup_{0\leq t\leq T}e^{-\lambda t}\left(\E\left\|\int_0^t\int_Ze^{(t-s)\mathcal{A}}\Big{(}G(s,X_1(s),z)-G(s,X_2(s),z)\Big{)}\tilde{N}(ds,dz)\right\|_{\mathcal{H}}^2\right)^{\frac{1}{2}}\\
          &=\sup_{0\leq t\leq T}e^{-\lambda t}\left(\E\int_0^t\int_Z\left\|e^{(t-s)\mathcal{A}}\Big{(}G(s,X_1(s),z)-G(s,X_2(s),z)\Big{)}\right\|_{\mathcal{H}}^2\nu(dz)\,ds\right)^{\frac{1}{2}}\\
          &\leq \E\int_0^T\left(e^{-\lambda t}\int_0^t\int_Z\|G(s,X_1(s),z)-G(s,X_2(s),z)\|_{\mathcal{H}}^2\nu(dz)\,ds\right)dt\\
          &\leq L_g\sup_{0\leq t\leq T}e^{-\lambda t}\left(\E\int_0^t\left\|X_1(s)-X_2(s)\right\|_{\mathcal{H}}^2ds\right)^{\frac{1}{2}}\\
 &\leq L_g\left(\int_0^Te^{-2\lambda(t-s)}ds\right)^{\frac{1}{2}}\left(\sup_{0\leq s\leq T}\E e^{-2\lambda s}\|X_1(s)-X_2(s)\|_{\mathcal{H}}^2\right)^{\frac{1}{2}}\\
    %       \end{align*}
    %      \begin{align*}
      &\leq \frac{L_g}{2\lambda}\|X_1(s)-X_2(s)\|_{\lambda}.
           \end{align*}
By substituting above estimates into the right-side of inequality
\eqref{eq-12}, we get that
\begin{align}\label{proof_theo_eq_1}
      \|\Phi_T(X_1)-\Phi_T(X_2)\|_{\lambda}^2\leq \frac{T^{\frac{1}{2}}L_f+L_g}{2\lambda}\|X_1-X_2\|_{\lambda}^2.
\end{align}
Therefore, if
$\frac{T^{\frac{1}{2}}L_f+L_g}{2\lambda}\leq\frac{1}{2}$,
 then $\Phi_T$ is a strict contraction in $\mathcal{M}_T^2$. We then apply the Banach Fixed Point Theorem to
infer that $\Phi_T$ has a unique fixed point in $\mathcal{M}_T^2$.
This implies that for any $0<T<\infty$, there exists a unique (up to modification) process
$(\bar{\mathfrak{u}}(t))_{0\leq t\leq T}\in\mathcal{M}_T^2$ such that
$\bar{\mathfrak{u}}=\Phi_T(\bar{\mathfrak{u}})$ in $\mathcal{M}_T^2$. \\
Notice that
we can always find a c\`{a}d\`{a}g version satisfying \eqref{defi: mild solution}. Indeed, we know that
the uniqueness holds in the sense that if there exists another process $\mathfrak{v}\in\mathcal{M}_T^2$ satisfying $\mathfrak{v}=\Phi_T\mathfrak{v}$, then for every $t\in[0,T]$, $\bar{\mathfrak{u}}(t)=\mathfrak{v}(t)$, $\mathbb{P}$-a.s. Let $\mathcal{N}:=\{X\in\mathcal{M}^2_T: X=\Phi_T X\}$. By the uniqueness, the set $\mathcal{N}$ contains all stochastically equivalent processes of the process $\bar{\mathfrak{u}}$. Among those stochastically equivalent processes in $\mathcal{N}$, we are trying to find a version $(\mathfrak{u}(t))$ of $(\bar{\mathfrak{u}}(t))$ such that $(\mathfrak{u}(t))$ is c\`{a}dl\`{a}g and $(\mathfrak{u}(t))$ satisfies \eqref{defi: mild solution}. For this, we define
\begin{align*}
      \mathfrak{u}(t)&=(\Phi_T\bar{\mathfrak{u}})(t)\\
      &=e^{t\mathcal{A}}\mathfrak{u}_0+\int_0^te^{(t-s)\mathcal{A}}F(s,\bar{\mathfrak{u}}(s))\,ds+\int_0^t\int_Ze^{(t-s)\mathcal{A}}G(s,\bar{\mathfrak{u}}(s),z)\tilde{N}(ds,dz),\  t\in[0,T],
\end{align*}
Note that the process $\mathfrak{u}$ is, by definition, c\`{a}dl\`{a}g, see \cite{[Brzezniak_Hau_Zhu]}.  Hence, we may define
\begin{align*}
      \mathfrak{v}(t)&=(\Phi_T\mathfrak{u})(t)\\
      &=e^{t\mathcal{A}}\mathfrak{u}_0+\int_0^te^{(t-s)\mathcal{A}}F(s,\mathfrak{u}(s))\,ds+\int_0^t\int_Ze^{(t-s)\mathcal{A}}G(s,\mathfrak{u}(s-),z)\tilde{N}(ds,dz),\  t\in[0,T].
\end{align*}
We observe by the definition of two processes $\mathfrak{u}$ and $\bar{\mathfrak{u}}$ that for all $t\in[0,T]$, $\E\|\mathfrak{u}(t)-\bar{\mathfrak{u}}(t)\|_{\mathcal{H}}^2=0$. This implies that $\mathfrak{u}$ is a c\`{a}dl\`{a}g version of $\bar{\mathfrak{u}}$. From this, we also find out that $\E\int_0^T \|\mathfrak{u}(t)-\bar{\mathfrak{u}}(t)\|_{\mathcal{H}}^2 dt=0$.
It follows form the continuity of functions $F(t,x)$ and $G(t,x,z)$ in the variable $x$ that for all $t\in[0,T]$,
\begin{align*}
    \E\|\mathfrak{u}(t)-\mathfrak{v}(t)\|^2_{\mathcal{H}} \leq& 2\E\left\|\int_0^te^{(t-s)\mathcal{A}}\Big{(}F(s,\bar{\mathfrak{u}}(s))-F(s,\mathfrak{u}(s))\Big{)}ds\right\|^2_{\mathcal{H}}\\
    &+ 2\E\Big{\|}\int_0^t\int_ZG(s,\bar{\mathfrak{u}}(s),z)-G(s,\mathfrak{u}(s-),z)\tilde{N}(ds,dz)\Big{\|}^2_{\mathcal{H}}\\
=& 2\E\left\|\int_0^te^{(t-s)\mathcal{A}}\Big{(}F(s,\bar{\mathfrak{u}}(s))-F(s,\mathfrak{u}(s))\Big{)}ds\right\|^2_{\mathcal{H}}\\
    &+ 2\E\int_0^t\int_Z\Big{\|}G(s,\bar{\mathfrak{u}}(s),z)-G(s,\mathfrak{u}(s),z)\Big{\|}^2_{\mathcal{H}}\nu(dz)\,ds=0.
\end{align*}
Hence, we infer that for all $t\in[0,T]$,
\begin{align}\label{eq-701}
\begin{split}
    \mathfrak{u}(t)&=\mathfrak{v}(t)\\
    &=e^{t\mathcal{A}}\mathfrak{u}_0+\int_0^te^{(t-s)\mathcal{A}}F(s,\mathfrak{u}(s))\,ds+\int_0^t\int_Ze^{(t-s)\mathcal{A}}G(s,\mathfrak{u}(s-),z)\tilde{N}(ds,dz),\ \mathbb{P}\text{-a.s.},
    \end{split}
\end{align}
which shows that $\mathfrak{u}$ satisfies \eqref{defi: mild solution}.
Since both sides of above equality are c\`{a}d\`{a}g, the stochastically equivalence becomes $\mathbb{P}$-equivalence.  More precisely,
we obtain a pathwise uniqueness c\`{a}dl\`{a}g process in $\mathcal{M}^2_T$ such that for all $t\in[0,T]$, the equality \eqref{defi: mild solution} holds. However, if we release the c\`{a}dl\`{a}g property, the pathwise uniqueness no longer holds and we could only have stochastically uniqueness instead.\\
 Now the
uniqueness feature of a solution on any given priori time interval $[0,T]$
allows us to amalgamate them into a solution $(u(t))_{t\geq0}$ to
problem \eqref{SDE1} on the positive real half-line. Moreover, this
solution $(u(t))_{t\geq0}$ to problem \eqref{SDE1} is unique up to distinguishable.

 In other words, for $t\geq 0$,
\begin{align}\label{eq32}
      \mathfrak{u}(t)=e^{t\mathcal{A}}\mathfrak{u}_0+\int_0^te^{(t-s)\mathcal{A}}F(s,\mathfrak{u}(s))\,ds+\int_0^t\int_Ze^{(t-s)\mathcal{A}}G(s,\mathfrak{u}(s-),z)\tilde{N}(ds,dz)\ \ \mathbb{P}\text{-a.s.}.
\end{align}
Note also that since $\mathfrak{u}\in\mathcal{M}^2_T$, for every
$T>0$,
\begin{align*}
     &\E\int_0^T\|F(s,\mathfrak{u}(s))\|^2_{\mathcal{H}}ds\leq L_f^2 \E\int_0^T(1+\|\mathfrak{u}(s)\|^2_{\mathcal{H}})\,ds\leq L_f^2T(1+\|\mathfrak{u}\|_T^2)<\infty;\\
     &\E\int_0^T\int_Z\|G(s,\mathfrak{u}(s),z)\|^2_{\mathcal{H}}\nu(dz)\,ds\leq L_g^2 \E\int_0^T(1+\|\mathfrak{u}(s)\|^2_{\mathcal{H}})\,ds\leq L_g^2T(1+\|\mathfrak{u}\|_T^2)<\infty;
\end{align*}
which shows that
$F(\cdot,\mathfrak{u}(\cdot))\in\mathcal{M}^2_{loc}(\mathcal{B}\mathcal{F})$
and
$G(\cdot,\mathfrak{u}(\cdot),z)\in\mathcal{M}^2_{loc}(\hat{\mathcal{P}})$.
In conclusion, Problem \eqref{SDE1} has a unique mild solution.

Now let us suppose that $\mathfrak{u}_0\in\mathcal{D}(\mathcal{A})$,
$F(\cdot,\mathfrak{u}(\cdot))\in\mathcal{M}^2_{loc}(\mathcal{B}\mathcal{F};\mathcal{D}(\mathcal{A}))$
and
$G(\cdot,\mathfrak{u}(\cdot))\in\mathcal{M}^2_{loc}(\hat{\mathcal{P}};\mathcal{D}(\mathcal{A}))$,
where $\mathcal{D}(\mathcal{A})$ is endowed with the graph norm. We
observe that $\mathfrak{u}(t)\in\mathcal{D}(\mathcal{A})$ for every $t\geq0$.
To see this, let us us fix $t\geq0$.
 Let $R(\lambda,\mathcal{A})=(\lambda I-\mathcal{A})^{-1}$, $\lambda>0$, be the resolvent of $\mathcal{A}$.
 Since
 $\mathcal{A}R(\lambda,\mathcal{A})=\lambda R(\lambda,\mathcal{A})-I_E$, $\mathcal{A}R(\lambda,\mathcal{A})$ is
 bounded. Hence, since $G(\cdot,\mathfrak{u}(\cdot))\in\mathcal{M}^2_{loc}(\hat{\mathcal{P}};\mathcal{D}(\mathcal{A}))$,
we obtain
\begin{align*}
     R(\lambda,\mathcal{A})\int_0^t\int_Z\mathcal{A}&e^{(t-s)\mathcal{A}}G(s,\mathfrak{u}(s-),z)\tilde{N}(ds,dz)\\
     &=\int_0^t\int_ZR(\lambda,\mathcal{A})\mathcal{A}e^{(t-s)\mathcal{A}}G(s,\mathfrak{u}(s-),z)\tilde{N}(ds,dz)\\
              &=\lambda
              R(\lambda,A)\int_0^t\int_Ze^{(t-s)\mathcal{A}}G(s,\mathfrak{u}(s-),z)\tilde{N}(ds,dz)\\
              &\hspace{1cm}-\int_0^t\int_Ze^{(t-s)\mathcal{A}}G(s,\mathfrak{u}(s-),z)\tilde{N}(ds,dz).
\end{align*}
Thus, it follows that
\begin{eqnarray*}
   &&\hspace{-2truecm}\lefteqn{\int_0^t\int_Ze^{(t-s)\mathcal{A}}G(s,\mathfrak{u}(s-),z)\tilde{N}(ds,dz)} \\ &=&R(\lambda,A)\Big[\lambda\int_0^t\int_Ze^{(t-s)\mathcal{A}}G(s,\mathfrak{u}(s-),z)\tilde{N}(ds,dz)\\
   &&\hspace{2truecm}\lefteqn{-\int_0^t\int_Z\mathcal{A}e^{(t-s)\mathcal{A}}G(s,\mathfrak{u}(s-),z)\tilde{N}(ds,dz)\Big].}
\end{eqnarray*}
Since $\Rng(R(\lambda,A))=\mathcal{D}(A)$, we infer that
$\int_0^t\int_Ze^{(t-s)\mathcal{A}}G(s,\mathfrak{u}(s-),z)\tilde{N}(ds,dz)\in \mathcal{D}(\mathcal{A})$. Here $\Rng$ denotes the range. In a similar manner, we can show that\\ $\int_0^te^{(t-s)\mathcal{A}}F(s,\mathfrak{u}(s))\,ds\in \mathcal{D}(\mathcal{A})$. Hence, $\mathfrak{u}(t)\in\mathcal{D}(\mathcal{A})$.
\\
Now we are in a position to show that $\mathbb{P}\text{-a.s.}$, for all $t\geq0$,
\begin{eqnarray}
\label{eq-1-proof}
     \nonumber
     \mathcal{A}\int_0^t\int_Ze^{(t-s)\mathcal{A}}G(s,\mathfrak{u}(s-),z)\tilde{N}(ds,dz)&=&
     \int_0^t\int_Z\mathcal{A}e^{(t-s)\mathcal{A}}G(s,\mathfrak{u}(s-),z)\tilde{N}(ds,dz)\,,\\
          \mathcal{A}\int_0^te^{(t-s)\mathcal{A}}F(s,\mathfrak{u}(s))\,ds&=&
          \int_0^t\mathcal{A}e^{(t-s)\mathcal{A}}F(s,\mathfrak{u}(s))\,ds\,.
\end{eqnarray}
  For this, let us take $h \in (0,t)$. Since $\frac{e^{h\mathcal{A}}-I}{h}$ is a bounded operator, we have the following
    \begin{align*}
     &\E\left\|\mathcal{A}\int_0^t\int_Ze^{(t-s)\mathcal{A}}G(s,\mathfrak{u}(s-),z)\tilde{N}(ds,dz)-\int_0^t\int_Z\mathcal{A}e^{(t-s)\mathcal{A}}G(s,\mathfrak{u}(s-),z)\tilde{N}(ds,dz)\right\|^2\\
     &\leq 2\E\left\|\left(\frac{e^{h\mathcal{A}}-I}{h}-\mathcal{A}\right)\int_0^t\int_Ze^{(t-s)\mathcal{A}}G(s,\mathfrak{u}(s-),z)\tilde{N}(ds,dz)\right\|^2\\
     &\hspace{1cm}+2\E\left\|\int_0^t\int_Z\left(\frac{e^{h\mathcal{A}}-I}{h}-\mathcal{A}\right)e^{(t-s)\mathcal{A}}G(s,\mathfrak{u}(s-),z)\tilde{N}(ds,dz)\right\|^2\\
     &=2\E\left\|\left(\frac{e^{h\mathcal{A}}-I}{h}-\mathcal{A}\right)\int_0^t\int_Ze^{(t-s)\mathcal{A}}G(s,\mathfrak{u}(s-),z)\tilde{N}(ds,dz)\right\|^2
           \end{align*}
      \begin{align*}
     &\hspace{1cm}+2\E\int_0^t\int_Z\left\|\left(\frac{e^{h\mathcal{A}}-I}{h}-\mathcal{A}\right)e^{(t-s)\mathcal{A}}G(s,\mathfrak{u}(s),z)\right\|^2\nu(dz)\,ds\\
     &:=I(h)+II(h).
  \end{align*}
  Since we showed that $\int_0^t\int_Ze^{(t-s)\mathcal{A}}G(s,\mathfrak{u}(s-),z)\tilde{N}(ds,dz)\in \mathcal{D}(\mathcal{A})$,
  we infer that the term $I(h)$ converges to $0$ a.s. as $h\downarrow0$.

It is easy to see that the integrand
$$\left\|\left(\frac{e^{h\mathcal{A}}-I}{h}-\mathcal{A}\right)e^{(t-s)\mathcal{A}}G(s,\mathfrak{u}(s),z)\right\|^2$$
 is bounded by a function $C_1|\mathcal{A}G(s,\mathfrak{u}(s),z)|^2$ which satisfies $$\E\int_0^t\int_Z|\mathcal{A}G(s,\mathfrak{u}(s),z)|^2\nu(dz)\,ds<\infty$$ for every $t\geq0$
by the assumptions. Since $\mathcal{A}$ is the infinitesimal generator of
the $C_0$-semigroup $(e^{t\mathcal{A}})_{t\geq0}$, the integrand
converges to $0$ pointwise on $[0,t]\times\Omega\times Z$.
Therefore, the Lebesgue
Dominated Convergence Theorem on interchanging a limit and an integral is applicable. So the second term $II(h)$ converges to $0$ as $h\downarrow0$ as well. Therefore, we have
$$\E\left\|\mathcal{A}\int_0^t\int_Ze^{(t-s)\mathcal{A}}G(s,\mathfrak{u}(s-),z)\tilde{N}(ds,dz)-\int_0^t\int_Z\mathcal{A}e^{(t-s)\mathcal{A}}G(s,\mathfrak{u}(s-),z)\tilde{N}(ds,dz)\right\|^2=0,$$
which gives that for $t\geq0$,
$$\mathcal{A}\int_0^t\int_Ze^{(t-s)\mathcal{A}}G(s,\mathfrak{u}(s-),z)\tilde{N}(ds,dz)=\int_0^t\int_Z\mathcal{A}e^{(t-s)\mathcal{A}}G(s,\mathfrak{u}(s-),z)\tilde{N}(ds,dz),\ \
     \mathbb{P}\text{-a.s.}.$$
Similarly, one can show that for $t\geq0$,
 \begin{align*}
      \mathcal{A}\int_0^te^{(t-s)\mathcal{A}}F(s,\mathfrak{u}(s))\,ds=\int_0^t\mathcal{A}e^{(t-s)\mathcal{A}}F(s,\mathfrak{u}(s))\,ds,
     \ \mathbb{P}\text{-a.s.}.
 \end{align*}

On the other hand, we have, for every $0<T<\infty$,
\begin{align*}
    \E\int_0^T\int_0^t \|\mathcal{A}e^{(t-s)\mathcal{A}}F(s,\mathfrak{u}(s))\|^2_{\mathcal{H}}dsdt\leq \E\int_0^T\int_0^t \|F(s,\mathfrak{u}(s))\|^2_{\mathcal{D}(\mathcal{A})}dsdt<\infty.
\end{align*}
It follows that for every $t\in[0,T]$,
\begin{align*}
   \int_0^T\int_0^t \|\mathcal{A}e^{(t-s)\mathcal{A}}F(s,\mathfrak{u}(s))\|^2_{\mathcal{H}}dsdt<\infty,\ \mathbb{P}\text{-a.s.}
\end{align*}
Similarly, we also find out that for every $0<t<T<\infty$, $\mathbb{P}\text{-a.s.}$
\begin{align*}
     \E\int_0^T\int_0^t\int_Z \|\mathcal{A}e^{(t-s)\mathcal{A}}&G(s,\mathfrak{u}(s),z)\|^2_{\mathcal{H}}\nu(dz)\,dsdt\\
     &\leq \E\int_0^T\int_0^t\int_Z \|G(s,\mathfrak{u}(s),z)\|^2_{\mathcal{D}(\mathcal{A})}\nu(dz)\,dsdt<\infty.
\end{align*}
Now one can apply the general Fubini's Theorem and the stochastic
Fubini's theorem to obtain for every $0<s<t<\infty$
\begin{align}\label{eq-2-proof}
    \int_0^t\int_0^s\mathcal{A}&e^{(s-r)\mathcal{A}}F(r,\mathfrak{u}(r))drds\nonumber\\
      &=\int_0^t\int_r^t\mathcal{A}e^{(s-r)\mathcal{A}}F(r,\mathfrak{u}(r))\,dsdr\nonumber
        =\int_0^t\left(e^{(t-r)\mathcal{A}}-I\right)F(r,\mathfrak{u}(r))dr\nonumber\\
        &=\int_0^te^{(t-r)\mathcal{A}}F(r,\mathfrak{u}(r))dr-\int_0^tF(r,\mathfrak{u}(r))dr\,,
\end{align}
and
\begin{align}\label{eq-3-proof}
      \int_0^t\int_0^s\int_Z\mathcal{A}&e^{(s-r)\mathcal{A}}G(r,\mathfrak{u}(r-),z)\tilde{N}(dr,dz)\,ds\nonumber\\
      &=\int_0^t\int_Z\left(\int_r^t\mathcal{A}e^{(s-r)\mathcal{A}}G(r,\mathfrak{u}(r-),z)\,ds\right)\tilde{N}(dr,dz)\nonumber\\
        &= \int_0^t\int_Z\left(\int_r^t\mathcal{A}e^{(s-r)\mathcal{A}}ds\right)G(r,\mathfrak{u}(r-),z)\tilde{N}(dr,dz)\nonumber\\
        &=\int_0^t\int_Z\left(e^{(t-r)\mathcal{A}}-I\right)G(r,\mathfrak{u}(r-),z)\tilde{N}(dr,dz)\nonumber\\
        &=\int_0^t\int_Ze^{(t-r)\mathcal{A}}G(r,\mathfrak{u}(r-),z)\tilde{N}(dr,dz)-\int_0^t\int_ZG(r,\mathfrak{u}(r-),z)\tilde{N}(dr,dz)\,.
\end{align}
In the above we used the fact that since the semigroup $e^{t\mathcal{A}}$,
$t\geq0$ is strongly continuous, $t\mapsto e^{t\mathcal{A}}x$ is
differentiable for every $x\in\mathcal{D}(\mathcal{A})$. From what we have proved
in the preceding part, we know that Problem \eqref{SDE1} has a
unique mild solution which satisfies
\begin{align*}
      \mathfrak{u}(t)=e^{t\mathcal{A}}\mathfrak{u}_0+\int_0^te^{(t-s)\mathcal{A}}F(s,\mathfrak{u}(s))\,ds+\int_0^t\int_Ze^{(t-s)\mathcal{A}}G(s,\mathfrak{u}(s-),z)\tilde{N}(ds,dz)\ \ \mathbb{P}\text{-a.s.}\ \ t\geq0.
\end{align*}
Hence first by \eqref{eq-1-proof} we conclude that $\mathcal{A}\mathfrak{u}$ is integrable
$\mathbb{P}$-a.s. and then by using \eqref{eq-2-proof} and \eqref{eq-3-proof} we obtain
\begin{align*}
      \int_0^t\mathcal{A}\mathfrak{u}(s)\,ds&=\int_0^t\mathcal{A}e^{t\mathcal{A}}\mathfrak{u}_0
      +\int_0^t\int_0^s\mathcal{A}e^{(s-r)\mathcal{A}}F(r,\mathfrak{u}(r))dr\\
      &\hspace{1cm}+\int_0^t\int_0^s\int_Ze^{(s-r)\mathcal{A}}G(r,\mathfrak{u}(r-),z)\tilde{N}(dr,dz)\\
      &=e^{t\mathcal{A}}\mathfrak{u}_0-\mathfrak{u}_0+\int_0^te^{(t-r)\mathcal{A}}F(r,\mathfrak{u}(r))dr-\int_0^tF(r,\mathfrak{u}(r))dr\\
      &\hspace{1cm}+\int_0^t\int_Ze^{(t-r)\mathcal{A}}G(r,\mathfrak{u}(r),z)\tilde{N}(dr,dz)-\int_0^t\int_ZG(r,\mathfrak{u}(r-),z)\tilde{N}(dr,dz)\\
      &=\mathfrak{u}(t)-\mathfrak{u}_0-\int_0^tF(r,\mathfrak{u}(r))dr-\int_0^t\int_ZG(r,\mathfrak{u}(r-),z)\tilde{N}(dr,dz)
\end{align*}
which shows that the mild solution is also a strong solution.\\
Conversely, let $\mathfrak{u}$ be a strong solution. By
making use of the It\^{o} formula, see \cite{[Zhu_2010]}, to the function
$\psi(s,y)=e^{(t-s)\mathcal{A}}y$ and process
$\mathfrak{u}_{\lambda}(s)=R(\lambda,\mathcal{A})\mathfrak{u}(s)$,
where $R(\lambda,\mathcal{A})$ is the resolvent of $\mathcal{A}$, we
infer for every $t\geq 0$
\begin{align*}
     e^{(t-s)\mathcal{A}}R(\lambda,\mathcal{A})\mathfrak{u}(s)&-R(\lambda,\mathcal{A})\mathfrak{u}_0\\
     &=
     -\int_0^te^{(t-s)\mathcal{A}}\mathcal{A}R(\lambda,\mathcal{A})\mathfrak{u}(s)\,ds
     +\int_0^te^{(t-s)\mathcal{A}}R(\lambda,\mathcal{A})\mathcal{A}\mathfrak{u}(s)\,ds\\
     &+\int_0^te^{(t-s)\mathcal{A}}R(\lambda,\mathcal{A})F(s,\mathfrak{u}(s))\,ds\\
     &+\int_0^t\int_Ze^{(t-s)\mathcal{A}}R(\lambda,\mathcal{A})G(s,\mathfrak{u}(s-),z)\tilde{N}(ds,dz),\ \mathbb{P}\text{-a.s.}
\end{align*}
It follows that for every $t\geq 0$
\begin{align*}
     R(\lambda,\mathcal{A})e^{(t-s)\mathcal{A}}\mathfrak{u}(s)=R(\lambda,\mathcal{A})\Big{(}\mathfrak{u}_0
     &+\int_0^te^{(t-s)\mathcal{A}}F(s,\mathfrak{u}(s))\,ds\\
     &+\int_0^t\int_Ze^{(t-s)\mathcal{A}}G(s,\mathfrak{u}(s-),z)\tilde{N}(ds,dz)\Big{)}\ \mathbb{P}\text{-a.s.}
\end{align*}
Hence we have for every $t\geq 0$,
\begin{align*}
      \mathfrak{u}(t)=e^{t\mathcal{A}}\mathfrak{u}_0+\int_0^te^{t-s\mathcal{A}}F(s,\mathfrak{u}(s))\,ds+\int_0^t\int_Ze^{t-s\mathcal{A}}G(s,\mathfrak{u}(s-),z)\tilde{N}(ds,dz),\
      \ \mathbb{P}\text{-a.s.}.
\end{align*}
Thus, we infer that $\mathfrak{u}$ is of the form \eqref{strong
solution}. Furthermore, the stochastic equivalence becomes $\mathbb{P}$-equivalence in view of the c\`{a}dl\`{a}g property of the strong solution and the mild  solution. Therefore, mild solution and strong solution are
$\mathbb{P}$-equivalent or in other word, uniqueness of strong solution holds.

\end{proof}

%----------------------------------------------------
In the proof of Theorem \ref{theo: exsitence and uniqueness of mild
solution} we will need the following Lemma.
\begin{lem}\label{lem: globally cont.}If a function $h:\mathcal{H}\rightarrow\mathcal{H}$ is locally
Lipschitz on a closed ball $B(0,R)\subset \mathcal{H}$, then the
function $\tilde{h}:\mathcal{H}\rightarrow\mathcal{H}$ defined by
\begin{eqnarray*}
    \tilde{h}(x):=\left\{
           \begin{array}{cc}
              h(x),& \text{if  } \|x\|_{\mathcal{H}}\leq R,\\
              h(\frac{Rx}{\|x\|_{\mathcal{H}}}),& \text{otherwise}.
           \end{array}
           \right.
\end{eqnarray*}
is globally Lipschitz.
\end{lem}

\begin{proof}[Proof of Theorem \ref{theo: existence and uniqueness of locally mild solution}]
Set
$\tilde{f}(x)=-f(t,\omega,x_1,x_2)-m(\|B^{\frac{1}{2}}x_1\|^2)Bx_1$.
Since $\tilde{f}$ is locally Lipschitz continuous,  for every
$n\in\mathbb{N}$ we may define the following mapping
  \begin{align*}
     \tilde{f}_n(x)=\left\{\begin{array}{cc} \tilde{f}(x) & \text{if }\|x\|_{\mathcal{H}}\leq n  \\ \tilde{f}(\frac{nx}{\|x\|_{\mathcal{H}}}) & \text{if }\|x\|_{\mathcal{H}}>n, \end{array}\right.
  \end{align*}
where $x\in\mathcal{H}$. Then $\tilde{f}_n$ is globally Lipschitz
continuous by Lemma \ref{lem: globally cont.}. Set
$F_n(x)=\left(0,\tilde{f}_n(x)\right)^{\top}$ for every
$x\in\mathcal{H}$. Therefore, by Theorem \ref{theo: exsitence and
uniqueness of mild solution} for every $n\in\mathbb{N}$ there exists
a unique mild solution $(X_n(t))_{t\geq 0}$ to Problem \eqref{SDE1}
with $F$ substituted for $F_n$ which is given by
\begin{align}\label{eq2}
  X_n(t)=e^{t\mathcal{A}}\mathfrak{u}_0+\int_0^te^{(t-s)\mathcal{A}}F_n(s,X_n(s))\,ds
  +\int_0^{t}\int_Ze^{(t-s)\mathcal{A}}G(s,X_n(s),z)\tilde{N}(ds,dz),\ t\geq0.
\end{align}
Define a sequence of stopping times $\{\tau_n\}_{n=1}^{\infty}$ by
\begin{align*}
            \tau_n:=  \inf\{t\geq 0:\ \|X_n(t)\|_{\mathcal{H}}> n\}.
\end{align*}
By the c\`{a}dl\`{a}g property of the solution $X_n$, we know that $\tau_n$ is indeed a stopping time.
First let us note that for every $n<m$, we have $F_n(x)=F_m(x)=F(x)$ for
all $\|x\|_{\mathcal{H}}\leq n$. Since $\|X_n(t)\|_{\mathcal{H}}\leq
n$ for all $t<\tau_n$, so by \eqref{eq2} we obtain that on $[0,\tau_n)$
\begin{align}\label{eq-330}
   X_n(t)&=e^{t\mathcal{A}}\mathfrak{u}_0+\int_0^{t}e^{(t-s)\mathcal{A}}F_n(s,X_n(s))\,ds +\int_0^{t}\int_Ze^{(t-s)\mathcal{A}}G(s,X_n(s-),z)\tilde{N}(ds,dz)\nonumber\\
   &=e^{t\mathcal{A}}\mathfrak{u}_0+\int_0^{t}e^{(t-s)\mathcal{A}}F(s,X_n(s))\,ds +\int_0^{t}\int_Ze^{(t-s)\mathcal{A}}G(s,X_n(s-),z)\tilde{N}(ds,dz)
   .
\end{align}
Set
\begin{align*}
\Phi(X_n):=e^{t\mathcal{A}}\mathfrak{u}_0+\int_0^{t}e^{(t-s)\mathcal{A}}F(s,X_n(s))\,ds +\int_0^{t}\int_Ze^{(t-s)\mathcal{A}}G(s,X_n(s-),z)\tilde{N}(ds,dz).
\end{align*}
Note that
\begin{align*}
    \triangle \Phi(X_n)(\tau_{n})=\int_Z G(\tau_{n},X_n(\tau_{n}-),z)\tilde{N}(\{\tau_n\},dz).
\end{align*}
which means that the value of $\Phi(X_n)$ at $\tau_n$ depends only on the values of $X_n$ on $[0,\tau_n)$. Hence we may extend the solution $X_n$ on $[0,\tau_n)$ to $X_n$ on $[0,\tau_n]$ by setting, see Appendix,
\begin{align}\label{eq-250}
      X_n(\tau_n)=\Phi(X_n)(\tau_n)=e^{\tau_n\mathcal{A}}\mathfrak{u}_0+\int_0^{\tau_n}e^{(\tau_n-s)\mathcal{A}}F(s,X_n(s))\,ds+I_{\tau_n}(G(X_n))(\tau_n)
\end{align}

where
\begin{align*}
    I_{\tau_n}(G(X_n))(t)=\int_0^{t}\int_Z1_{[0,\tau_n]}e^{(t-s)\mathcal{A}}G(s,X_n(s-),z)\tilde{N}(ds,dz),\ t\geq0.
\end{align*}
In such a case, combining \eqref{eq-330} together with \eqref{eq-250}, we deduce that the stopped process $X(\cdot\wedge\tau_n)$ satisfies
\begin{align}\label{eq-332}
   X_n(t\wedge\tau_n)
   =e^{(t\wedge\tau_n)\mathcal{A}}\mathfrak{u}_0+\int_0^{t\wedge\tau_n}e^{(t\wedge\tau_n-s)\mathcal{A}}F(s,X_n(s))\,ds+I_{\tau_n}(G(X_n))(t\wedge\tau_n),\ t\geq0.
\end{align}
In a similar way, we have
\begin{align*}
   X_m(t\wedge\tau_m)
   &=e^{(t\wedge\tau_n)\mathcal{A}}\mathfrak{u}_0+\int_0^{t\wedge\tau_n}e^{(t\wedge\tau_m-s)\mathcal{A}}F(s,X_m(s))\,ds+I_{\tau_n}(G(X_m))(t\wedge\tau_n),\ t\geq0.
\end{align*}

Set $\tau_{n,m}=\tau_n\wedge\tau_m$. It follows that $\|X_n(t)\|\leq n<m$ and $\|X_m(t)\|\leq m$ for $t\in[0,\tau_{n,m})$. So $F_n(s,X_n(s))=F(s,X_n(s))$ and $F_m(s,X_m(s))=F(s,X_m(s))$. Therefore, $X_n$ and $X_m$ both satisfy the same Equation
\begin{align*}
       X(t)&=e^{t\mathcal{A}}\mathfrak{u}_0+\int_0^{t}e^{(t-s)\mathcal{A}}F(s,X(s))\,ds +\int_0^{t}\int_Ze^{(t-s)\mathcal{A}}G(s,X(s-),z)\tilde{N}(ds,dz),\ \text{on }[0,\tau_{n,m}).
\end{align*}
Hence by the uniqueness of mild solution proved in the theorem \ref{theo: exsitence and uniqueness of mild solution}, we have
\begin{align*}
        X_n(t)=X_m(t),\ on\ [0,\tau_{n,m})\ a.s.
\end{align*}
Since \begin{align*}
    \triangle X_n(\tau_{n,m})=\int_Z G(\tau_{n,m},X_n(\tau_{n,m}-),z)\tilde{N}(\{\tau_n\},dz),
\end{align*}
and the Remark \ref{rem-main} tells us that $G(s,X_n,z)$ and $G(s,X_m,z)$ coincide on $[0,\tau_{n,m}]$, we infer $$X_n=X_m\text{ on }[0,\tau_{n,m}].$$
It follows that
$$\tau_n\leq \tau_m\ \text{if}\ n<m. $$ We will show this assertion by contradiction. Let us fix $n<m$. Suppose that $\mathbb{P}(\tau_n>\tau_m)>0$. Set $A=\{\tau_n>\tau_m\}$. By the definition of the stopping time $\tau_n$, we have $\|X_n(t)\|_{\mathcal{H}}\leq n$ for $t\in[0,\tau_n)$ and $\|X_m(\tau_m)\|_{\mathcal{H}}\geq m>n $. Since $X_n$ coincides with $X_m$ on $[0,\tau_{n,m}]$, we find $\|X_n(\tau_m)\|_{\mathcal{H}}=\|X_m(\tau_m)\|_{\mathcal{H}}>n$ on $A$ which would contradict the fact that $\|X_n(t)\|_{\mathcal{H}}\leq n$ for $t\in[0,\tau_n)$. Therefore, we conclude that $\tau_n\leq \tau_m$ a.s. for $n<m$. This means that
$\{\tau_n\}_{n=1}^{\infty}$ is an increasing sequence. So the limit
$\lim_{n\rightarrow\infty}\tau_n$ exists a.s. Let us denote this limit by
$\tau_{\infty}$. Let
$\Omega_0=\{\omega:\lim_{n\rightarrow\infty}\tau_n=\tau_{\infty}\}$.
Note that $\mathbb{P}(\Omega_0)=1$.

 Now define a local process
$(X_t)_{0\leq t<\tau_{\infty}}$ as follows. If
$\omega\notin\Omega_0$, set $X(t,\omega)=0$ for all $0\leq t<\tau_{\infty}$.
If $\omega\in\Omega_0$, then there exists a number $n\in\mathbb{N}$
such that $t\leq \tau_n(\omega)$ and we
set $X(t,\omega)=X_n(t,\omega)$. The process is well defined since
$X_n(t)$ exists uniquely on $\{t\leq \tau_n\}$.
 Indeed, for every $t\in\mathbb{R}_+$ by \eqref{eq-332} we have
\begin{align*}
  X_n(t\wedge\tau_n)=e^{(t\wedge\tau_n)\mathcal{A}}\mathfrak{u}_0+\int_0^{t\wedge\tau_n}e^{(t\wedge\tau_n-s)\mathcal{A}}F(s,X_n(s))\,ds+I_{\tau_n}(G(X_n))(t\wedge\tau_n)
\end{align*}
Since $X(t)=X_n(t)$ for $t\leq\tau_n$, we infer that
\begin{align*}
        X(t\wedge\tau_n)=e^{(t\wedge\tau_n)\mathcal{A}}\mathfrak{u}_0+\int_0^{t\wedge\tau_n}e^{(t\wedge\tau_n-s)\mathcal{A}}F(s,X(s))\,ds+I_{\tau_n}(G(X))(t\wedge\tau_n)
\end{align*}
where we used the fact that for all $t\geq0$,
\begin{align*}
  I_{\tau_n}(G(X_n))(t)&=\int_0^{t}\int_Z1_{[0,\tau_n]}(s)e^{(t-s)\mathcal{A}}G(s,X_n(s-),z)\tilde{N}(ds,dz)\\
                       &=\int_0^{t}\int_Z1_{[0,\tau_n]}(s)e^{(t-s)\mathcal{A}}G(s\wedge\tau_n,X_n(s\wedge\tau_n-),z)\tilde{N}(ds,dz)\\
                       &=\int_0^{t}\int_Z1_{[0,\tau_n]}(s)e^{(t-s)\mathcal{A}}G(s\wedge\tau_n,X(s\wedge\tau_n-),z)\tilde{N}(ds,dz)\\
                       &=I_{\tau_n}(G(X))(t).
\end{align*}
Furthermore, by the definition of the sequence
$\{\tau_n\}_{n=1}^{\infty}$ we obtain
\begin{align}\label{eq-334}
    \lim_{t\nearrow\tau_{\infty}(\omega)}\|X(t,\omega)\|_{\mathcal{H}}=\lim_n \|X(\tau_n(\omega),\omega)\|_{\mathcal{H}}\geq \lim_{n}n=\infty\ \ \text{a.s.}.
\end{align}
To show that the process $X(t)$, $0\leq t<\tau_{\infty}$ is a maximal local mild solution to Problem \eqref{SDE1}. Let us suppose that $\tilde{X}=(\tilde{X}(t))_{0\leq t<\tilde{\tau}}$ is another local mild solution to Problem \eqref{SDE1}
   such that $\tilde{\tau}\geq \tau_{\infty}$ a.s. and $\tilde{X}|_{[0,\tau_{\infty})\times\Omega}\sim
   X$.
It follows from \eqref{eq-334} and the $\mathbb{P}$-equivalence of $X$ and $\tilde{X}$ on $[0,\tau_{\infty})$ that
\begin{align}\label{eq-336}
   \lim _{t\nearrow\tau_{\infty}(\omega)}\|\tilde{X}(t,\omega)\|_{\mathcal{H}}= \lim_{t\nearrow\tau_{\infty}(\omega)}\|X(t,\omega)\|_{\mathcal{H}}=\infty.
\end{align}
In order to get the maximality of $X$, we need to show that $\mathbb{P}(\tilde{\tau}>\tau_{\infty})=0$. To prove this, assume the contrary, namely $\mathbb{P}(\tilde{\tau}>\tau_{\infty})>0$. Since $\tilde{X}$ is a local mild solution, there exists a sequence $\{\tilde{\tau}_n\}$ of increasing stopping times such that $\tilde{X}$ is a mild solution on the interval $[0,\tilde{\tau}_n]$, i.e. the equation \eqref{locally mild solution} is satisfied. Define a new family of stopping times by
\begin{align*}
    & \sigma_{n,k}:=\tilde{\tau}_n\wedge\inf\{t:\|\tilde{X}(t)\|>k\};\\
     &\sigma_k:=\sup\sigma_{n,k}.
\end{align*}
Since $\sigma_{n,k}\leq \tilde{\tau}_n$, $\sigma_k\leq \tilde{\tau}_n$. Also, observe that $\lim_{k}\sigma_k=\tilde{\tau}$. Since $\sigma_k\nearrow\tilde{\tau}$ and $\mathbb{P}(\tilde{\tau}>\tau_{\infty})>0$, there exists a number $k$ such that $\mathbb{P}(\sigma_k>\tau_{\infty})>0$.
Hence, we have $\|\tilde{X}(t,\omega)\|_{\mathcal{H}}\leq k$ for $t\in[\tau_{\infty}(\omega),\sigma_k(\omega))$ contradicting the earlier observation \eqref{eq-336}.

Now we continue to show the uniqueness of the solution. Actually, the uniqueness of the solution has already shown in above construction of solution $X$. Alternatively, we may prove it in another way. Let $X$ and $Y$ be two mild solution to Problem \eqref{SDE1} on the stochastic intervals $[0,\tau]$ and $[0,\sigma]$. respectively. We shall show that $X=Y$, $\mathbb{P}$-a.s. on $[0,\tau\wedge\sigma]$.\\
  For each $n\in\mathbb{N}$, define
\begin{align*}
   \sigma_n=\inf\{t\geq0:\|Y_n(t)\|_{\mathcal{H}}>n \text{ or } \|X(t)\|_{\mathcal{H}}>n\}\wedge\tau\wedge\sigma\wedge n.
\end{align*}
Then $\|Y(t)\|_{\mathcal{H}}\leq n$ and
$\|X(t)\|_{\mathcal{H}}\leq n$ on $[0,\sigma_n)$. Further, we find out that $\lim_{n\rightarrow\infty}\mathbb{P}(\sigma_n<\sigma\wedge\tau)=0$. Hence we only need to verify that $X=Y$ on $[0,\sigma_n]$, $\mathbb{P}$-a.s.
Since  $X(t)$, $t\in[0,\tau]$ and $Y(t)$, $t\in[0,\sigma]$ are both mild solutions to Problem
\eqref{SDE1}, we infer for that
\begin{align*}
    X(t)&=e^{t\mathcal{A}}\mathfrak{u}_0+\int_0^te^{(t-s)\mathcal{A}}F(s,X(s))\,ds\\
    &\hspace{2cm}+\int_0^{t}\int_Ze^{(t-s)\mathcal{A}}G(s,X(s-),z)\tilde{N}(ds,dz)\ \text{on } [0,\sigma_n)\ \mathbb{P}\text{-a.s.}\\
    Y(t)&=e^{t\mathcal{A}}\mathfrak{u}_0+\int_0^te^{(t-s)\mathcal{A}}F(s,Y(s))\,ds\\
    &\hspace{2cm}+\int_0^{t}\int_Ze^{(t-s)\mathcal{A}}G(s,Y(s-),z)\tilde{N}(ds,dz)\  \text{on }[0,\sigma_n)\ \mathbb{P}\text{-a.s.}.
\end{align*}

Therefore, by using the Cauchy-Schwarz and Burkholder-Davis inequalities, see \cite{[Brzezniak_Hau_Zhu]},
that
\begin{align*}
   &\E\left(\sup_{0\leq s<\sigma_n}\Big{\|}X(s)-Y(s)\Big{\|}^2_{\mathcal{H}}\right)\\
   &\hspace{1cm}\leq 2\E\left(\sup_{0\leq s<\sigma_n}\left\|\int_0^{s}e^{(s-r)\mathcal{A}}\Big{(}F(X(r))-F(Y(r))\Big{)}dr
   \right\|^2_{\mathcal{H}}\right)\\
   &\hspace{1cm}+2\E\left(\sup_{0\leq s<\sigma_n}\left\|\int_0^s\int_Z e^{(s-r)\mathcal{A}}\Big{(}G(r,X(r-),z)-G(r,Y(r-),z)\Big{)}\tilde{N}(dr,dz)\right\|^2_{\mathcal{H}}\right)\\
   &\hspace{1cm}\leq2n L_n^2\E\int_0^{\sigma_n}\|X(s)-Y(s)\|^2_{\mathcal{H}}ds+2CL_g\E\int_0^{\sigma_n}\|X(s)-Y(s)\|^2_{\mathcal{H}}ds\\
   &\hspace{1cm}\leq C(n)\E\int_0^{t\wedge\sigma_n}\sup_{0\leq u\leq \sigma_n}\|X(u)-Y(u)\|^2_{\mathcal{H}}ds,
\end{align*}
where $C(n)=2(nL_n+CL_g)$. By applying the Gronwall Lemma, we obtain
for every $t\geq0$,
\begin{align*}
    \E\left(\sup_{0\leq s<\sigma_n}\Big{\|}X(s)-Y(s)\Big{\|}^2_{\mathcal{H}}\right)=0.
\end{align*}
This implies that for every  $X|_{[0,\sigma_n)}$ and
$Y|_{[0,\sigma_n}$ are indistinguishable. By Remark \ref{rem-main}, we infer that $X=Y$ on $[0,\sigma_n]$ $\mathbb{P}$-a.s.
\end{proof}

%%%%%%%%%%%%%%%%%%%%%%%%%%%%%%%%%%%%-----------------------------------------------------------------
We are now in a position to prove the main result. To prove this, we
need another certain auxiliary Lemma known as Khas'minskii's test.
\begin{lem}\label{lem: Khasminskii test}(Khas'minskii's test for nonexplosions) Let $\mathfrak{u}(t)$, $0\leq t<\tau_{\infty}$ be a maximal local mild solution to Equation \eqref{SDE1} with an approximating sequence $\{\tau_n\}_{n\in\mathbb{N}}$. Suppose that there exists a function $V:\mathcal{H}\rightarrow\mathbb{R}$ such that
\begin{enumerate}
   \item[1.] $V\geq0$ on $\mathcal{H}$,
   \item[2.] $q_R=\inf_{\|x\|_{\mathcal{H}}\geq R}V(x)\rightarrow +\infty$
   \item[3.] $\E V(\mathfrak{u}(t\wedge\tau_n))\leq\E V(\mathfrak{u}_0)+C\int_0^t\Big{(}1+\E(V(\mathfrak{u}(s\wedge\tau_n)))\Big{)}ds$ for each $n\in\mathbb{N}$,
   \item[4.] $\E V(\mathfrak{u}_0)<\infty$.
\end{enumerate}
Then $\tau_{\infty}=+\infty$ $\mathbb{P}$-a.s. We call $V$ a
Lyapunov function for \eqref{SDE1}.
\end{lem}
\begin{proof}
   Since by assumption we know
   \begin{align*}
         \E V(\mathfrak{u}(t\wedge\tau_n))\leq\E V(\mathfrak{u}_0)+C\int_0^t\Big{(}1+\E(V(\mathfrak{u}(s\wedge\tau_n)))\Big{)}ds
   \end{align*}
   So
   \begin{align*}
         1+\E V(\mathfrak{u}(t\wedge\tau_n))\leq1+\E V(\mathfrak{u}_0+C\int_0^t\Big{(}1+\E(V(\mathfrak{u}(s\wedge\tau_n)))\Big{)}ds,
   \end{align*}
    then apply the Gronwall Lemma to obtain
    \begin{align*}
      1+\E V(\mathfrak{u}(t\wedge\tau_n))\leq \Big{(}1+\E V(\mathfrak{u}_0)\Big{)}e^{Ct}.
    \end{align*}
    Hence we have for each $n\in\mathbb{N}$
    \begin{align*}
          \E V(\mathfrak{u}(t\wedge\tau_n))\leq \Big{(}1+\E V(\mathfrak{u}_0)\Big{)}e^{Ct}-1,\ \ t\geq0.
    \end{align*}
    It then follows that
    \begin{align*}
         \mathbb{P}(\{\tau_n<t\})&=\E1_{\{\tau_n<t\}}=\int_{\Omega}1_{\{\tau_n<t\}}d\mathbb{P}=\int_{\Omega}\frac{q_n}{q_n}1_{\{\tau_n<t\}}d\mathbb{P}\leq\frac{1}{q_n}\int_{\Omega}V(\mathfrak{u}(t\wedge\tau_n))1_{\{\tau_n<t\}}d\mathbb{P}\\
                                 &\leq\frac{1}{q_n}\int_{\Omega}V(\mathfrak{u}(t\wedge\tau_n))d\mathbb{P}=\frac{1}{q_n}\E V(\mathfrak{u}(t\wedge\tau_n))\leq\frac{1}{q_n}\left[\Big{(}1+\E V(\mathfrak{u}_0)\Big{)}e^{Ct}-1\right].
    \end{align*}
    Since $\E V(\mathfrak{u}_0)<\infty$ and $q_n\rightarrow\infty$ as $n\rightarrow\infty$, so $\lim_{n\rightarrow\infty}\mathbb{P}(\{\tau_n<t\})=0$. Since $\tau_n$ is an increasing stopping time, the set $\{\tau_n<t\}$ is decreasing. Thus we infer that for every $t\geq0$,
    \begin{align*}
     \mathbb{P}(\{\tau_{\infty}<t\})=\mathbb{P}(\{\lim_{n\rightarrow\infty}\tau_n<t\})
     =\mathbb{P}\left(\bigcap_{n\in\mathbb{N}}\{\tau_n<t\}\right)=\lim_{n\rightarrow\infty}\mathbb{P}(\{\tau_n<t\})=0
    \end{align*}
    Hence $\tau_{\infty}=+\infty$, $\mathbb{P}$-a.s..
\end{proof}
%-----------------------
\begin{proof}[Proof of Theorem \ref{theo: lifespan}]
Let $\mathfrak{u}(t)$, $0\leq t<\tau_{\infty}$, be a maximal local
mild solution to problem \eqref{SDE1}. Define a sequence of stopping
times by
\begin{align*}
     \tau_n=\inf\{t\geq0:\|\mathfrak{u}(t)\|_{\mathcal{H}}\geq n\},\ \ n\in\mathbb{N}.
\end{align*}
Then in the proof of Theorem \ref{theo: existence and uniqueness of locally mild solution}, we showed that
$\{\tau_n\}_{n\in\mathbb{N}}$ is an approximating sequence of the
accessible stopping time $\tau_{\infty}$. In order to apply the 
Khas'minskii's test, we need to find a Lyapunov function.
Define a function $V:\mathcal{H}\rightarrow\mathbb{R}^+$ by
\begin{align*}
     V(x)=\frac{1}{2}\|x\|_{\mathcal{H}}^2+\frac{1}{2}M(\|B^{\frac{1}{2}}x_1\|_H^2),
\end{align*}
where $x=(x_1,x_2)^{\top}\in\mathcal{H}$ and $M(s)=\int_0^sm(r)dr$,
$s\geq0$. It is clear that for every $x\in\mathcal{H}$,
\begin{align*}
    V(x)\geq0.
\end{align*}
Observe that
\begin{align*}
    q_R&=\inf_{\|x\|_{\mathcal{H}}\geq R}V(x)=\frac{1}{2}\inf_{\|x\|_{\mathcal{H}}\geq R}\|x\|_{\mathcal{H}}^2+\frac{1}{2}\inf_{\|x\|_{\mathcal{H}}\geq R}M(\|B^{\frac{1}{2}}x_1\|^2)\\
    &=\frac{1}{2}R^2+\frac{1}{2}\inf_{\|x\|_{\mathcal{H}}\geq R}M(\|B^{\frac{1}{2}}x_1\|^2)\\
    &=\frac{1}{2}R^2+\frac{1}{2}\inf_{\|x\|_{\mathcal{H}}\geq R}\int_0^{\|B^{\frac{1}{2}}x_1\|^2}m(r)dr.
\end{align*}
Taking the limit in this equality as $R\rightarrow\infty$, we obtain
that $q_R\rightarrow+\infty$.
Meanwhile, we have
\begin{align*}
      \E(V(\mathfrak{u}_0))=\frac{1}{2}\E\|\mathfrak{u}_0\|^2_{\mathcal{H}}+\frac{1}{2}\E M(\|B^{\frac{1}{2}}\mathfrak{u}_0\|_H)<\infty.
\end{align*}
Thus conditions 1,2,4 in the definition of Lyapunov function are
satisfied. It remains to verify condition 3 from Lemma \ref{lem:
Khasminskii test}, namely,
\begin{align}\label{eq350}
\E V(\mathfrak{u}(t\wedge\tau_n))\leq\E
V(\mathfrak{u}_0)+C\int_0^t\Big{(}1+\E(V(\mathfrak{u}(s\wedge\tau_n)))\Big{)}ds,\
t\geq0.
\end{align}
The idea is to prove \eqref{eq350} first for global strong solution and then extend to the case when $\mathfrak{u}$ is a local mild solution.

\textbf{\textit{Step 1: Inequality \eqref{eq350} holds for global
strong solutions.}} Suppose that $\mathfrak{u}$ is a global strong
solution to Problem \eqref{SDE1} satisfying
\begin{align*}
    \mathfrak{u}(t)=\mathfrak{u}_0&+\int_0^t\Big{[}\mathcal{A}\mathfrak{u}(s)+F(s,u(s),u_t(s))\Big{]}ds\\
    &+\int_0^{t}\int_ZG(s,\mathfrak{u}(s-),z)\tilde{N}(ds,dz),\ \mathbb{P}\text{-a.s. }
    t\geq0.
\end{align*}
Applying the It\^{o} formula, see \cite{[Zhu_2010]}, to the process
$\mathfrak{u}(\cdot\wedge\tau_n)$ and function
$V(x)=\frac{1}{2}\|x\|^2_{\mathcal{H}}+\frac{1}{2}M(\|B^{\frac{1}{2}}x_1\|_H^2)$,
we obtain for $t\geq0$,
\begin{align}\label{eq4}
       V(\mathfrak{u}(t\wedge\tau_n))-V(\mathfrak{u}_0)=&\int_0^{t\wedge\tau_n}\langle DV(\mathfrak{u}(s),\mathcal{A}\mathfrak{u}(s)+F(s,\mathfrak{u}(s)))\rangle_{\mathcal{H}}ds\nonumber\\
       &+\int_0^{t\wedge\tau_n}\int_Z\Big{[}V(\mathfrak{u}(s)+G(s,\mathfrak{u}(s),z))-V(\mathfrak{u}(s))\\
       &\hspace{2cm}-\langle DV(\mathfrak{u}(s)),G(s,\mathfrak{u}(s),z)\rangle\Big{]}\nu(dz)\,ds\nonumber\\
       &+\int_0^{t\wedge\tau_n}\int_Z\Big{[}V\big{(}\mathfrak{u}(s-)+G(s,\mathfrak{u}(s-),z)\big{)}-V(\mathfrak{u}(s-))\Big{]}\tilde{N}(ds,dz).\nonumber
\end{align}
Note that for any $x=(x_1,x_2)^{\top}$ and $h=(h_1,h_2)^{\top}$,
\begin{align*}
       DV(x)h&=\langle x,h\rangle_{\mathcal{H}}+m(\|B^{\frac{1}{2}x_1}\|^2)\langle B^{\frac{1}{2}}x_1,B^{\frac{1}{2}}h_1\rangle\\
       &=\langle x,h\rangle_{\mathcal{H}}+m(\|B^{\frac{1}{2}}x_1\|^2_H)\langle Bx_1,h_1\rangle\\
       &=\langle x,h\rangle_{\mathcal{H}}+m(\|B^{\frac{1}{2}}x_1\|^2_H)\langle AA^{-2}Bx_1,Ah_1\rangle\\
       &=\langle x,h\rangle_{\mathcal{H}}+m(\|B^{\frac{1}{2}}x_1\|^2_H)\langle \left( \begin{array}{c} A^{-2}Bx_1 \\ 0 \end{array}\right),\left(\begin{array}{c} h_1 \\ h_2 \end{array}\right) \rangle_{\mathcal{H}}.
\end{align*}
Hence for any $x=(x_1,x_2)^{\top}\in\mathcal{H}$,
\begin{align*}
        DV(x)=x+m(\|B^{\frac{1}{2}}x_1\|_H^2)\left(\begin{array}{c} A^{-2}Bx_1 \\ 0 \end{array}\right).
\end{align*}
It follows that for $x\in\mathcal{D}(\mathcal{A})$,
\begin{align*}
   &\langle DV(x),\mathcal{A}x\rangle_{\mathcal{H}}=\langle x,\mathcal{A}x\rangle_{\mathcal{H}}+m(\|B^{\frac{1}{2}}x_1\|^2_H)\langle \left(\begin{array}{c} A^{-2}Bx_1 \\ 0 \end{array}\right),\mathcal{A}x\rangle_{\mathcal{H}}\\
   &=\langle \left(\begin{array}{c} x_1 \\ x_2 \end{array}\right),\left(\begin{array}{c} x_2 \\ -A^2x_1 \end{array}\right)\rangle_{\mathcal{H}}+m(\|B^{\frac{1}{2}}x_1\|^2_H)\langle\left(\begin{array}{c} A^{-2}Bx_1 \\ 0 \end{array}\right),\left(\begin{array}{c} x_2 \\ -A^{-2}x_1 \end{array}\right)\rangle_{\mathcal{H}}\\
   &=\langle Ax_1,Ax_2\rangle_H+\langle x_2,-A^{-2}x_1\rangle_H+m(\|B^{\frac{1}{2}}x_1\|^2_H)\langle AA^{-2}Bx_1,Ax_2\rangle+0\\
   &=m(\|B^{\frac{1}{2}}x_1\|^2_H)\langle Bx_1,x_2\rangle_H.
\end{align*}
Moreover,
\begin{align*}
    & \langle DV(x),F(x)\rangle_{\mathcal{H}}=\langle x,F(x)\rangle_{\mathcal{H}}+m(\|B^{\frac{1}{2}}x_1\|^2_H)\langle \left(\begin{array}{c} A^{-2}Bx_1 \\ 0 \end{array}\right),F(x)\rangle_{\mathcal{H}}\\
     &=\langle \left(\begin{array}{c} x_1 \\ x_2\end{array}\right),\left(\begin{array}{c} 0 \\ -m(\|B^{\frac{1}{2}}x_1\|^2_H)Bx_1-f(x_1,x_2) \end{array}\right)\rangle_{\mathcal{H}}\\
     &\hspace{1cm}+m(\|B^{\frac{1}{2}}x_1\|^2_H)\langle \left(\begin{array}{c} A^{-2}Bx_1 \\ 0 \end{array}\right),\left(\begin{array}{c} 0 \\ -m(\|B^{\frac{1}{2}}x_1\|^2_H)Bx_1-f(x_1,x_2) \end{array}\right)\rangle_{\mathcal{H}}\\
     &=\langle x_2,-m(\|B^{\frac{1}{2}}x_1\|^2_H)Bx_1-f(x_1,x_2)\rangle_H+0\\
     &=-m(\|B^{\frac{1}{2}}x_1\|^2_H)\langle x_2,Bx_1\rangle_H-\langle
     x_2,f(x_1,x_2)\rangle_H,\ x\in\mathcal{H}.
\end{align*}
Combining the above equalities, we infer that
\begin{align*}
    \langle DV(x),\mathcal{A}x+F(x)\rangle_{\mathcal{H}}=-\langle x_2,f(x_1,x_2)\rangle_H\ \ \text{for all }x\in\mathcal{D}(\mathcal{A}).
\end{align*}
On the other hand, we find
\begin{align*}
    & \langle DV(x),G(x,z)\rangle_{\mathcal{H}}=\langle x,G(x,z)\rangle_{\mathcal{H}}+m(\|B^{\frac{1}{2}}x_1\|^2_H)\langle\left(\begin{array}{c} A^{-2}Bx_1 \\ 0 \end{array}\right),G(x,z)\rangle_{\mathcal{H}}\\
     &=\langle \left(\begin{array}{c} x_1 \\ x_2 \end{array}\right),\left(\begin{array}{c} 0 \\ g(x_1,x_2,z) \end{array}\right)\rangle_{\mathcal{H}}
     +m(\|B^{\frac{1}{2}}x_1\|^2_H)\langle\left(\begin{array}{c} A^{-2}Bx_1 \\ 0 \end{array}\right),\left(\begin{array}{c} 0 \\ g(x_1,x_2,z) \end{array}\right)\rangle_{\mathcal{H}}\\
     &=\langle x_2,g(x_1,x_2,z)\rangle_H.
\end{align*}
and
\begin{align*}
   & V(x+G(x,z))-V(x)\\
    &=\frac{1}{2}\|x+G(x,z)\|_{\mathcal{H}}^2+\frac{1}{2}M(\|B^{\frac{1}{2}}x_1\|^2_H)-\frac{1}{2}\|x\|^2_{\mathcal{H}}-\frac{1}{2}M(\|B^{\frac{1}{2}}x_1\|^2_H)\\
    &=\frac{1}{2}\|x\|_{\mathcal{H}}^2+\langle x,G(x,z)\rangle_{\mathcal{H}}+\frac{1}{2}\|G(x,z)\|_{\mathcal{H}}^2-\frac{1}{2}\|x\|_{\mathcal{H}}^2\\
    &=\langle x_2,g(x_1,x_2,z)\rangle_H+\frac{1}{2}\|g(x_1,x_2,z)\|^2_H.
\end{align*}
From these relations we obtain
\begin{align*}
       &V(\mathfrak{u}(t\wedge\tau_n))-V(\mathfrak{u}_0)=\int_0^{t\wedge\tau_n}\langle DV(\mathfrak{u}(s),\mathcal{A}\mathfrak{u}(s)+F(s,\mathfrak{u}(s)))\rangle_{\mathcal{H}}ds\\
       &\hspace{1cm}+\int_0^{t\wedge\tau_n}\int_Z\Big{[}V(\mathfrak{u}(s)+G(s,\mathfrak{u}(s),z))-V(\mathfrak{u}(s))\\
       &\hspace{3cm}-\langle DV(\mathfrak{u}(s)),G(s,\mathfrak{u}(s),z)\rangle\Big{]}\nu(dz)\,ds\\
       &\hspace{1cm}+\int_0^{t\wedge\tau_n}\int_Z\Big{[}V\big{(}\mathfrak{u}(s-)+G(s,\mathfrak{u}(s-),z)\big{)}-V(\mathfrak{u}(s-))\Big{]}\tilde{N}(ds,dz)\\
       &=-\int_0^{t\wedge\tau_n}\langle u_t(s),f(s,u(s),u_t(s))\rangle_Hds\\
       &\hspace{1cm}+\int_0^{t\wedge\tau_n}\int_Z\Big{[}\langle u_t(s),g(s,u(s),u_t(s),z)\rangle_{\mathcal{H}}+\frac{1}{2}\|g(s,u(s),u_t(s),z)\|^2_H\\
       &\hspace{2cm}-\langle u_t(s),g(u(s),u_t(s),z)\rangle_H\Big{]}\nu(dz)\,ds\\
       &\hspace{1cm}
       +\int_0^{t\wedge\tau_n}\int_Z\Big{[}\langle \mathfrak{u}_t(s-),g(s,\mathfrak{u}(s-),z)\rangle_{\mathcal{H}}+\frac{1}{2}\|g(s,\mathfrak{u}(s-),z)\|_{\mathcal{H}}^2\Big{]}\tilde{N}(ds,dz)\\
       &=-\int_0^{t\wedge\tau_n}\langle u_t(s),f(s,u(s),u_t(s))\rangle_Hds+\frac{1}{2}\int_0^{t\wedge\tau_n}\int_Z\|g(s,\mathfrak{u}(s),z)\|_{\mathcal{H}}^2\nu(dz)\,ds\\
       &\hspace{1cm}+\int_0^{t\wedge\tau_n}\int_Z\Big{[}\langle \mathfrak{u}_t(s-),g(s,\mathfrak{u}(s-),z)\rangle_{\mathcal{H}}+\frac{1}{2}\|g(s,\mathfrak{u}(s-),z)\|_{\mathcal{H}}^2\Big{]}\tilde{N}(ds,dz).
\end{align*}
Taking expectation to both sides of the above equalities we infer
that
\begin{align*}
       \E V(\mathfrak{u}(t\wedge\tau_n))&=\E V(\mathfrak{u}_0)-\E \int_0^{t\wedge\tau_n}\langle u_t(s),f(s,u(s),u_t(s))\rangle_Hds\\
       &\hspace{1cm}+\frac{1}{2}\E\int_0^{t\wedge\tau_n}\int_Z\|g(s,\mathfrak{u}(s),z)\|_{\mathcal{H}}^2\nu(dz)\,ds\\
       &=\E V(\mathfrak{u}_0)-\E \int_0^t\langle u_t(s),f(s,u(s),u_t(s))\rangle_H1_{(0,{t\wedge\tau_n}]}(s)\,ds\\
       &\hspace{1cm}+\frac{1}{2}\E\int_0^{t}\int_Z\|g(s,\mathfrak{u}(s),z)\|_{\mathcal{H}}^21_{(0,{t\wedge\tau_n}]}(s)\nu(dz)\,ds
       \end{align*}
       \begin{align*}
       &\leq \E V(\mathfrak{u}_0)+\frac{1}{2}(1+K_f)\E\int_0^{t}(1+\|\mathfrak{u}(s\wedge\tau_n)\|_{\mathcal{H}}^2)\,ds\\
       &\hspace{1cm}+\frac{1}{2}K_g\E\int_0^{t\wedge\tau_n}(1+\|\mathfrak{u}(s\wedge\tau_n)\|_{\mathcal{H}}^2)\,ds\\
       &=\E
       V(\mathfrak{u}_0)+\frac{1}{2}(1+K_f+K_g)\int_0^{t}(1+\E\|\mathfrak{u}(s\wedge\tau_n)\|_{\mathcal{H}}^2)\,ds,
       \ t\geq0.
\end{align*}
Above we used the growth conditions \eqref{growth condition of
f}-\eqref{growth condition of g} of functions $f$ and $g$.
Therefore, inequality \eqref{eq350} holds if we set $C=\frac{1}{2}(1+K_f+K_g)$.\\
\textbf{\textit{Step 2: Inequality \eqref{eq350} holds for a local mild solution.}}\\
In this case, one of the main obstacles is that the solution
$\mathfrak{u}$ to Problem \eqref{SDE1} under Assumptions \ref{assu:
growth condition} and \ref{assu: locally Lipschitz condition} is a
local mild solution, so the lifespan of solution $\tau_{\infty}$ may
be finite. For this, we fix $n\in\mathbb{N}$ and introduce the
following functions
\begin{align*}
      \tilde{f}(t)&=1_{[0,\tau_n)}(t)f(t,\mathfrak{u}(t\wedge\tau_n)),\ t\geq0,\\
      \tilde{g}(t,z)&=1_{[0,\tau_n]}(t)g(t,\mathfrak{u}(t\wedge\tau_n-),z),\ t\geq0 \text{ and }z\in Z.
\end{align*}
Here $\mathfrak{u}(t)$, $0\leq t<\tau_{\infty}$, with
$\tau_{\infty}=\lim_{n\rightarrow\infty}\tau_n$, is the unique local
mild solution of Problem \eqref{SDE1} under Assumptions \ref{assu:
growth condition} and \ref{assu: locally Lipschitz condition}.
Denote
\begin{align*}
   \tilde{F}(t)=\left(\begin{array}{c} 0 \\ -\tilde{f}(t)-m(\|B^{\frac{1}{2}}u(t\wedge\tau_n)\|^2_H)Bu(t\wedge\tau_n)1_{[0,\tau_n)}(t) \end{array}\right)\ \ \text{and}\ \  \tilde{G}(t,z)=\left(\begin{array}{c} 0 \\ \tilde{g}(t,z) \end{array}\right).
\end{align*}
One can see that the process $\tilde{F}$ and $\tilde{G}$ are bounded. So
  Consider
the following linear non-homogeneous stochastic equation
\begin{align}
\begin{split}\label{SDE3}   dv(t)&=\mathcal{A}v(t)dt+\tilde{F}(t)dt+\int_Z\tilde{G}(t,z)\tilde{N}(dt,dz),\ t\geq0,\\
  v(0)&=\mathfrak{u}(0). 
\end{split}
\end{align}
By Theorem \ref{theo: exsitence and uniqueness of mild solution},
there exists a unique global mild solution of this equation which is given by
\begin{align}\label{eq7}
       v(t)=e^{t\mathcal{A}}\mathfrak{u}(0)+\int_0^te^{(t-s)\mathcal{A}}\tilde{F}(s)\,ds+\int_0^{t}\int_Ze^{(t-s)\mathcal{A}}\tilde{G}(s,z)\tilde{N}(ds,dz),\ t\geq0.
\end{align}
Hence the stopped process $v(\cdot\wedge\tau_n)$ satisfies
\begin{align*}
   v(t\wedge\tau_n)
    =e^{(t\wedge\tau_n)\mathcal{A}}\mathfrak{u}(0)
    +\int_0^{t\wedge\tau_n}e^{(t\wedge\tau_n-s)\mathcal{A}}\tilde{F}(s)\,ds
    +I_{\tau_n}(\tilde{G})(t\wedge\tau_n),\ t\geq0,
\end{align*}
where as usual
\begin{align*}
I_{\tau_n}(\tilde{G})(t)=\int_0^{t}\int_Z1_{[0,\tau_n]}(s)e^{(t-s)\mathcal{A}}\tilde{G}(s,z)\tilde{N}(ds,dz).
\end{align*}
One can observe that
\begin{align*}
    I_{\tau_n}(\tilde{G})(t)&=\int_0^{t}\int_Z1_{[0,\tau_n]}(s)e^{(t-s)\mathcal{A}}\tilde{G}(s,z)\tilde{N}(ds,dz)\\
               &=\int_0^{t}\int_Z1_{[0,\tau_n]}(s)e^{(t-s)\mathcal{A}}G(s,\mathfrak{u}(s\wedge\tau_n-),z)\tilde{N}(ds,dz)\\
               &=\int_0^{t}\int_Z1_{[0,\tau_n]}(s)e^{(t-s)\mathcal{A}}G(s,\mathfrak{u}(s-),z)\tilde{N}(ds,dz)\\
               &=I_{\tau_n}(G(\mathfrak{u}))(t),\ t\geq0.
\end{align*}
 Therefore, on the basis of Lemma \ref{stochasitc convolution}, we find out that for
each $n\in\mathbb{N}$
\begin{align*}
    v(t\wedge\tau_n)
    &=e^{(t\wedge\tau_n)\mathcal{A}}\mathfrak{u}(0)
    +\int_0^{t\wedge\tau_n}e^{(t\wedge\tau_n-s)\mathcal{A}}\tilde{F}(s)\,ds
    +I_{\tau_n}(\tilde{G})(t\wedge\tau_n)\\
    &=e^{(t\wedge\tau_n)\mathcal{A}}\mathfrak{u}(0)
    +\int_0^{t\wedge\tau_n}1_{(0,\tau_n]}e^{(t\wedge\tau_n-s)\mathcal{A}}\tilde{F}(s)\,ds
    +I_{\tau_n}(G(\mathfrak{u}))(t\wedge\tau_n)\\
                    &=e^{(t\wedge\tau_n)\mathcal{A}}\mathfrak{u}(0)
                    +\int_0^{t\wedge\tau_n}e^{(t\wedge\tau_n-s)\mathcal{A}}1_{[0,\tau_n]}(s)F(s,\mathfrak{u}(s\wedge\tau_n))\,ds
                    +I_{\tau_n}(G(\mathfrak{u}))(t\wedge\tau_n)\\
                    &=\mathfrak{u}(t\wedge\tau_n)\ \ \
                    \mathbb{P}\text{-a.s.}\
                    t\geq0.
\end{align*}
The second difficulty here is that the It\^{o} formula is only
applicable to strong solution. So our next step is to find a
sequence of global strong solutions which converges to the global
mild solution $v$ uniformly. To do this, let us set, with
$R(m;\mathcal{A})=(mI-\mathcal{A})^{-1}$,
\begin{align*}
\mathfrak{u}_m(0)&=mR(m;\mathcal{A})\mathfrak{u}(0);\\\
\tilde{F}_m(t,\omega)&=mR(m;\mathcal{A})\tilde{F}(t,\omega)\ \
\text{for}\ (t,\omega)\in\mathbb{R}_+\times\Omega;\\
\tilde{G}_m(t,\omega,z)&=mR(m;\mathcal{A})\tilde{G}(t,\omega,z)\ \
\text{ for } (t,\omega,z)\in\mathbb{R}_+\times\Omega\times Z.
\end{align*}

 Since $\mathcal{A}$ is the infinitesimal
generator of a contraction $C_0$-semigroup
$(e^{t\mathcal{A}})_{t\geq0}$, by the Hille-Yosida Theorem,
$\|R(m;\mathcal{A})\|\leq\frac{1}{m}$,
$\tilde{F}_m(t,\omega)\in\mathcal{D}(\mathcal{A})$, for every
$(t,\omega)\in\mathbb{R}_+\times\Omega$ and
$\tilde{G}_m(t,\omega,z)\in\mathcal{D}(\mathcal{A})$ for every
$(t,\omega,z)\in\mathbb{R}_+\times\Omega\times Z$. Moreover,
$\tilde{F}_m(t,\omega)\rightarrow \tilde{F}(t,\omega)$ pointwise
on $\mathbb{R}_+\times\Omega$ and
$\tilde{G}_m(t,\omega,z)\rightarrow\tilde{G}(t,\omega,z)$
pointwise on $\mathbb{R}_+\times\Omega\times Z$. Next, we note that $\|\tilde{F}_m\|_{\mathcal{H}}$ and
$\|\tilde{F}_m-\tilde{F}\|_{\mathcal{H}}$ is bounded from above by a function
$2\|F\|_{\mathcal{H}}$ belonging to $\mathcal{M}^2_{loc}(\mathcal{P};\mathbb{R})$,
so the Lebesgue Dominated Convergence Theorem tells us that for every $T>0$,
\begin{align}
 \lim_{m\rightarrow\infty}\E\int_0^T\|\tilde{F}_m(t)-\tilde{F}(t)\|_{\mathcal{H}}^2\,dt=0\label{eq41}
\end{align}
Analogously, we know that $\|\tilde{G}_m \|_{\mathcal{H}} $ and  $\|\tilde{G}_m-\tilde{G}\|_{\mathcal{H}}$ are bounded by
functions $\|G\|_{\mathcal{H}}$ and $2\|G\|_{\mathcal{H}}$, respectively, which are both belonging to $\mathcal{M}^2_{loc}(\hat{\mathcal{P}};\mathbb{R})$.
So again we can apply the Lebesgue Dominated Convergence Theorem to
find out that for all $T>0$,
\begin{align}
   \lim_{m\rightarrow\infty}\E\int_0^T\int_Z|\tilde{G}_m(t,z)-\tilde{G}(t,z)|^2_{\mathcal{H}}\nu(dz)dt=0.\label{eq42}
\end{align}
Clearly, by the definition,  $\tilde{F}_m(t,\omega)\in\mathcal{D}(\mathcal{A})$, for all $(t,\omega)\in\mathbb{R}_+\times\Omega $ and $\tilde{G}_m(t,\omega,z)\in\mathcal{D}(\mathcal{A})$, for all $(t,\omega,z)\in\mathbb{R}_+\times\Omega \times Z$. Hence by the boundedness discussed before, we infer
$\tilde{F}_m\in\mathcal{M}^2_{loc}(\mathcal{B}\mathcal{F};\mathcal{D}(\mathcal{A}))$
and
$\tilde{G}_m\in\mathcal{M}^2_{loc}(\hat{\mathcal{P}};\mathcal{D}(\mathcal{A}))$, $m\in\mathbb{N}$.\\
 From Theorem \ref{theo:
exsitence and uniqueness of mild solution}, it follows that the
following Equation
\begin{align*}
     dv_m(t)&=\mathcal{A}v_m(t)dt+\tilde{F}_m(t)dt+\int_Z\tilde{G}_m(t,z)\tilde{N}(dt,dz),\ t\geq0\\
      v_m(0)&=\mathfrak{u}_m(0).
\end{align*}
has a unique global strong solution which satisfies that $\mathbb{P}$-a.s. for all $t\geq0$,
\begin{align}\label{eq5}
     v_m(t)=e^{t\mathcal{A}_m}\mathfrak{u}_m(0)+\int_0^te^{(t-s)\mathcal{A}_m}\tilde{F}_m(s)\,ds
     +\int_0^{t}\int_Ze^{(t-s)\mathcal{A}}\tilde{G}_m(s,z)\tilde{N}(ds,dz),
\end{align}
Note that we can rewrite this global strong solution in the following
form
\begin{align}\label{eq6}
    v_m(t)=\mathfrak{u}_m(0)+\int_0^t\Big{[}\mathcal{A}v_m(s)+\tilde{F}_m(s)\Big{]}ds+\int_0^{t}\int_Z\tilde{G}_m(s,z)\tilde{N}(ds,dz),\
    t\geq0.
\end{align}
Let $\sigma$ be a stopping time. Now we can apply It\^{o} formula, see \cite{[Zhu_2010]}, to the process $v_m$ of the form \eqref{eq6}
and the function $V$ to get
\begin{align}\label{eq40}
\begin{split}
       &V(v_m(\sigma))-V(\mathfrak{u}_m(0))\\
       &=\int_0^{\sigma}\langle DV(v_m(s),\mathcal{A}v_m(s)+\tilde{F}_m(s))\rangle_{\mathcal{H}}ds\\
       &+\int_0^{\sigma}\int_Z\Big{[}V(v_m(s)+\tilde{G}_m(s,z))-V(v_m(s))-\langle DV(v_m(s)),\tilde{G}_m(s,z)\rangle\Big{]}\nu(dz)\,ds\\
       &+\int_0^{\sigma}\int_Z\Big{[}V\big{(}v_m(s-)+\tilde{G}_m(s,z)\big{)}-V(v_m(s-))\Big{]}\tilde{N}(ds,dz).
\end{split}
\end{align}
We next observe that for every $T>0$,
\begin{align}\label{eq8}
    \lim_{m\rightarrow\infty}\E\sup_{0\leq t\leq T}\|v_m(t)-v(t)\|^2_{\mathcal{H}}=0.
\end{align}
Indeed, from \eqref{eq7} and \eqref{eq5} we find out that
\begin{align*}
   v_m(t)-v(t)&=\int_0^te^{(t-s)\mathcal{A}}\left(\tilde{F}(s)-\tilde{F}_m(s)\right)\,ds\\
   &\hspace{1cm}+\int_0^{t}\int_Ze^{(t-s)\mathcal{A}}\left(\tilde{G}(s,z)-\tilde{G}_m(s,z)\right)\tilde{N}(ds,dz),\
   t\geq0.
\end{align*}
Using the Cauchy-Swartz inequality, we obtain
\begin{align*}
     \E\sup_{0\leq t\leq T}\left\|\int_0^te^{(t-s)\mathcal{A}}\left(\tilde{F}(s)-\tilde{F}_m(s)\right)\,ds\right\|^2_{\mathcal{H}}
     &\leq
     T\E\sup_{0\leq t\leq
     T}\int_0^t\left\|e^{(t-s)\mathcal{A}}\left(\tilde{F}(s)-\tilde{F}_m(s)\right)\right\|^2_{\mathcal{H}}ds\\
     &\leq
     T\E\int_0^T\left\|\tilde{F}(s)-\tilde{F}_m(s)\right\|^2_{\mathcal{H}}ds.
\end{align*}
The right side of above inequality converges to 0, as
$m\rightarrow\infty$, as we have already shown before in
\eqref{eq41}. Therefore, we obtain
\begin{align*}
      \lim_{m\rightarrow\infty}\E\sup_{0\leq t\leq T}\left\|\int_0^te^{(t-s)\mathcal{A}}\left(\tilde{F}(s)-\tilde{F}_m(s)\right)\,ds\right\|^2_{\mathcal{H}}=0.
\end{align*}
Meanwhile, we can use the Davis inequality for stochastic convolution
processes, see \cite{[Brzezniak_Hau_Zhu]}, to deduce that
\begin{align}\label{eq43}
    \E\sup_{0\leq t\leq T}\left\|\int_0^t\int_Ze^{(t-s)\mathcal{A}}\left(\tilde{G}(s,z)-\tilde{G}_m(s,z)\right)\tilde{N}(ds,dz)\right\|^2_{\mathcal{H}}\nonumber\\
      \leq
      C\E\int_0^T\int_Z\|\tilde{G}(t,z)-\tilde{G}_m(t,z)\|^2_{\mathcal{H}}\nu(dz)dt.
\end{align}
Note that the right side of \eqref{eq43} converges to $0$ as
$m\rightarrow\infty$ by \eqref{eq42}. Hence we have
\begin{align*}
   \lim_{m\rightarrow\infty}\E
   \sup_{0\leq t\leq T} \left\|\int_0^t\int_Ze^{(t-s)\mathcal{A}}\left(\tilde{G}(s,z)-\tilde{G}_m(s,z)\right)\tilde{N}(ds,dz)\right\|^2_{\mathcal{H}}=0,
\end{align*}
which proves equality \eqref{eq8}.\\
Therefore, we conclude that $v_m(t)$
  converges to $v(t)$  uniformly on any closed interval $[0,T]$, $0<T<\infty$, $\mathbb{P}$-a.s.
  Hence, by taking a subsequence if necessary we may assume that  $v_m(t)\rightarrow v(t)$,
 uniformly and
$\tilde{F}_{m}(s)\rightarrow\tilde{F}(s)$ and
$\tilde{G}_{m}(s,z)\rightarrow\tilde{G}(s,z)$  on
$[0,\sigma(\omega)]$, as $m\rightarrow\infty$, for almost all $\omega$ in $\Omega$.

We introduce the following canonical linear projection mappings
\begin{align*}
      &\pi_1:\mathcal{H}\ni\left(\begin{array}{c} x \\ y \end{array}\right)\mapsto x\in\mathcal{D}(A)\\
      &\pi_2:\mathcal{H}\ni\left(\begin{array}{c} x \\ y \end{array}\right)\mapsto y\in\mathcal{H}.
\end{align*}

Calculations similar to those performed in Step 1 yield
\begin{align*}
   &\langle DV(v_m(s),\mathcal{A}v_m(s)+\tilde{F}_m(s))\rangle_{\mathcal{H}}\\
   &=\langle v_m(s),\mathcal{A}v_m(s)+\tilde{F}_m(s))\rangle_{\mathcal{H}}\\
   &\hspace{1cm}+m(\|B^{\frac{1}{2}}\pi_1v_m(s)\|^2_H)\langle \left(\begin{array}{c} A^{-2}B\pi_1v_m(s) \\ 0 \end{array}\right),\mathcal{A}v_m(s)+\tilde{F}_m(s))\rangle_{\mathcal{H}}\\
   &=\langle v_m(s),\mathcal{A}v_m\rangle_{\mathcal{H}}+\langle v_m(s),\tilde{F}_m(s)\rangle_{\mathcal{H}}\\
   &\hspace{1cm}+m(\|B^{\frac{1}{2}}\pi_1v_m(s)\|^2_H)\langle
   B\pi_1v_m(s),\pi_1\mathcal{A}v_m(s)+
   \pi_1\tilde{F}_m(s)\rangle_H\\
   &\leq\langle v_m(s),\tilde{F}_m(s)\rangle_{\mathcal{H}}+m(\|B^{\frac{1}{2}}\pi_1v_m(s)\|^2_H)\langle
   B\pi_1v_m(s),\pi_2v_m(s)+
   \pi_1\tilde{F}_m(s)\rangle_H,\ s\geq0,
\end{align*}
where we used the facts that $\pi_1\mathcal{A}v_m(t)=\pi_2v_m(t)$ for
$t\in [0,T]$ and $$\langle v(s),\mathcal{A}v(s)\rangle_{\mathcal{H}}\leq0, \text{ for all }s,$$
since the operator $\mathcal{A}$ is dissipative.

Moreover, since for all appropriate $(s,z)$, $$\pi_1\tilde{G}_m(s,z)=m(m^2I+A^2)^{-1}\tilde{g}(s,z)$$ and
$$\pi_2\tilde{G}_m(s,z)=m^2(m^2I+A^2)^{-1}\tilde{g}(s,z),$$

we infer that
\begin{align}\label{eq-360}
\begin{split}
    & \langle DV(v_m(s)),\tilde{G}_m(s,z)\rangle_{\mathcal{H}}
     \\
     &=\langle v_m(s),\tilde{G}_m(s,z)\rangle_{\mathcal{H}}+m(\|B^{\frac{1}{2}}\pi_1v_m(s)\|^2_H)\langle\left(
     \begin{array}{c} A^{-2}B\pi_1v_m(s) \\ 0 \end{array}\right),\tilde{G}_m(s,z)\rangle_{\mathcal{H}}
     \\
     &=\langle v_m(s),\tilde{G}_m(s,z)\rangle_{\mathcal{H}}
     +m(\|B^{\frac{1}{2}}\pi_1v_m(s)\|^2_H)\langle\left(\begin{array}{c} A^{-2}B\pi_1v_m(s) \\ 0 \end{array}\right),\left(\begin{array}{c} \pi_1\tilde{G}_m(s,z) \\
      \pi_1\tilde{G}_m(s,z) \end{array}\right)\rangle_{\mathcal{H}}
     \\
     &=\langle v_m(s),\tilde{G}_m(s,z)\rangle_{\mathcal{H}}+m(\|B^{\frac{1}{2}}\pi_1v_m(s)\|^2_H)\langle
     A^{-2}B\pi_1v_m(s),\pi_1\tilde{G}_m(s,z)\rangle_H,\ s\geq0.
     \end{split}
    \end{align}

 Furthermore, we have
\begin{align}
\begin{split}\label{eq-361}
    V(v_m(s)&+\tilde{G}_m(s,z))-V(v_m(s))\\
    &=\frac{1}{2}\|v_m(s)+\tilde{G}_m(s,z)\|_{\mathcal{H}}^2+\frac{1}{2}M(\|B^{\frac{1}{2}}\pi_1(v_m(s)+\tilde{G}_m(s,z))\|^2_H)\\
    &\hspace{1cm}-\frac{1}{2}\|v_m(s)\|^2_{\mathcal{H}}-\frac{1}{2}M(\|B^{\frac{1}{2}}\pi_1v_m(s)\|^2_H)\\
    &=\langle v_m(s),\tilde{G}_m(s,z)\rangle_{\mathcal{H}}+\frac{1}{2}\|\tilde{G}_m(s,z)\|^2_{\mathcal{H}}+\frac{1}{2}M(\|B^{\frac{1}{2}}\pi_1(v_m(s)+\tilde{G}_m(s,z))\|^2_H)\\
    &\hspace{1cm}-\frac{1}{2}M(\|B^{\frac{1}{2}}\pi_1v_m(s)\|^2_H).
    \end{split}
\end{align}
Hence equality \eqref{eq40} becomes
\begin{align}\label{eq48}
       &V(v_m(\sigma))-V(\mathfrak{u}_m(0))\\
       &\leq\int_0^{\sigma}\Big{[}\langle v_m(s),\tilde{F}_m(s)\rangle_{\mathcal{H}}+m(\|B^{\frac{1}{2}}\pi_1v_m(s)\|^2_H)\langle
   B\pi_1v_m(s),\pi_2v_m(s)+
      \pi_1\tilde{F}_m(s)\rangle_H\Big{]}ds\nonumber\\
       &\hspace{1cm}+\int_0^{\sigma}\int_Z\Big{[}V(v_m(s)+\tilde{G}_m(s,z))-V(v_m(s))-\langle DV(v_m(s)),\tilde{G}_m(s,z)\rangle\Big{]}\nu(dz)\,ds\nonumber\\
       &\hspace{1cm}+\int_0^{\sigma}\int_Z\Big{[}V\big{(}v_m(s-)+\tilde{G}_m(s,z)\big{)}-V(v_m(s-))\Big{]}\tilde{N}(ds,dz).
\end{align}

 Note that $\pi_1\tilde{F}(s,\omega)=0$ on $\mathbb{R}_+\times\Omega$ and
$\pi_1\tilde{G}(s,\omega,z)=0$ on $\mathbb{R}_+\times\Omega\times Z$.
Since the functions $m(\cdot)$ and $M(\cdot)$ are continuous and the operator $B\in\mathcal{L}(\mathcal{D}(\mathcal{A}),H)$, we have $\mathbb{P}$-a.s.
 \begin{align*}
        \pi_1v_m(s)&\rightarrow\pi_1v(s),\\
        m(\|B^{\frac{1}{2}}\pi_1 v_m(s)\|^2_H)&\rightarrow m(\|B^{\frac{1}{2}}\pi_1 v(s)\|^2_H),\\
        \pi_2v_m(s)&\rightarrow\pi_2v(s),\\
        B\pi_1v_m(s)&\rightarrow B\pi_1v_m(s),\ \ \  \end{align*}
  uniformly on $[0,\sigma(\omega)]$, as $m\rightarrow\infty$ and
\begin{align*}
        \pi_1\tilde{G}_m(s,z)&\rightarrow 0,\\
        M(\|B^{\frac{1}{2}}\pi_1(v_m(s)+\tilde{G}_m(s,z))\|^2_H)&\rightarrow
        M(\|B^{\frac{1}{2}}\pi_1v(s)\|^2_H)
\end{align*}
on $[0,\sigma(\omega)]$ for all most all $\omega\in\Omega$, as
$m\rightarrow\infty$.  We also notice that for every
$m\in\mathbb{N}$, the set $\{v_m(t,\omega): t\in [0, T]\}$ is relatively 
compact for almost all $\omega$ and the sequence
$\{v_m\}_{m\in\mathbb{N}}$ converges uniformly to $v$,
$\mathbb{P}$-a.s.. Hence the set $\{v_m(s),s\in[0,T], m\in\mathbb{N}\}$ is
bounded in $\mathcal{H}$, $\mathbb{P}$-a.s. It follows that
\begin{align*}
     \langle
v_m(s),\tilde{F}_m(s)\rangle_{\mathcal{H}}\leq
\|v_m(s)\|_{\mathcal{H}}\|\tilde{F}_m(s)\|_{\mathcal{H}}\leq
\|\tilde{F}_m(s)\|_{\mathcal{H}}\sup_{0\leq s\leq
T}\|v_m(s)\|_{\mathcal{H}}\leq C\|\tilde{F}(s)\|_{\mathcal{H}}.
\end{align*}

 Therefore, on the basis of the Lebesgue
Dominated convergence theorem, we conclude that
\begin{align*}
    \int_0^{\sigma}\langle v_m(s),\tilde{F}_m(s)\rangle_{\mathcal{H}}ds\rightarrow\int_0^{\sigma}\langle
v_m(s),\tilde{F}_m(s)\rangle_{\mathcal{H}}ds\ \
\mathbb{P}\text{-a.s.}
\end{align*}
Analogously, by the continuity of the function $m$ and the fact that
$B\in\mathcal{L}(\mathcal{D}(A),H)$, we infer for some constants
$C_1$, $C_2$,
\begin{align*}
     m(\|B^{\frac{1}{2}}\pi_1v_m(s)\|^2_H)\langle
   B\pi_1v_m(s),\pi_2v_m(s)+
   \pi_1\tilde{F}_m(s)\rangle_H\leq C_1+C_2\|\tilde{F}(s)\|_{\mathcal{H}}.
\end{align*}
Moreover, we know that for almost all $\omega\in\Omega$
\begin{align*}
m(\|B^{\frac{1}{2}}\pi_1v_m(s)\|^2_H)\langle
   B\pi_1v_m(s),\pi_2v_m(s)+
   \pi_1\tilde{F}_m(s)\rangle_H
\end{align*}
converges on $[0,\sigma(\omega)]$  as
  $m\rightarrow\infty$ to
$$m(\|B^{\frac{1}{2}}\pi_1v(s)\|^2_H)\langle
   B\pi_1v(s),\pi_2v(s)
   \rangle_H.$$
Again, it follows from the Lebesgue Dominated convergence theorem that
$\mathbb{P}$-a.s.
\begin{align*}
    \int_0^{\sigma}
  \Big{[}m(\|B^{\frac{1}{2}}\pi_1v_m(s)\|^2_H)\langle
   B\pi_1v_m(s),\pi_2v_m(s)+
   \pi_1\tilde{F}_m(s)\rangle_H\rangle_{\mathcal{H}}\Big{]}ds
\end{align*}
converges as $m\rightarrow\infty$ to
\begin{align*}
\int_0^{\sigma}m(\|B^{\frac{1}{2}}\pi_1v(s)\|^2_H)\langle
   B\pi_1v(s),\pi_2v(s)\,ds.
\end{align*}
In conclusion, $\mathbb{P}$-a.s. the first term on the right
side of inequality \eqref{eq48} converges as $m\rightarrow\infty$ to
$$\int_0^{\sigma}\Big{[}\langle v(s),\mathcal{A}v\rangle_{\mathcal{H}}+\langle v(s),\tilde{F}(s)\rangle_{\mathcal{H}}+m(\|B^{\frac{1}{2}}\pi_1v(s)\|^2_H)\langle
   B\pi_1v(s),\pi_2v(s)
   \rangle_H\Big{]}ds.$$
    \\
  Also,  we know from \eqref{eq-360} and \eqref{eq-361} that as $m\rightarrow\infty$ for all $t\in[0,T]$, $z\in Z$,
$\mathbb{P}$-a.s.
\begin{align*}
 &V(v_m(s)+\tilde{G}_m(s,z))-V(v_m(s))-\langle
DV(v_m(s)),\tilde{G}_m(s,z)\rangle_{\mathcal{H}}\rightarrow\frac{1}{2}\|\tilde{G}(s,z)\|^2_H\\
&V\big{(}v_m(s)+\tilde{G}_m(s,z)\big{)}-V(v_m(s))\rightarrow\langle
v(s),\tilde{G}(s,z)\rangle_{\mathcal{H}}+\frac{1}{2}\|\tilde{G}(s,z)\|^2_{\mathcal{H}}.
\end{align*}

Set $X(\omega)=\{v_m(t,\omega): t\in [0, T],m\in\mathbb{N}\}$, for
$\omega\in\Omega$. As we have noticed before, $X(\omega)$ is a
bounded subset of $\mathcal{H}$ for almost all $\omega\in\Omega$.
Since the functions $DV$ and $D^2V$ are uniformly continuous on
bounded subsets of $\mathcal{H}$, so $\sup_{x\in X} |DV(x)|<\infty$
and $\sup_{x\in X}|D^2V(x)|<\infty$, $\mathbb{P}$-a.s. Hence by the Taylor
formula, one have
\begin{align*}
   V(v_m(s)+\tilde{G}_m(s,z))&-V(v_m(s))-\langle
DV(v_m(s)),\tilde{G}_m(s,z)\rangle_{\mathcal{H}}\\
&\leq
\frac{1}{2}\|D^2V(v_m(s))\|\|\tilde{G}_m(s,z)\|^2_{\mathcal{H}}\\
&\leq \frac{1}{2}\sup_{x\in
X}\|D^2V(x)\|\|\tilde{G}(s,z)\|^2_{\mathcal{H}}.
\end{align*}
We also observe that since
$\tilde{G}\in\mathcal{M}^2_{loc}(\hat{\mathcal{P}};\mathcal{H})$,
for every $0<T<\infty$, we have
$$\int_0^T\int_Z\|\tilde{G}(s,z)\|^2\nu(dz)\,ds<\infty,
\ \mathbb{P}\text{-a.s.}$$ By using above result, along with the Lebesgue Dominated
Convergence Theorem, we obtain that
\begin{align*}
\int_0^{\sigma}\int_Z V(v_m(s)+\tilde{G}_m(s,z))&-V(v_m(s))-\langle
DV(v_m(s)),\tilde{G}_m(s,z)\rangle_{\mathcal{H}}\nu(dz)\,ds
\end{align*}
converges to
\begin{align*}
\int_0^{\sigma}\int_Z\frac{1}{2}\tilde{G}(s,z)\nu(dz)\,ds,\
\mathbb{P}\text{-a.s.  as }m\rightarrow\infty.
\end{align*}
On the other hand, by the It\^{o} isometry property of the stochastic integral, see \cite{[Zhu_2010]}, we have
   \begin{align*}
      &\E\Big{\|}\int_0^{\sigma}\int_Z\Big{[}V\big{(}v_m(s-)+\tilde{G}_m(s,z)\big{)}-V(v_m(s-))\Big{]}\tilde{N}(ds,dz)\\
      &\hspace{2cm}-\int_0^{\sigma}\int_Z\Big{[}\langle
v(s-),\tilde{G}(s,z)\rangle_{\mathcal{H}}+\frac{1}{2}\|\tilde{G}(s,z)\|^2_{\mathcal{H}}\Big{]}\tilde{N}(ds,dz)\Big{\|}_{\mathcal{H}}^2\\
      &=\E\int_0^{\sigma}\int_Z\left|V\big{(}v_m(s)+\tilde{G}_m(s,z)\big{)}-V(v_m(s))-\langle
v(s),\tilde{G}(s,z)\rangle_{\mathcal{H}}+\frac{1}{2}\|\tilde{G}(s,z)\|^2_{\mathcal{H}}\right|^2\nu(dz)\,ds.
   \end{align*}
   Moreover, we note that the integrand $$\Big{|}V\big{(}v_m(s)+\tilde{G}_m(s,z)\big{)}-V(v_m(s))-\langle
v(s),\tilde{G}(s,z)\rangle_{\mathcal{H}}+\frac{1}{2}\|\tilde{G}(s,z)\|^2_{\mathcal{H}}\Big{|}^2$$
is bounded by $2\sup_{x\in
X}\|DV(x)\|^2\|\tilde{G}(s,z)\|^2_{\mathcal{H}}\leq
C\|\tilde{G}(s,z)\|^2_{\mathcal{H}}$. Since $G\in\mathcal{M}^2(\mathcal{P};\mathcal{H})$, for every
$0<T<\infty$,
$\E\int_0^T\int_Z\|\tilde{G}(s,z)\|^2_{\mathcal{H}}\nu(dz)\,ds<\infty$.
So
   again, we can apply the Lebesgue Dominated Converges Theorem to get
\begin{align*}
\lim_{m\rightarrow\infty}&\E\Big{\|}\int_0^{\sigma}\int_Z\Big{[}V\big{(}v_m(s-)+\tilde{G}_m(s,z)\big{)}-V(v_m(s-))\Big{]}\tilde{N}(ds,dz)\\
      &\hspace{2cm}-\int_0^{\sigma}\int_Z\Big{[}\langle
v(s-),\tilde{G}(s,z)\rangle_{\mathcal{H}}+\frac{1}{2}\|\tilde{G}(s,z)\|^2_{\mathcal{H}}\Big{]}\tilde{N}(ds,dz)\Big{\|}_{\mathcal{H}}^2=0.
\end{align*}
Hence by taking a subsequence, we infer that
$$\int_0^{\sigma}\int_Z\Big{[}V\big{(}v_m(s-)+\tilde{G}_m(s,z)\big{)}-V(v_m(s-))\Big{]}\tilde{N}(ds,dz)$$
converges $\mathbb{P}$-a.s. to $$\int_0^{\sigma}\int_Z\Big{[}\langle
v(s-),\tilde{G}(s,z)\rangle_{\mathcal{H}}+\frac{1}{2}\|\tilde{G}(s,z)\|^2_{\mathcal{H}}\Big{]}\tilde{N}(ds,dz).$$
Also, it is not hard to see that
\begin{align}
    \lim_{n\rightarrow\infty} V(v_m(\sigma))&=\frac{1}{2}\lim_{n\rightarrow\infty}\|v_m(\sigma)\|^2_H+\frac{1}{2}\lim_{n\rightarrow\infty}M(\|B^{\frac{1}{2}}\pi_1v_m(\sigma)\|^2_H)\nonumber\\
                  &=\frac{1}{2}\|v(\sigma)\|_H^2+\frac{1}{2}M(\|B^{\frac{1}{2}}\pi_1v(\sigma)\|^2_H)\nonumber\\
                  &= V(v(\sigma)).
\end{align}
From above observation, by letting $m\rightarrow\infty$ in
inequality \eqref{eq48}, one easily deduces that
\begin{align}
\begin{split}
      &V(v(\sigma))-V(\mathfrak{u}_0)\\
       &=\int_0^{\sigma}\Big{[}\langle v(s),\mathcal{A}v\rangle_{\mathcal{H}}+\langle v(s),\tilde{F}(s)\rangle_{\mathcal{H}}+m(\|B^{\frac{1}{2}}\pi_1v(s)\|^2_H)\langle
   B\pi_1v(s),\pi_1\mathcal{A}v(s)
   \rangle_H\Big{]}ds\\
       &\hspace{1cm}+\int_0^{\sigma}\int_Z\frac{1}{2}\tilde{G}(s,z)\nu(dz)\,ds\\
       &\hspace{1cm}+\int_0^{\sigma}\int_Z\Big{[}\langle
v(s-),\tilde{G}(s,z)\rangle_{\mathcal{H}}+\frac{1}{2}\|\tilde{G}(s,z)\|^2_{\mathcal{H}}\Big{]}\tilde{N}(ds,dz)\\
&\leq \int_0^{\sigma}\Big{[}\langle
v(s),\tilde{F}(s)\rangle_{\mathcal{H}}+m(\|B^{\frac{1}{2}}\pi_1v(s)\|^2_H)\langle
   B\pi_1v(s),\pi_1\mathcal{A}v(s)
   \rangle_H\Big{]}ds\\
       &\hspace{1cm}+\int_0^{\sigma}\int_Z\frac{1}{2}\tilde{G}(s,z)\nu(dz)\,ds\\
       &\hspace{1cm}+\int_0^{\sigma}\int_Z\Big{[}\langle
v(s-),\tilde{G}(s,z)\rangle_{\mathcal{H}}+\frac{1}{2}\|\tilde{G}(s,z)\|^2_{\mathcal{H}}\Big{]}\tilde{N}(ds,dz),\
\mathbb{P}\text{-a.s.}
\end{split}
\end{align}

Therefore, $\mathbb{P}$-a.s.
  \begin{align*}
     V(v(\sigma))-V(\mathfrak{u}_0)
              &\leq\int_0^{\sigma}\Big{[}\langle \pi_2v(s),\tilde{f}(s)\rangle_{\mathcal{H}}+m(\|B^{\frac{1}{2}}\pi_1v(s)\|^2_H)\langle B\pi_1v(s),\pi_1\mathcal{A}v(s)\rangle\Big{]}ds\\
       &\hspace{1cm}+\frac{1}{2}\int_0^{\sigma}\int_Z\|\tilde{g}(s,z)\|^2_H\nu(dz)\,ds\\
       &\hspace{1cm}+\int_0^{\sigma}\int_Z\Big{[}\langle \pi_2v(s-),\tilde{g}(s,z)\rangle_{H}+\frac{1}{2}\|\tilde{g}(s,z)\|^2_H\Big{]}\tilde{N}(ds,dz).
  \end{align*}
  Taking expectation to both sides, we have
    \begin{align*}
              \E V(v(\sigma))\leq \E V(\mathfrak{u}_0)&+\E\int_0^{\sigma}\Big{[}\langle \pi_2v(s),\tilde{f}(s)\rangle_{\mathcal{H}}+m(\|B^{\frac{1}{2}}\pi_1v(s)\|^2_H)\langle B\pi_1v(s),\pi_1\mathcal{A}v(s)\rangle\Big{]}ds\\
       &+\frac{1}{2}\E\int_0^{\sigma}\int_Z\|\tilde{g}(s,z)\|^2_H\nu(dz)\,ds.
  \end{align*}
  Now let us recall that $v(t\wedge\tau_n)=\mathfrak{u}(t\wedge\tau_n)$, $ \tilde{F}(t)=1_{(0,\tau_n]}(t)F(t,\mathfrak{u}(t\wedge\tau_n))$
     and $\tilde{G}(t)=1_{(0,\tau_n]}(t)G(t,\mathfrak{u}(t\wedge\tau_n-),z)$ for $t\geq0$. Thus by setting $\sigma=t\wedge\tau_n$ and using the results achieved in step 1, we infer that
  \begin{align*}
      & \E V(\mathfrak{u}(t\wedge\tau_n))\\
       &\leq \E V(\mathfrak{u}_0)+\E\int_0^{t\wedge\tau_n}\Big{[}\langle \pi_2\mathfrak{u}(s),\pi_2\tilde{F}(s)\rangle_{\mathcal{H}}+m(\|B^{\frac{1}{2}}\pi_1\mathfrak{u}(s)\|^2_H)\langle B\pi_1\mathfrak{u}(s),\pi_1\mathcal{A}\mathfrak{u}(s)\rangle\Big{]}ds\\
       &\hspace{1.5cm}+\frac{1}{2}\E\int_0^{\sigma}\int_Z\|\tilde{g}(s,z)\|^2_H\nu(dz)\,ds\\
       &=\E V(\mathfrak{u}_0)+\E\int_0^{t\wedge\tau_n}\Big{[}-m(\|B^{\frac{1}{2}}u(s\wedge\tau_n)\|^2_H)\langle u_t(s),Bu(s\wedge\tau_n)\rangle1_{(0,\tau_n]}(s)\\
       &\hspace{3cm}-\langle u_t(s),f(\mathfrak{u}(s\wedge\tau_n)\rangle_{\mathcal{H}}1_{(0,\tau_n]}(s)+m(\|B^{\frac{1}{2}}u(s)\|^2_H)\langle Bu(s),u_t(s)\rangle\Big{]}ds\\
       &\hspace{1.5cm}+\frac{1}{2}\E\int_0^{t\wedge\tau_n}\int_Z\|g(s,\mathfrak{u}(s\wedge\tau_n-),z)\|^2_H1_{(0,\tau_n]}(t)\nu(dz)\,ds\\
       &=\E V(\mathfrak{u}_0)-\E\int_0^{t\wedge\tau_n}\langle u_t(s),f(\mathfrak{u}(s\wedge\tau_n)\rangle_{\mathcal{H}}ds+\frac{1}{2}\E\int_0^{t\wedge\tau_n}\int_Z\|g(s,\mathfrak{u}(s-),z)\|^2_H\nu(dz)\,ds           \end{align*}
           \begin{align*}
       &=\E V(\mathfrak{u}_0)-\E \int_0^t\langle u_t(s),f(\mathfrak{u}(s))\rangle_H1_{(0,{t\wedge\tau_n}]}(s)\,ds\\
       &\hspace{3cm}+\frac{1}{2}\E\int_0^{t}\int_Z\|g(s,\mathfrak{u}(s-),z)\|_{\mathcal{H}}^21_{(0,{t\wedge\tau_n}]}(s)\nu(dz)\,ds\\
       &\leq \E V(\mathfrak{u}_0)+\frac{1}{2}(1+K_f)\E\int_0^{t}(1+\|\mathfrak{u}(s\wedge\tau_n)\|_{\mathcal{H}}^2)\,ds+\frac{1}{2}K_g\E\int_0^{t\wedge\tau_n}(1+\|\mathfrak{u}(s\wedge\tau_n)\|_{\mathcal{H}}^2)\,ds\\
       &=\E V(\mathfrak{u}_0)+\frac{1}{2}(1+K_f+K_g)\int_0^{t}(1+\E\|\mathfrak{u}(s\wedge\tau_n)\|_{\mathcal{H}}^2)\,ds.
  \end{align*}
  This finally proves inequality \eqref{eq350}.
  In conclusion, we proved that $V$ is indeed a Lyapunov function and hence we can apply Lemma \ref{lem: Khasminskii test} to deduce that $\tau_{\infty}=\infty$.
\end{proof}

\begin{proof}[Proof of Theorem \ref{theo: stability}]
Define a new Lyapunov function in terms of operator $P$ by
\begin{align*}
        \Phi(x)=\frac{1}{2}\langle
        Px,x\rangle_{\mathcal{H}}+M(\|B^{\frac{1}{2}}x_1\|^2_{H}),\
        x\in\mathcal{H}.
\end{align*}
Since $m\in \mathcal{C}^1$ and $P\in\mathcal{L}(\mathcal{H})$, we infer that
$\Phi\in\mathcal{C}^2(\mathcal{H})$. Under Assumptions \eqref{assu:
growth condition} and \eqref{assu: locally Lipschitz condition},
Theorems \ref{theo: existence and uniqueness of locally mild
solution} and \ref{theo: lifespan} imply that Equation \eqref{SDE1}
has a unique global mild solution $\mathfrak{u}(t)$, $t\geq 0$ given
by \eqref{locally mild solution}, i.e.
     \begin{align*}
     \mathfrak{u}(t\wedge\tau_n)=e^{t\mathcal{A}}\mathfrak{u}_0+\int_0^{t\wedge\tau_n}e^{(t\wedge\tau_n-s)\mathcal{A}}F(s,\mathfrak{u}(s))\,ds+I_{\tau_n}(G(\mathfrak{u}))(t\wedge\tau_n) \ \text{a.s.,
     }t\geq 0.
  \end{align*}
where the process $I_{\tau_n}(G(\mathfrak{u}))$ is defiend by \eqref{eqn-I_tau_n}, i.e.
\begin{align*}I_{\tau_n}(G(\mathfrak{u}))(t)=\int_0^{t}\int_Z1_{(0,\tau_n]}e^{(t-s)\mathcal{A}}G(s,\mathfrak{u}(s-),z)\tilde{N}(ds,dz),\
t\geq0
\end{align*}
 and $\{\tau_n\}_{n\in\mathbb{N}}$ is an accessible sequence and
$\lim_{n\rightarrow\infty}\tau_n=\tau_{\infty}=\infty$. We have
already seen in the proof of Theorem \ref{theo: lifespan} that the
idea of getting an estimate for our Lyapunov function with a mild
solution is to approximate the mild solution by a sequence of strong
solutions to which we can apply the It\^{o} formula. We shall examine
the new Lyapunov function $\Phi$ in the same way as we did for $V$. Let $n$
be fixed. We first define functions $\tilde{F}$ and $\tilde{G}$ by the following formulae
\begin{align*}
      \tilde{F}(t)&=1_{(0,\tau_n]}(t)F(\mathfrak{u}(t\wedge\tau_n))\\
      &=\left(\begin{array}{c} 0 \\ -\tilde{f}(t) -m(\|B^{\frac{1}{2}}u(t\wedge\tau_n)\|^2_H)Bu(t\wedge\tau_n)1_{(0,\tau_n]}(t) \end{array}\right),\ t\in[0,T],\\
      \tilde{G}(t)&=1_{(0,\tau_n]}(t)G(t,\mathfrak{u}(t\wedge\tau_n-),z)\\
      &=\left(\begin{array}{c} 0 \\ \tilde{g}(t,z)
      \end{array}\right),\ t\in[0,T].
\end{align*}
Here $\tilde{f}(t)=-1_{(0,\tau_n]}(t)\beta u(t\wedge\tau_n)$,
$t\geq0$ and
$\tilde{g}(t,z)=1_{(0,\tau_n]}(t)g(t,\mathfrak{u}(t\wedge\tau_n-),z)$,
$t\geq0$. Then the following Equation
\begin{align}
\begin{split}
  dv(t)&=\mathcal{A}v(t)dt+\tilde{F}(t)dt+\int_Z\tilde{G}(t,z)\tilde{N}(dt,dz),\; t\geq 0\\
  v(0)&=\mathfrak{u}(0)
\end{split}
\end{align}
has a unique global mild solution which satisfies
\begin{align}
       v(t)=e^{t\mathcal{A}}\mathfrak{u}(0)+\int_0^te^{(t-s)\mathcal{A}}\tilde{F}(s)\,ds+\int_0^{t}\int_Ze^{(t-s)\mathcal{A}}\tilde{G}(s,z)\tilde{N}(ds,dz),\ \mathbb{P}\text{-a.s.}, t\geq0.
\end{align}
Since $\mathfrak{u}$ is the local mild solution, so $\mathfrak{u}$
satisfies \eqref{locally mild solution},  a similar argument used in
the proof of Theorem \ref{theo: lifespan} yields that for each
$n\in\mathbb{N}$
\begin{align*}
     v(t\wedge\tau_n)=\mathfrak{u}(t\wedge\tau_n)\ \ \mathbb{P}\text{-a.s. }t\geq0.
\end{align*}
Set
\begin{align*}
\mathfrak{u}_m(0)&=mR(m;\mathcal{A})\mathfrak{u}(0)\\
\tilde{F}_m(t,\omega)&=mR(m;\mathcal{A})\tilde{F}(t,\omega)\
\text{for}\  (t,\omega)\in\mathbb{R}_+\times\Omega;\\
\tilde{G}_m(t,\omega,z)&=mR(m;\mathcal{A})\tilde{G}(t,\omega,z)\text{
for }(t,\omega,z)\in\mathbb{R}_+\times\Omega\times Z. \end{align*}

In exactly the same manner as in the proof of Theorem \ref{theo:
lifespan} we infer that $\mathbb{P}$-a.s.
\begin{align}
     &\lim_{m     \rightarrow  \infty}
     \int_0^T\|\tilde{F}_m(t)-\tilde{F}(t)\|_{\mathcal{H}}^2\,dt=0, \label{eq44}\\
   &\lim_{m
     \rightarrow
     \infty}\int_0^T\int_Z|\tilde{G}_m(t,z)-\tilde{G}(t,z)|^2_{\mathcal{H}}\nu(dz)dt=0.\label{eq45}
\end{align}
Also, we find out that
$\tilde{G}_m\in\mathcal{M}^2([0,T]\times\Omega\times
Z,\hat{\mathcal{P}},\lambda\otimes\mathbb{P}\times\nu;\mathcal{D}(\mathcal{A}))$.
 By using the Theorem \ref{theo: exsitence and uniqueness of mild
solution}, one can see that the equation
\begin{align*}
     dv_m(t)&=\mathcal{A}v_m(t)dt+\tilde{F}_m(t)dt+\int_Z\tilde{G}_m(t,z)\tilde{N}(dt,dz)\\
      v_m(0)&=\mathfrak{u}(0)
\end{align*}
has a unique strong solution $v_m$  given by
\begin{equation}
    v_m(t)=\mathfrak{u}_m(0)+\int_0^t\Big{[}\mathcal{A}v_m(s)+\tilde{F}_m(s)\Big{]}ds+\int_0^{t}\int_Z\tilde{G}_m(ss,z)\tilde{N}(ds,dz)\
    \mathbb{P}\text{-a.s.}, t\geq0.
\end{equation}
Equivalently, we can also write the solution in the mild form
\begin{align}
     v_m(t)=e^{t\mathcal{A}}\mathfrak{u}(0)+\int_0^te^{(t-s)\mathcal{A}}\tilde{F}_m(s)\,ds+\int_0^{t}\int_Ze^{(t-s)\mathcal{A}}\tilde{G}_m(s,z)\tilde{N}(ds,dz),\
     \mathbb{P}\text{-a.s.}, t\geq0.
\end{align}
Now applying the It\^{o} Formula, see \cite{[Zhu_2010]}, to function $\Phi(x)e^{\lambda t}$
and the strong solution $v_m$ yields
\begin{align}\label{eq15}
\begin{split}
      &\Phi(v_m(t))e^{\lambda t}\\
      &=\Phi(v_m(s))e^{\lambda s}+\int_s^te^{\lambda r}\Big{[}\lambda \Phi(v_m(r))+\langle D\Phi(v_m(r)),\mathcal{A}v_m(r)+\tilde{F}_m(r)\rangle_{\mathcal{H}}\Big{]}dr\\
      &+\int_s^t\int_Ze^{\lambda r}\Big{[}\Phi(v_m(r)+\tilde{G}_m(r,z))-\Phi(v_m(r))-\langle D\Phi(v_m(s),\tilde{G}_m(r,z))\rangle_{\mathcal{H}}\Big{]}\nu(dz)dr\\
      &+\int_s^{t}\int_Ze^{\lambda t}\Big{[}\Phi(v_m(r-)+\tilde{G}_m(r,z))-\Phi(v_m(r-))\Big{]}\tilde{N}(dr,dz).
      \end{split}
\end{align}
We first find the following fact
\begin{align*}
     D\Phi(x)h&=\langle Ph,x\rangle_{\mathcal{H}}+2m(\|B^{\frac{1}{2}}x_1\|^2_H)\langle B^{\frac{1}{2}}x_1,B^{\frac{1}{2}}h_1\rangle\\
     &=\langle Ph,x\rangle_{\mathcal{H}}+2m(\|B^{\frac{1}{2}}x_1\|^2_H)\langle\left(\begin{array}{c} A^{-2}Bx_1 \\ 0 \end{array}\right),\left(\begin{array}{c} h_1 \\ h_2 \end{array}\right)\rangle_{\mathcal{H}},
\end{align*}
where $x=(x_1,x_2)^{\top}$, $h=(h_1,h_2)^{\top}$ and
$k=(k_1,k_2)^{\top}$ are all in $\mathcal{H}$. One can also rewrite
the derivative $D\Phi$ as follows
\begin{align*}
   D\Phi(x)&=Px+2m(\|B^{\frac{1}{2}}x_1\|^2_H)\left(\begin{array}{c} A^{-2}Bx_1 \\ 0 \end{array}\right)\ x
   \in\mathcal{H}.
\end{align*}
We adopt the projections $\pi_1$ and $\pi_2$ which are defined in
the proof of Theorem \ref{theo: lifespan}.

Therefore, by using above derivative formula we get for $r\in[0,T]$, 
\begin{align}\label{eq18}
      &\langle D\Phi(v_m(r)),\mathcal{A}v_m(r)+\tilde{F}_m(r)\rangle_{\mathcal{H}}\\
      &=\langle D\Phi(v_m(r)),\mathcal{A}v_m(r)\rangle_{\mathcal{H}}+\langle D\Phi(v_m(r)),\tilde{F}_m(r)\rangle_{\mathcal{H}}\nonumber\\
      &=\langle Pv_m(r)+2m(\|B^{\frac{1}{2}}\pi_1 v_m(r)\|^2_H)\left(\begin{array}{c} A^{-2}B\pi_1 v_m(r) \\ 0 \end{array}\right),\mathcal{A}v_m(r)\rangle_{\mathcal{H}}\nonumber\\
      &\hspace{1cm}+\langle Pv_m(r)+2m(\|B^{\frac{1}{2}}\pi_1 v_m(r)\|^2_H)\left(\begin{array}{c} A^{-2}B\pi_1 v_m(r) \\ 0 \end{array}\right),\tilde{F}_m(r)\rangle_{\mathcal{H}}\nonumber\\
      &=\langle Pv_m(r),\mathcal{A}v_m(r)\rangle_{\mathcal{H}}+2m(\|B^{\frac{1}{2}}\pi_1 v_m(r)\|^2_H)\langle B\pi_1 v_m(r),\pi_2v_m(r)\rangle_{H}\nonumber\\
      &\hspace{1cm}+\langle Pv_m(r),\tilde{F}_m(r)(r)\rangle_{\mathcal{H}}+2m(\|B^{\frac{1}{2}}\pi_1 v_m(r)\|^2_H)\langle B\pi_1 v_m(r),\pi_1\tilde{F}_m(r)\rangle_{H}.\nonumber
\end{align}
From Lemma \ref{lem: property of operator P} and the fact that
$A\geq\mu I$ for some $\mu>0$, we have
\begin{align*}
       &\langle Pv_m(r),\mathcal{A}v_m(r)\rangle_{\mathcal{H}}\\
       &=-\beta\|A\pi_1v_m(r)\|^2_H+\beta^2\langle \pi_1v_m(r),\pi_2v_m(r)\rangle+\beta\|\pi_2v_m(r)\|_H^2\\
&\leq-\beta\|A\pi_1v_m(r)\|^2_H+\beta^2\| \pi_1v_m(r)\|_H\|\pi_2v_m(r)\|_H+\beta\|\pi_2v_m(r)\|_H^2\\
&\leq -\beta\|A\pi_1v_m(r)\|^2_H+\frac{\beta^2}{2}(\| \pi_1v_m(r)\|_H^2+\|\pi_2v_m(r)\|_H^2)+\beta\|\pi_2v_m(r)\|_H^2\\
&\leq -\beta\|A\pi_1v_m(r)\|^2_H+\frac{\beta^2}{2\mu^2}\| A\pi_1v_m(r)\|_H^2+\frac{\beta^2}{2}\|\pi_2v_m(r)\|_H^2+\beta\|\pi_2v_m(r)\|_H^2\\
&=\Big{(}\frac{\beta^2}{2\mu^2}-\beta\Big{)}\|A\pi_1v_m(r)\|^2_H+\Big{(}\frac{\beta^2}{2}+\beta\Big{)}\|\pi_2v_m(r)\|^2_H,\
r\geq 0.
\end{align*}
Recall that in the proof of Theorem \ref{theo: lifespan} we have
shown that for every $0<T<\infty$,
\begin{align}\label{eq23}
\lim_{m\rightarrow\infty}\E\sup_{0\leq t\leq
T}\|v_m(t)-v(t)\|^2_{\mathcal{H}}=0,
\end{align}
So there exists a subsequence, denoted also by $\{v_m(t)\}_{m\in\mathbb{N}}$ for simplicity, such that $v_m(t)\rightarrow v(t)$ uniformly on $[s,t]$ as $k\rightarrow\infty$ a.s.\\
Therefore,
\begin{align*}
\limsup_{m\rightarrow\infty}\langle Pv_m(r),\mathcal{A}v_m(r)\rangle_{\mathcal{H}}&\leq \limsup_{m\rightarrow\infty}\left(\frac{\beta^2}{2\mu^2}-\beta\right)\|A\pi_1v_m(r)\|^2_H+\left(\frac{\beta^2}{2}+\beta\right)\|\pi_2v_m(r)\|^2_H\\
&=\left(\frac{\beta^2}{2\mu^2}-\beta\right)\|A\pi_1v(r)\|^2_H+\left(\frac{\beta^2}{2}+\beta\right)\|\pi_2v(r)\|^2_H,\
r\in[s,t].
\end{align*}
Now by applying Fatou Lemma we infer
\begin{align*}
   &\limsup_{m\rightarrow\infty}\int_s^te^{\lambda r}\langle Pv_m(r),\mathcal{A}v_m(r)\rangle_{\mathcal{H}}dr\\
   &\leq \int_s^te^{\lambda r}\limsup_{m\rightarrow\infty}\langle \mathcal{A}v_m(r),Pv_m(r)\rangle_{\mathcal{H}}\nonumber\\
      &\leq\int_s^t\Big{[}\Big{(}\frac{\beta^2}{2\mu^2}-\beta\Big{)}\|A\pi_1v(r)\|^2_H+\Big{(}\frac{\beta^2}{2}+\beta\Big{)}\|\pi_2v(r)\|^2_H\Big{]}dr.
\end{align*}

Further, by above derivative formula of $D\Phi$ and definition of
Lyapunov function of $\Phi$ we get
\begin{align*}
   &\langle D\Phi(v_m(r)),\tilde{G}_m(r,z)\rangle_{\mathcal{H}}\\
   &\hspace{1cm}=\langle Pv_m(r),\tilde{G}_m(r,z)\rangle_{\mathcal{H}}
   +2m(\|B^{\frac{1}{2}}\pi_1v_m(r)\|_H^2)\langle  B\pi_1v_m(r),\pi_1\tilde{G}_m(r,z)\rangle_{\mathcal{H}},\ r\geq0.
\end{align*}
and
\begin{align}\label{eq20}
\begin{split}
    &\Phi(v_m(r)+\tilde{G}_m(r,z))-\Phi(v_m(r))\\
    &=\frac{1}{2}\langle     P\big{(}v_m(r)+\tilde{G}_m(r,z)\big{)},v_m(r)+\tilde{G}_m(r,z)\rangle_{\mathcal{H}}+M\big{(}\|B^{\frac{1}{2}}\pi_1\big{(}v_m(r)+\tilde{G}_m(r,z)\big{)}\|\big{)}\\
    &\hspace{1cm}-\frac{1}{2}\langle Pv_m(r),v_m(r)\rangle_{\mathcal{H}}-M(\|B^{\frac{1}{2}}\pi_1v_m(r)\|^2_H)\\
    &=\frac{1}{2}\langle Pv_m(r),\tilde{G}_m(r,z)\rangle_{\mathcal{H}}+\frac{1}{2}\langle P\tilde{G}_m(r,z),v_m(r)\rangle_{\mathcal{H}}\\
    &\hspace{1cm}+\frac{1}{2}\langle P\tilde{G}_m(r,z),\tilde{G}_m(r,z)\rangle_{\mathcal{H}}+M(\|B^{\frac{1}{2}}\pi_1(v_m(r)+\tilde{G}_m(r,z))\|^2_H)\\
    &\hspace{1cm}-M(\|B^{\frac{1}{2}}\pi_1v_m(r)\|^2_H)\\
    &=\langle Pv_m(r),\tilde{G}_m(r,z)\rangle_{\mathcal{H}}+\frac{1}{2}\langle P\tilde{G}_m(r,z),\tilde{G}_m(r,z)\rangle_{\mathcal{H}}\\
    &\hspace{1cm}+M(\|B^{\frac{1}{2}}\pi_1(v_m(r)+\tilde{G}_m(r,z))\|^2_H)-M(\|B^{\frac{1}{2}}\pi_1v_m(r)\|^2_H),\ r\geq0.
    \end{split}
\end{align}
Combining above tow equalities, we find
\begin{align*}
     &\Phi(v_m(r)+\tilde{G}_m(r,z))-\Phi(v_m(r))-\langle D\Phi(v_m(r)),\tilde{G}_m(r,z)\rangle_{\mathcal{H}}\\
&\hspace{1cm}=\frac{1}{2}\langle P\tilde{G}_m(r,z),\tilde{G}_m(r,z)\rangle_{\mathcal{H}}+M(\|B^{\frac{1}{2}}\pi_1(v_m(r)+\tilde{G}_m(r,z))\|^2_H)\nonumber\\
    &\hspace{2cm}-M(\|B^{\frac{1}{2}}\pi_1v_m(r)\|^2_H)-2m(\|B^{\frac{1}{2}}\pi_1v_m(r)\|_H^2)\langle  B\pi_1v_m(r),\pi_1\tilde{G}_m(r,z)\rangle_{\mathcal{H}},
\end{align*}
which converges $\mathbb{P}$-a.s. to
\begin{align*}
      \Phi(v(r)+\tilde{G}(r,z))-\Phi(v(r))-\langle D\Phi(v(r)),\tilde{G}(r,z)\rangle_{\mathcal{H}}=\frac{1}{2}\langle P\tilde{G}(r,z),\tilde{G}(r,z)\rangle_{\mathcal{H}},
\end{align*}
as $m\rightarrow\infty$, $r\geq0$.
Also, we find that
\begin{align*}
  \Phi(v_m(r)&+\tilde{G}_m(r,z))-\Phi(v_m(r))\\
  &=\langle Pv_m(r),\tilde{G}_m(r,z)\rangle_{\mathcal{H}}+\frac{1}{2}\langle P\tilde{G}_m(r,z),\tilde{G}_m(r,z)\rangle_{\mathcal{H}}\nonumber\\
    &\hspace{1cm}+M(\|B^{\frac{1}{2}}\pi_1(v_m(r)+\tilde{G}_m(r,z))\|^2_H)-M(\|B^{\frac{1}{2}}\pi_1v_m(r)\|^2_H),\ r\geq0.
\end{align*}
This converges $\mathbb{P}$-a.s. to
\begin{align*}
\Phi(v(r)+\tilde{G}(r,z))-\Phi(v(r))=\langle
Pv(r),\tilde{G}(r,z)\rangle_{\mathcal{H}}+\frac{1}{2}\langle
P\tilde{G}(r,z),\tilde{G}(r,z)\rangle_{\mathcal{H}},
\end{align*}
as $m\rightarrow\infty$, $r\geq0$.
On the other hand, since the function $\Phi$ is in
$\mathcal{C}^2(\mathcal{H})$, by the Taylor formula we infer that
\begin{align*}
       \Phi(v_m(r)+\tilde{G}_m(r,z))-\Phi(v_m(r))&\leq \sup_{x\in X}\|D \Phi(x)\|\|\tilde{G}_m(r,z)\|_{\mathcal{H}},\ r\in[s,t]
\end{align*}
and
\begin{align*}
       \Phi(v_m(r) +\tilde{G}_m(r,z))-\Phi(v_m(r))&- \langle D\Phi(v_m(r)),\tilde{G}_m(r,z)\rangle_{\mathcal{H}}\\
       &\leq \frac{1}{2}\|D^2\Phi(v_m(r))\|\|\tilde{G}_m(r,z)\|^2_{\mathcal{H}}\\
&\leq \frac{1}{2}\sup_{x\in X}\|D^2
\Phi(x)\|\|\tilde{G}_m(r,z)\|^2_{\mathcal{H}},
\end{align*}
where we used the uniformly boundedness of
$\{v_m\}_{m\in\mathbb{N}}$ on $[s,t]$. Hence it follows from the
Lebesgue Dominated Convergence Theorem that for $0\leq s\leq t\leq
\infty$,
\begin{align*}
  \int_s^t\int_Ze^{\lambda r}\Big{[}\Phi(v_m(r)+\tilde{G}_m(r,z))-\Phi(v_m(r))-\langle D\Phi(v_m(s),\tilde{G}_m(r,z))\rangle_{\mathcal{H}}\Big{]}\nu(dz)dr\nonumber
\end{align*}
converges $\mathbb{P}$-a.s. to
\begin{align*}
  \int_s^t\int_Z\frac{e^{\lambda r}}{2}\langle P\tilde{G}(r,z),\tilde{G}(r,z)\rangle_{\mathcal{H}}\nu(dz)dr.
\end{align*}
On the basis of the It\^{o} isometry for stochastic integral w.r.t.
compensated Poisson random measure, we obtain
\begin{align*}
        \E\Big{\|}\int_s^{t}\int_Z&e^{\lambda t}\Big{[}\Phi(v_m(r-)+\tilde{G}_m(r,z))-\Phi(v_m(r-))\Big{]}\tilde{N}(dr,dz)\\
&-\int_s^{t}\int_Ze^{\lambda t}\Big{[}\langle Pv(r-),\tilde{G}(r,z)\rangle_{\mathcal{H}}+\frac{1}{2}\langle P\tilde{G}(r,z),\tilde{G}(r,z)\rangle_{\mathcal{H}}\Big{]}\tilde{N}(dr,dz)\Big{\|}^2\\
&\leq \E\int_s^t\int_Ze^{2\lambda}\Big{\|}\Phi(v_m(r)+\tilde{G}_m(r,z))-\Phi(v_m(r))-\langle Pv(r),\tilde{G}(r,z)\rangle_{\mathcal{H}}\\
&-\frac{1}{2}\langle
P\tilde{G}(r,z),\tilde{G}(r,z)\rangle_{\mathcal{H}}\Big{\|}^2_{\mathcal{H}}\nu(dz)dr.
\end{align*}
Note that the integrand on the right side of above equality is
dominated by $$2\sup_{x\in X}\|D\Phi(x)\|^2\|\tilde{G}(s,z)\|^2,$$
where $X$ is a compact set on $\mathcal{H}$. Again, by passing to
the limit as $m\rightarrow\infty$, the Lebesgue Dominated
Convergence Theorem tells us that the right-side of above
equality converges to $0$. Hence, by taking a subsequence we infer
that
\begin{align*}
\int_s^{t}\int_Z&e^{\lambda
t}\Big{[}\Phi(v_m(r-)+\tilde{G}_m(r-,z))-\Phi(v_m(r-))\Big{]}\tilde{N}(dr,dz)
\end{align*}
converges $\mathbb{P}$-a.s. to
\begin{align*}
   \int_s^{t}\int_Ze^{\lambda t}\Big{[}\langle Pv(r-),\tilde{G}(r-,z)\rangle_{\mathcal{H}}+\frac{1}{2}\langle P\tilde{G}(r-,z),\tilde{G}(r-,z)\rangle_{\mathcal{H}}\Big{]}\tilde{N}(dr,dz)\ \text{as }m\rightarrow\infty.
\end{align*}
Combining all the observations together and letting
$m\rightarrow\infty$ yields that for $0\leq s\leq  t<\infty$,  $\mathbb{P}\text{-a.s.}$,
\begin{eqnarray*}
    \Phi(v(t))e^{\lambda t}&\leq &
    \Phi(v(s))e^{\lambda s}+\int_s^te^{\lambda r}\Big{[}\lambda \Phi(v(r))+\Big{(}\frac{\beta^2}{2\mu^2}-\beta\Big{)}\|A\pi_1v(r)\|^2_H\\
    &+&\Big{(}\frac{\beta^2}{2}+\beta\Big{)}\|\pi_2v(r)\|^2_H+ 2m(\|B^{\frac{1}{2}}\pi v(r)\|^2_H)\langle B\pi_1 v(r),\pi_2v(r)\rangle_{H}\\
    &-&\langle\beta\pi_1v(r)+2\pi_2v(r),\pi_2\tilde{F}(r)\rangle_{H}\Big{]}dr
    +\int_s^t\int_Ze^{\lambda r}\|\tilde{g}(r,z)\|^2_H\nu{dz}dr
   \\ &+&\int_s^{t+}\int_Ze^{\lambda r}\Big{[}\langle \beta \pi_1v(r-)+2\pi_2v(r-),\tilde{g}(r,z)\rangle_{\mathcal{H}}+\|\tilde{g}(r,z)\|^2_H\Big{]}\tilde{N}(dr,dz).
\end{eqnarray*}
Recall that for every $n\in\mathbb{N}$,
$v(t\wedge\tau_n)=\mathfrak{u}(t\wedge\tau_n)$ $\mathbb{P}$-a.s., by
replacing $t$ by $t\wedge\tau_n$ in above inequality we have, for every $0\leq s\leq t<\infty$,
\begin{align*}
    &\Phi(\mathfrak{u}(t\wedge\tau_n))e^{\lambda (t\wedge\tau_n)}\\
    &\leq\Phi(\mathfrak{u}(s))+\int_s^{t\wedge\tau_n}e^{\lambda r}\Big{[}\lambda \Phi(\mathfrak{u}(r))+\Big{(}\frac{\beta^2}{2\mu^2}-\beta\Big{)}\|Au(r)\|^2_H+\Big{(}\frac{\beta^2}{2}+\beta\Big{)}\|u_t(r)\|^2_H\\
    &\hspace{3cm}+2m(\|B^{\frac{1}{2}}u(r)\|^2_H)\langle Bu(r),u_t(r)\rangle_{H}-\langle\beta u(r)+2u_t(r),\beta u_t(r)\rangle_{H}\\
    &\hspace{3cm}-m(\|B^{\frac{1}{2}}u(r)\|^2_H)\langle\beta u(r)+2u_t(r),Bu(r)\rangle_{H}\Big{]}dr\\
    &\hspace{0.5cm}+\int_s^{t\wedge\tau_n}\int_Ze^{\lambda r}\|g(r,\mathfrak{u}(r),z)\|_H^2\nu(dz)dr\\
    &\hspace{0.5cm}+\int_s^{t\wedge\tau_n+}\int_Ze^{\lambda r}\Big{[}\langle \beta \pi_1\mathfrak{u}(r-)+2\pi_2\mathfrak{u}(r-),g(r,\mathfrak{u}(r),z)\rangle_{\mathcal{H}}+\|g(r,\mathfrak{u}(r),z)\|^2_H\Big{]}\tilde{N}(dr,dz)
    \end{align*}
    \begin{align*}
    &=\Phi(\mathfrak{u}(s))+\int_s^{t\wedge\tau_n}e^{\lambda r}\Big{[}\lambda \Phi(\mathfrak{u}(r))+\Big{(}\frac{\beta^2}{2\mu^2}-\beta\Big{)}\|Au(r)\|^2_H+\Big{(}\frac{\beta^2}{2}-\beta\Big{)}\|u_t(r)\|^2_H\\
    &\hspace{3cm}-\beta^2\langle u(r), u_t(r)\rangle_{H}-m(\|B^{\frac{1}{2}}u(r)\|^2_H)\langle\beta u(r),Bu(r)\rangle_{H}\Big{]}dr\\
    &\hspace{0.5cm}+\int_s^{t\wedge\tau_n}\int_Ze^{\lambda r}\|g(r,\mathfrak{u}(r),z)\|^2_H\nu(dz)dr\\
    &\hspace{0.5cm}+\int_s^{t\wedge\tau_n+}\int_Ze^{\lambda r}\Big{[}\langle \beta \pi_1\mathfrak{u}(r-)+2\pi_2\mathfrak{u}(r-),g(r,\mathfrak{u}(r),z)\rangle_{\mathcal{H}}+\|\tilde{g}(r,z)\|^2_H\Big{]}\tilde{N}(dr,dz)\\
        &\leq \Phi(\mathfrak{u}(s))+\int_s^{t\wedge\tau_n}e^{\lambda r}\Big{[}\lambda \Phi(\mathfrak{u}(r))+(2C\beta^2-\beta)\|\mathfrak{u}(r)\|^2_{\mathcal{H}}\\
        &\hspace{3cm}-\beta m(\|B^{\frac{1}{2}}u(r)\|^2_H)\|B^{\frac{1}{2}}u(r)\|^2_H\Big{]}dr\\
    &\hspace{0.5cm}+\int_s^{t\wedge\tau_n}\int_Ze^{\lambda r}\|g(r,\mathfrak{u}(r),z)\|^2_H\nu{dz}dr\\
    &\hspace{0.5cm}+\int_s^{t\wedge\tau_n+}\int_Ze^{\lambda r}\Big{[}\langle \beta \pi_1\mathfrak{u}(r-)+2\pi_2\mathfrak{u}(r-),g(r,\mathfrak{u}(r),z)\rangle_{\mathcal{H}}+\|g(r,\mathfrak{u}(r),z)\|^2_H\Big{]}\tilde{N}(dr,dz),
\end{align*}
where $C=\max\{\frac{1}{2\mu^2},\frac{1}{2}\}$. Now applying part
$(3)$ of the Assumption \eqref{assu: stability} and the definition
of the function $\Phi$ yields that for $0\leq s\leq t<\infty$,
\begin{align*}
    &\Phi(\mathfrak{u}(t\wedge\tau_n))e^{\lambda (t\wedge\tau_n)}\\
    &\leq\Phi(\mathfrak{u}(s))+\int_s^{t\wedge\tau_n}e^{\lambda r}\Big{[}\lambda \Phi(\mathfrak{u}(r))+(2C\beta^2-\beta)\|\mathfrak{u}(r)\|^2_{\mathcal{H}}-\beta m(\|B^{\frac{1}{2}}u(r)\|^2_H)\|B^{\frac{1}{2}}u(r)\|^2_H\\
    &\hspace{4cm}+R^2_g\|\mathfrak{u}(r)\|^2_{\mathcal{H}}+K\Big{]}dr\\
    &\hspace{1cm}+\int_s^{t\wedge\tau_n}\int_Ze^{\lambda r}\Big{[}\langle \beta \pi_1\mathfrak{u}(r-)+2\pi_2\mathfrak{u}(r-),g(r,\mathfrak{u}(r),z)\rangle_{\mathcal{H}}+\|g(r,\mathfrak{u}(r),z)\|^2_H\Big{]}\tilde{N}(dr,dz)
               \end{align*}
           \begin{align*}
    &=\Phi(\mathfrak{u}(s))+\int_s^{t\wedge\tau_n}e^{\lambda r}\Big{[}\frac{\lambda }{2}\langle P\mathfrak{u}(r),\mathfrak{u}(r)\rangle_{\mathcal{H}}+\lambda M(\|B^{\frac{1}{2}}u(r)\|^2_H)+(R_g^2+2C\beta^2-\beta)\|\mathfrak{u}(r)\|^2_{\mathcal{H}}\\
    &\hspace{4cm}-\beta m(\|B^{\frac{1}{2}}u(r)\|^2_H)\|B^{\frac{1}{2}}u(r)\|^2_H+K\Big{]}dr\\
    &\hspace{1cm}+\int_s^{t\wedge\tau_n}\int_Ze^{\lambda r}\Big{[}\langle \beta \pi_1\mathfrak{u}(r-)+2\pi_2\mathfrak{u}(r-),g(r,\mathfrak{u}(r),z)\rangle_{\mathcal{H}}+\|g(r,\mathfrak{u}(r),z)\|^2_H\Big{]}\tilde{N}(dr,dz)\\
    &\leq\Phi(\mathfrak{u}(s))+\int_s^{t\wedge\tau_n}e^{\lambda r}\Big{[}\Big{(}\frac{\lambda }{2}\|P\|_{\mathcal{L}(H)}+R_g^2+2C\beta^2-\beta\Big{)}\|\mathfrak{u}(r)\|^2_{\mathcal{H}}\\
    &\hspace{4cm}+\big{(}\frac{\lambda}{\alpha}-\beta\big{)} m(\|B^{\frac{1}{2}}u(r)\|^2_H)\|B^{\frac{1}{2}}u(r)\|^2_H+K\Big{]}dr\\
    &\hspace{1cm}+\int_s^{t\wedge\tau_n}\int_Ze^{\lambda r}\Big{[}\langle \beta \pi_1\mathfrak{u}(r-)+2\pi_2\mathfrak{u}(r-),g(r,\mathfrak{u}(r),z)\rangle_{\mathcal{H}}+\|g(r,\mathfrak{u}(r),z)\|^2_H\Big{]}\tilde{N}(dr,dz),
\end{align*}
where in the last inequality we used the following inequality $$\langle Px,x\rangle_H\leq\|Px\|_{\mathcal{L}(H)}\|x\|^2_H.$$
Now let $n\rightarrow\infty$. Since by Theorem \ref{theo: lifespan},
$\tau_{\infty}=\infty$, we have for $0\leq s\leq t<\infty$,
\begin{align*}
    \Phi(\mathfrak{u}(t))e^{\lambda t}
    &\leq\Phi(\mathfrak{u}(s))+\int_s^{t}e^{\lambda r}\Big{[}\Big{(}\frac{\lambda }{2}\|P\|_{\mathcal{L}(H)}+R_g^2+2C\beta^2-\beta\Big{)}\|\mathfrak{u}(r)\|^2_{\mathcal{H}}\\
    &\hspace{4cm}+\big{(}\frac{\lambda}{\alpha}-\beta\big{)} m(\|B^{\frac{1}{2}}u(r)\|^2_H)\|B^{\frac{1}{2}}u(r)\|^2_H+K\Big{]}dr\\
    &\hspace{0.5cm}+\int_s^{t+}\int_Ze^{\lambda r}\Big{[}\langle \beta \pi_1v(r-)+2\pi_2v(r-),\tilde{g}(r,z)\rangle_{\mathcal{H}}+\|\tilde{g}(r,z)\|^2_H\Big{]}\tilde{N}(dr,dz).
\end{align*}
Choose $\lambda$ such that
$0<\lambda<2\|P\|^{-1}_{\mathcal{L}(H)}(\beta-2C\beta^2-R^2_g)\wedge\alpha\beta$.
It follows that
\begin{align*}
\frac{\lambda }{2}\|P\|_{\mathcal{L}(H)}+R_g^2+2C\beta^2-\beta<0\
\text{and} \ \frac{\lambda}{2}-\beta<0.
\end{align*}
Therefore, we infer that for $0\leq s\leq t<\infty$,
\begin{align}\label{eq25}
    \Phi(\mathfrak{u}(t))e^{\lambda t}\leq&\Phi(\mathfrak{u}(s))+\int_s^te^{\lambda r}Kdr\\
    &+\int_s^{t}\int_Ze^{\lambda r}\Big{[}\langle \beta \pi_1v(r-)+2\pi_2v(r-),\tilde{g}(r,z)\rangle_{\mathcal{H}}+\|\tilde{g}(r,z)\|^2_H\Big{]}\tilde{N}(dr,dz).\nonumber
\end{align}
First consider the case when $K=0$. Then equality \eqref{eq25}
becomes,
\begin{eqnarray}\nonumber \label{eq26}
  \Phi(\mathfrak{u}(t))e^{\lambda t}& \leq &\Phi(\mathfrak{u}(s))+\int_s^{t}\int_Ze^{\lambda r}\Big{[}\langle \beta \pi_1v(r-)\\
  &+&2\pi_2v(r-),\tilde{g}(r,z)\rangle_{\mathcal{H}}+\|\tilde{g}(r,z)\|^2_H\Big{]}\tilde{N}(dr,dz),\;\; 0\leq s\leq t<\infty.
\end{eqnarray}
Taking conditional expectation with respect to $\mathcal{F}_s$ to
both sides yields
\begin{eqnarray*}
      \E\big{(} \Phi(\mathfrak{u}(t))e^{\lambda t}\big{|}\mathcal{F}_s\big{)}&\leq& \E\big{(}\Phi(\mathfrak{u}(s))\big{|}\mathcal{F}_s\big{)}
      +\E\Big{(}\int_s^{t}\int_Ze^{\lambda r}\Big{[}\langle \beta\pi_1v(r-)\\
      &+&2\pi_2v(r-),\tilde{g}(r,z)\rangle_{\mathcal{H}}+\|\tilde{g}(r,z)\|^2_H\Big{]}\tilde{N}(dr,dz)\Big{|}\mathcal{F}_s\Big{)}\\
      &=&\Phi(\mathfrak{u}(s)),\ 0\leq s\leq t<\infty,
\end{eqnarray*}
where the equality follows from the measurability of $\Phi(\mathfrak{u}(s))$ with respect to $\mathcal{F}_s$ and independence of the integrals with respect to $\mathcal{F}_s$. This means that the process $\Phi(\mathfrak{u}(t))e^{\lambda t}$ is a supermartingale.\\
Take $\lambda^\ast\in(0,\lambda)$. We observe that for every
$k=0,1,2\cdots$,
\begin{align*}
    \sup_{t\in[k,k+1]}e^{\lambda^\ast t}\Phi(\mathfrak{u}(t))=\sup_{t\in[k,k+1]}e^{(\lambda^\ast-\lambda)t}e^{\lambda t}\Phi(\mathfrak{u}(t))\leq e^{(\lambda^\ast-\lambda)k}\sup_{t\in[k,k+1]}e^{\lambda t}\Phi(\mathfrak{u}(t)).
\end{align*}
Therefore,
\begin{align*}
      \mathbb{P}\Big{\{}\sup_{t\in[k,k+1]}e^{\lambda^\ast t}\Phi(\mathfrak{u}(t))\geq\E\Phi(\mathfrak{u}(0))\Big{\}}&\leq
      \mathbb{P}\Big{\{}\sup_{t\in[k,k+1]}e^{\lambda t}\Phi(\mathfrak{u}(t))\geq e^{(\lambda-\lambda^\ast)k}\E\Phi(\mathfrak{u}(0))\Big{\}}\\
      &\leq\frac{\E\big{(} e^{\lambda k}\Phi(\mathfrak{u}(k))\big{)}}{e^{(\lambda-\lambda^\ast)k}\E\Phi(\mathfrak{u}(0))}\\
      &\leq\frac{\E\Phi(\mathfrak{u}(0))}{e^{(\lambda-\lambda^\ast)k}\E\Phi(\mathfrak{u}(0))}=e^{-(\lambda-\lambda^\ast)k}
\end{align*}
By the ratio test, we know that the series
$\sum_{k=1}^{\infty}e^{-(\lambda-\lambda^\ast)k}$ is convergent. Thus
\begin{align*}
         \sum_{k=1}^{\infty}\mathbb{P}\Big{\{}\sup_{t\in[k,k+1]}e^{\lambda^\ast t}\Phi(\mathfrak{u}(t))\geq\E\Phi(\mathfrak{u}(0))\Big{\}}
         \leq\sum_{k=1}^{\infty}e^{-(\lambda-\lambda^\ast)}k<\infty.
\end{align*}
Now by applying Borel-Cantelli Theorem, we have
\begin{align*}
      \mathbb{P}\Big{(}\bigcap_{j=1}^{\infty}\bigcup_{k\geq j}\Big{\{}\sup_{t\in[k,k+1]}e^{\lambda^\ast t}\Phi(\mathfrak{u}(t))\geq \E\Phi(\mathfrak{u}(0))\Big{\}}\Big{)}=0.
\end{align*}
It follows that
\begin{align*}
     \mathbb{P}\Big{(}\bigcup_{j=1}^{\infty}\bigcap_{k\geq j}\Big{\{}\sup_{t\in[k,k+1]}e^{\lambda^\ast t}\Phi(\mathfrak{u}(t))\geq \E\Phi(\mathfrak{u}(0))\Big{\}}\Big{)}=1.
\end{align*}
Therefore, there exists $j\in\mathbb{N}$ such that for every $k\geq
j$,
\begin{align*}
    \sup_{t\in[k,k+1]}e^{\lambda^\ast}\Phi(\mathfrak{u}(t))\leq\E\Phi(\mathfrak{u}(0))\ \ \mathbb{P}\text{-a.s.}
\end{align*}
Then we can infer that for every $t\geq j$
\begin{align*}
       e^{\lambda^\ast t}\Phi(\mathfrak{u}(t))\leq\E\Phi(\mathfrak{u}(0))\ \ \mathbb{P}\text{-a.s.}.
\end{align*}
It then follows that
\begin{align*}
       \E\|\mathfrak{u}(t)\|^2_H&\leq\E\langle P\mathfrak{u}(t),\mathfrak{u}(t)\rangle_{\mathcal{H}}\\
       &\leq2\E\Big{[}\frac{1}{2}\langle P\mathfrak{u}(t),\mathfrak{u}(t)\rangle_{\mathcal{H}}+M\big{(}\|B^{\frac{1}{2}}u(t)\|_H^2\big{)}\Big{]}\\
       &=\E\Phi(\mathfrak{u}(r))\\
       &\leq2e^{-\lambda^\ast t}\E\Phi(\mathfrak{u}(0)),
\end{align*}
where the first inequality follows from part $(1)$ of Lemma
\ref{lem: property of operator P}, the last inequality follows from
above result. Also, note that
\begin{align}\label{eq27}
      \E\Phi(\mathfrak{u}(0))&=\E\left[\frac{1}{2}\langle P\mathfrak{u}(0),\mathfrak{u}(0)\rangle_{\mathcal{H}}+M(\|B^{\frac{1}{2}}u(0)\|)\right]\nonumber\\
      &=\E\Big{[}\frac{1}{2}\| P\|_{\mathcal{L}(\mathcal{H})}\|\mathfrak{u}(0)\|^2_{\mathcal{H}}+M(\|B^{\frac{1}{2}}u(0)\|)\Big{]}\nonumber\\
      &\leq \Big{(}\frac{1}{2}\| P\|_{\mathcal{L}(\mathcal{H})}+1\Big{)}\E\left[\|\mathfrak{u}(0)\|^2_{\mathcal{H}}+M(\|B^{\frac{1}{2}}u(0)\|)\right]\nonumber\\
      &=\Big{(}\frac{1}{2}\| P\|_{\mathcal{L}(\mathcal{H})}+1\Big{)}\mE(\mathfrak{u}(0)).
\end{align}
Therefore, we conclude that
\begin{align*}
     \E\|\mathfrak{u}(t)\|^2_H\leq 2\Big{(}\frac{1}{2}\| P\|_{\mathcal{L}(\mathcal{H})}+1\Big{)}e^{-\lambda^\ast t}\mE(\mathfrak{u}(0)),\ t\geq0.
\end{align*}
Set $C=\| P\|_{\mathcal{L}(\mathcal{H})}+2$. In conclusion, we find out that
\begin{align*}
    \E\|\mathfrak{u}(t)\|^2_H\leq Ce^{-\lambda^\ast t}\mE(\mathfrak{u}(0)),\ t\geq0,
\end{align*}
which shows the exponentially mean-square stable of our mild solution.\\
For the case $K\neq0$, first taking expectation to both side of
\eqref{eq25} and setting $s=0$ gives
\begin{align*}
    \E\left(\Phi(\mathfrak{u}(t))e^{\lambda t}\right)\leq\E\Phi(\mathfrak{u}(s))+\frac{K}{\lambda}\big{(}e^{\lambda t}-1\big{)},\ 0\leq s\leq t<\infty.
\end{align*}
Thus
\begin{align*}
    \E\Phi(\mathfrak{u}(t))&\leq e^{-\lambda t}\E\Phi(\mathfrak{u}(s))+\frac{K}{\lambda}\big{(}1-e^{-\lambda t}\big{)}, \ 0\leq s\leq t<\infty.
\end{align*}
By the definition of function $\Phi$, we obtain
\begin{align*}
    \E\Big{(}\frac{1}{2}\langle P\mathfrak{u}(t),\mathfrak{u}(r)\rangle_{\mathcal{H}}\Big{)}+\E(M(\|B^{\frac{1}{2}}u(r)\|^2_H))&=\E\Phi(\mathfrak{u}(t))\\
    &\leq e^{-\lambda t}\E\Phi(\mathfrak{u}(0))+\frac{K}{\lambda}\big{(}1-e^{-\lambda t}\big{)},\ t\geq0.
\end{align*}
Thus applying the inequality $\|x\|^2_H\leq\langle x,Px\rangle_H $
from Lemma \ref{lem: property of operator P} gives that
\begin{align*}
      \E\|\mathfrak{u}(t)\|^2_H\leq \E\langle \mathfrak{u}(t),P\mathfrak{u}(t)\rangle_{\mathcal{H}}&\leq 2e^{-\lambda t}\E\Phi(\mathfrak{u}(0))+\frac{2K}{\lambda}\big{(}1-e^{-\lambda t}\big{)}\\
      &\leq 2e^{-\lambda t}\E\Phi(\mathfrak{u}(0))+\frac{2K}{\lambda},\ t\geq0.
\end{align*}
It then follows from inequality \eqref{eq27} that
\begin{align*}
       \E\|\mathfrak{u}(t)\|^2_H\leq 2e^{-\lambda t}\Big{(}\frac{1}{2}\| P\|_{\mathcal{L}(\mathcal{H})}+1\Big{)}\mE(\mathfrak{u}(0))+\frac{2K}{\lambda},\ t\geq0.
\end{align*}
Therefore,
\begin{align*}
        \sup_{t\geq0}E\|\mathfrak{u}(t)\|^2_H\leq \left(\| P\|_{\mathcal{L}(\mathcal{H})}+2\right)\mE(\mathfrak{u}(0))+\frac{2K}{\lambda}<\infty,
\end{align*}
which completes our proof of of Theorem \ref{theo: stability}.
\end{proof}

\section{Appendix}
Let $X=(X(t))_{t\geq0}$ be an $\mathcal{H}$-valued process.
Let $(e^{t\mathcal{A}})_{t\in\mathbb{R}}$ be a contraction
$C_0$-group.
Let $\varphi$ be an $\mathcal{H}$-valued process belonging to
$\mathcal{M}^2_{loc}(\hat{\mathcal{P}};\mathcal{H})$. Set
\begin{align*}
     I(t)&=\int_0^{t}\int_Ze^{(t-s)\mathcal{A}}\varphi(s,z)\tilde{N}(ds,dz),\ t\geq0,\\
     I_{\tau}(t)&=\int_0^{t}\int_Z1_{[0,\tau]}(s)e^{(t-s)\mathcal{A}}\varphi(s,z)\tilde{N}(ds,dz),\ t\geq0.
\end{align*}
By the choice of process $\varphi$,  Proposition
\ref{Lem:predictability} and the assumption about
$(e^{t\mathcal{A}})_{t\in\mathbb{R}}$, the stochastic convolution process $I(t)$,
$t\geq0$, is well defined. Also for any stopping time $\tau$, the process
$1_{[0,\tau]}(t,\omega)$ is predictable. In fact, the predictable
$\sigma$-field is generated by the family of closed stochastic
intervals $\{[0,T]:T\text{ is a stopping time}\}$, see \cite{[Metivier_1982]}. This together with the predictability of
$\varphi$ and Proposition \ref{Lem:predictability} implies that
integrand of $I_{\tau}(t)$ is predictable. Thus the stochastic
convolution $I_{\tau}(t)$ is well defined as well. Moreover, one can always assume that the stochastic convolution process $I(t)$, $t\geq0$ is c\`{a}dl\`{a}g, see \cite{[Brzezniak_Hau_Zhu]}. The following lemma, which was  first implicitely stated in \cite{Brz+Gat_1999} and explicitly the  Ph.D thesis \cite{[Carroll]} of Andrew Caroll, verifies the definition \eqref{locally mild solution} of a local  mild solution. The proof below is mainly based on \cite{Brz+Masl+S_2005} (which in turn was provided by Martin Ondrej\'at).

\begin{lem}\label{lem: stochastic convolution}For any stopping time $\tau$,
       \begin{align}\label{stochasitc convolution}
          e^{(t-t\wedge\tau)\mathcal{A}}I(t\wedge\tau)=I_{\tau}(t)
       \end{align}
       holds for all $t\geq 0$, $\mathbb{P}$-a.s.
\end{lem}

\begin{proof}
We first verify it for deterministic time. Let $\tau=a$. If $t<a$,
then
\begin{align*}
     e^{(t-t\wedge a)\mathcal{A}}I(t\wedge a)=e^{(t-t)\mathcal{A}}I(t)&=I(t)=\int_0^{T}\int_Z1_{[0,t]}e^{(t-s)\mathcal{A}}\varphi(s,z)\tilde{N}(ds,dz)\\
&=\int_0^{T}\int_Z1_{[0,t]}1_{[0,a]}e^{(t-s)\mathcal{A}}\varphi(s,z)\tilde{N}(ds,dz)\\
     &=\int_0^{t}\int_Z1_{[0,a]}e^{(t-s)\mathcal{A}}\varphi(s\wedge a,z)\tilde{N}(ds,dz)=I_{a}(t),
\end{align*}
where we used in the equality the fact that
$1_{[0,a]}(s)\varphi(s,z)=1_{[0,a]}(s)\varphi(s\wedge a,z)$. If
$t\geq a$, then
\begin{align*}
     e^{(t-t\wedge a)\mathcal{A}}I(t\wedge a)&=e^{(t-a)\mathcal{A}}I(a)=e^{(t-a)\mathcal{A}}\int_0^{a}\int_Z e^{(a-s)\mathcal{A}}\varphi(s,z)\tilde{N}(ds,dz)\\
     &=e^{(t-a)\mathcal{A}}\int_0^{T}\int_Z 1_{[0,a]}(s)e^{(a-s)\mathcal{A}}\varphi(s,z)\tilde{N}(ds,dz)\\
      &\hspace{2cm}+e^{(t-a)\mathcal{A}}\int_0^{T}\int_Z 1_{(a,t]}(s)1_{[0,a]}(s)e^{(a-s)\mathcal{A}}\varphi(s,z)\tilde{N}(ds,dz)\\
     &=e^{(t-a)\mathcal{A}}\int_0^{a}\int_Z 1_{[0,a]}(s)e^{(a-s)\mathcal{A}}\varphi(s,z)\tilde{N}(ds,dz)\\
     &\hspace{2cm}+e^{(t-a)\mathcal{A}}\int_a^{t}\int_Z 1_{[0,a]}(s)e^{(a-s)\mathcal{A}}\varphi(s,z)\tilde{N}(ds,dz)\\
     &=e^{(t-a)\mathcal{A}}\int_0^{t}\int_Z 1_{[0,a]}(s)e^{(a-s)\mathcal{A}}\varphi(s\wedge a,z)\tilde{N}(ds,dz)\\
     &=\int_0^{t}\int_Z 1_{[0,a]}(s)e^{(t-s)\mathcal{A}}\varphi(s,z)\tilde{N}(ds,dz)=I_{a}(t).
\end{align*}
Thus equality \eqref{stochasitc convolution} holds for any
deterministic time. Now let $\tau$ be an arbitrary stopping time.
Define $\tau_n:=2^{-n}([2^n\tau]+1)$, for each $n\in\mathbb{N}$.
That is $\tau_n=\frac{k+1}{2^n}$ if $\frac{k}{2^n}\leq
\tau<\frac{k+1}{2^n}$. Then $\tau_n$ converges down to $\tau$ as
$n\rightarrow\infty$ pointwisely. Note that the equality
\eqref{stochasitc convolution} proved above holds for each
deterministic time $k2^{-n}$. It follows that
\begin{align}\label{eq28}
    e^{(t-t\wedge\tau_n)\mathcal{A}}I(t\wedge\tau_n)&=\sum_{k=0}^{\infty}1_{\{k2^{-n}\leq\tau<(k+1)2^{-n}\}}e^{(t-t\wedge (k+1)2^{-n})\mathcal{A}}I(t\wedge (k+1)2^{-n})\nonumber\\
    &=\sum_{k=0}^{\infty}1_{\{k2^{-n}\leq\tau<(k+1)2^{-n}\}}I_{(k+1)2^{-n}}(t)=I_{\tau_n}(t).
\end{align}
Since $\tau_n$ converges down to $\tau$, so by the $\mathbb{P}$-a.s.
right-continuity  of $I(t)$, $I(t\wedge\tau_n)$ converges pointwise
on $\Omega$ to $I(t\wedge\tau)$ as $n\rightarrow\infty$ for every
$t\geq0$ $\mathbb{P}$-a.s. Also, observe that
\begin{align*}
    &\left\|e^{(t-t\wedge\tau_n)\mathcal{A}}I(t\wedge\tau_n)-e^{(t-t\wedge\tau)\mathcal{A}}I(t\wedge\tau)\right\|\\
    &\hspace{2cm}\leq
    \left\|e^{(t-t\wedge\tau_n)\mathcal{A}}\big{(}I(t\wedge\tau_n)-I(t\wedge\tau)\big{)}\right\|
    +\left\|\left(e^{(t-t\wedge\tau_n)\mathcal{A}}-e^{(t-t\wedge\tau)\mathcal{A}}\right)I(t\wedge\tau)\right\|\\
    &\hspace{2cm}\leq \|I(t\wedge\tau_n)-I(t\wedge\tau)\|+\left\|\left(e^{(t-t\wedge\tau_n)\mathcal{A}}-e^{(t-t\wedge\tau)\mathcal{A}}\right)I(t\wedge\tau)\right\|.
\end{align*}
converges to $0$ as $n\rightarrow\infty$. Thus we conclude that
$e^{(t-t\wedge\tau_n)\mathcal{A}}I(t\wedge\tau_n)$ converges to
$e^{(t-t\wedge\tau)\mathcal{A}}I(t\wedge\tau)$, for each $t\geq0$,
$\mathbb{P}$-a.s. For the term $I_{\tau_n}(t)$, by the isometry we
find out that
\begin{align*}
      \E\|I_{\tau_n}(t)-I_{\tau}(t)\|^2&=\E\left\|\int_0^{t}\int_Z\big{(}1_{[0,\tau_n]}(s)-1_{[0,\tau]}(s)\big{)}e^{(t-s)\mathcal{A}}\varphi(s,z)\tilde{N}(ds,dz)\right\|^2\\
                                       &=\E\int_0^{t}\int_Z\left\|\big{(}1_{[0,\tau_n]}(s)-1_{[0,\tau]}(s)\big{)}e^{(t-s)\mathcal{A}}\varphi(s,z)\right\|^2\nu(dz)\,ds.
\end{align*}
Recall that that $\tau_n\downarrow\tau$ as $n\rightarrow\infty$. So
$1_{[0,\tau_n]}$ converges to $1_{[0,\tau]}$ as
$n\rightarrow\infty$. Obviously, the integrand is bounded by
$\|\varphi(s,z)\|^2$ for all $n$. It then follows from dominated
convergence theorem that
\begin{align*}
   \lim_{n\rightarrow\infty}\E\|I_{\tau_n}(t)-I_{\tau}(t)\|^2\rightarrow0.
\end{align*}
Hence we can always find a subsequence which is convergent a.s.
Finally, Letting $n\rightarrow\infty$ in both sides of \eqref{eq28}
yields
\begin{align*}
         e^{(t-t\wedge\tau)\mathcal{A}}I(t\wedge\tau)=I_{\tau}(t)
\end{align*}
which completes our proof.

\end{proof}
\begin{remark}
Note in particular that if we replace $t$ by $t\wedge\tau$ in
\eqref{stochasitc convolution}, we obtain
\begin{align*}
       I(t\wedge\tau)=I_{\tau}(t\wedge\tau).
\end{align*}
\end{remark}

\begin{acknowledgements}
The authors would like to thank Professor Pao-Liu Chow and Professor  Jerzy Zabczyk for checking over an earlier draft and providing comments that greatly improved the manuscript. The research of the second named
author was partially supported by an ORS award at the University of
York.
\end{acknowledgements}

%-------------------------------------

\end{document}